\documentclass[11pt,reqno,tbtags]{amsart}

\usepackage{amsthm,amssymb,amsfonts,amsmath}
\usepackage{amscd} 
\usepackage{mathrsfs}
\usepackage[numbers, sort&compress]{natbib} 
\usepackage{subcaption, graphicx}
\usepackage{vmargin}
\usepackage{setspace}
\usepackage{paralist}
\usepackage{color}
\usepackage[normalem]{ulem}
\usepackage{stmaryrd} 
\usepackage[refpage,noprefix]{nomencl} 
\usepackage{tcolorbox}
\usepackage{bbm}
\usepackage{thm-restate}
\usepackage[pagebackref=true]{hyperref}
\usepackage[nameinlink,capitalise]{cleveref}
\usepackage[linesnumbered]{algorithm2e}

\usepackage{microtype}

\usepackage{tikz}
\usetikzlibrary{
  arrows.meta,
  positioning,
  calc
}

\definecolor{treegreen}{RGB}{40,170,90}

\providecommand{\R}{}
\providecommand{\Z}{}
\providecommand{\N}{}

\renewcommand{\R}{\mathbb{R}}
\renewcommand{\Z}{\mathbb{Z}}
\renewcommand{\N}{{\mathbb N}}

\newcommand{\one}{\mathbbm{1}}

\newcommand{\E}[1]{{\mathbf E}\left[#1\right]}	
\renewcommand{\deg}{\mathrm{deg}}
\newcommand{\e}{{\mathbf E}}

\renewcommand{\P}[1]{{\mathbf P}\left\{#1\right\}}
\newcommand{\p}{\mathbf P}

\newcommand{\I}[1]{{\mathbf 1}_{[#1]}}
\newcommand{\set}[1]{\left\{ #1 \right\}}
\newcommand{\Cprob}[2]{\mathbf{P}\set{\left. #1 \; \right| \; #2}} 
\newcommand{\probC}[2]{\mathbf{P}\set{#1 \; \left|  \; #2 \right. }}
\newcommand{\Ce}[2]{\mathbf{E}\left[\left. #1 \; \right| \; #2\right]}

\newcommand\cA{\mathcal A}
\newcommand\cB{\mathcal B}
\newcommand\cC{\mathcal C}

\newcommand\cF{\mathcal F}

\newcommand\cL{{\mathcal L}}
\newcommand\cM{\mathcal M}

\newcommand\cP{\mathcal P}
\newcommand\cQ{\mathcal Q}
\newcommand\cR{{\mathcal R}}
\newcommand\cS{{\mathcal S}}
\newcommand\cT{{\mathcal T}}

\newcommand\cX{{\mathcal X}}

\newcommand\cZ{{\mathcal Z}}

\newcommand{\bA}{\mathbf{A}}

\newcommand{\bM}{\mathbf{M}}

\newcommand{\be}{\mathbf{e}} 

\newcommand{\bl}{\mathbf{l}} 
 
\newcommand{\bn}{\mathfrak{n}}

\newcommand{\bt}{\mathbf{t}}

\newcommand{\bx}{\mathbf{x}} 
 
\newcommand{\bz}{\mathbf{z}}

\newcommand{\fC}{\mathfrak{C}}

\newcommand{\fG}{\mathfrak{G}} 
 
\newcommand{\fI}{\mathfrak{I}}

\newcommand{\fT}{\mathfrak{T}}

\newcommand{\ff}{\mathfrak{f}}

\newcommand{\fm}{\mathfrak{m}}

\newcommand{\fs}{\mathfrak{s}}

\newcommand{\eqdist}{\ensuremath{\stackrel{\mathrm{d}}{=}}}
\newcommand{\convd}{\ensuremath{\stackrel{\mathrm{d}}{\rightarrow}}}
\newcommand{\convp}{\ensuremath{\stackrel{\mathrm{p}}{\rightarrow}}}

\newcommand{\pran}[1]{\left(#1\right)}
\providecommand{\eps}{}
\renewcommand{\eps}{\varepsilon}
\providecommand{\ora}[1]{}
\renewcommand{\ora}[1]{\overrightarrow{#1}}
\newcommand\urladdrx[1]{{\urladdr{\def~{{\tiny$\sim$}}#1}}} 
\DeclareRobustCommand{\SkipTocEntry}[5]{} 

\newcommand*\tageq{\refstepcounter{equation}\tag{\theequation}} 

\newcommand{\height}{\mathrm{ht}}
\newcommand{\rtree}{\mathrm{T}}

\newcommand{\rforest}{\mathrm{F}}
\newcommand{\dstat}{\mathbf{n}}
\newcommand{\bdstat}{\mathbf{b}}
\newcommand{\dsTree}{\boldsymbol{\mathrm{T}}}
\newcommand{\dsForest}{\boldsymbol{\mathrm{F}}}
\newcommand{\rf}{\mathrm{f}}

\newcommand{\prob}{\mathbf{P}}
\newcommand{\Span}{\mathrm{span}} 
\newcommand{\branch}{\mathbf{br}}
\newcommand{\T}{\mathbb{T}}
\newcommand{\dgh}{\ensuremath{d_{\mathrm{GH}}}}
\newcommand{\dghp}{\ensuremath{d_{\mathrm{GHP}}}}
\newcommand{\dis}{\ensuremath{\mathrm{dis}}}
\newcommand{\type}{\bt}
\newcommand{\skel}{S^\lambda_n}
\newcommand{\skelr}{\cS^\lambda_n}
\newcommand{\cskel}{\cS^{\lambda,\delta}}
\newcommand{\fforest}{F_n^\lambda} 
\newcommand{\pack}{\mathrm{Pack}}

\reversemarginpar
\newcommand{\Exp}{\operatorname{Exp}}

\definecolor{cser}{RGB}{0,145,135}

\definecolor{csim}{RGB}{255,105,180}

\definecolor{cela}{RGB}{255, 140, 0}

\newtheorem{thm}{Theorem}
\newtheorem{lem}[thm]{Lemma}
\newtheorem{prop}[thm]{Proposition}
\newtheorem{cor}[thm]{Corollary}
\newtheorem{dfn}[thm]{Definition}

\newtheorem{fact}[thm]{Fact}

\newtheorem{rem}[thm]{Remark}
\numberwithin{thm}{section}
\numberwithin{prop}{section}
\numberwithin{cor}{section}
\numberwithin{lem}{section}
\numberwithin{dfn}{section}

\makenomenclature										
\graphicspath{ {./images/} }
\usepackage{wrapfig}

\begin{document}

\title[Packing and covering random trees]{Packing and covering random trees}


\author[Caelan Atamanchuk]{Caelan Atamanchuk}
\address{Department of Mathematics and Statistics, McGill University, Montr\'eal, Canada}
\email{caelan.atamanchuk@mcgill.ca}
\urladdrx{https://www.caelanatamanchuk.com/}
\date{\today} 

\author[Simon Briend]{Simon Briend}
\address{Department of Mathematics, Unidistance Suisse, Brig, Switzerland}
\email{simon.briend@unidistance.ch}
\urladdrx{https://simon-briend.github.io/}

\author[Serte Donderwinkel]{Serte Donderwinkel}
\address{Bernoulli Institute and CogniGron, University of Groningen, Groningen, The Netherlands}
\email{s.a.donderwinkel@rug.nl}
\urladdrx{https://www.sertedonderwinkel.com/}

\author[Elaine Herrera]{Elaine Herrera}
\address{Bernoulli Institute, University of Groningen, Groningen, The Netherlands}
\email{m.e.herrera@rug.nl}


\begin{abstract} 
 We study the sizes of optimal packings and coverings of conditioned branching processes by balls with growing radii. We find a universal limiting constant, and thereby establish universality of packing and covering numbers across different scales and different offspring distributions.

To be precise, for $\mu$ an offspring distribution with mean $1$ and finite variance $\sigma^2$ we let $T_n$ be a Bienaymé tree with offspring distribution $\mu$ conditioned to have size $n$. We prove that the packing number (resp. covering number) of $T_n$ with balls of radii $1\ll \gamma_n \ll n^{1/2}$ is asymptotically equivalent to $\mathfrak{C} n /(\sigma^2 \gamma_n^2)$ (resp.\ $\mathfrak{C} n /(4\sigma^2 \gamma_n^2)$), for a universal constant $\mathfrak{C}$. 

We also show that the order of limits can be exchanged and that the same limit occurs when we first consider the limiting covering number of $T_n$ by balls of fixed radius $R$, and thereafter let $R$ tend to $\infty$. 

We then extend these results to the Brownian continuum random tree conditioned to have mass $1$. For this model, we prove that the  packing (resp.\ covering) number with balls of radii $\eps$ are asymptotically equivalent to $\mathfrak{C}/\eps^2$ (resp. $\mathfrak{C}/(4\eps^2)$), for the same constant $\mathfrak{C}$, in contrast with a conjecture by Aldous. 

Finally, our results answer a question by Devroye, Eide and Pra\l at on the asymptotic burning number of mean $1$ finite variance Bienaymé trees of size $n$, by showing that it lies between $(c+o(1))n^{1/3}$ and $(2c+o(1))n^{1/3}$ with high probability, for a constant $c$ that only depends on the offspring variance.
\end{abstract}

\maketitle

\setcounter{tocdepth}{1}
\tableofcontents

\newpage

\section{Introduction}

In any metric space $(\cX,d)$ we define coverings and packings as follows.

\begin{dfn}[Covering]\label{def:covering}
    For a metric space $(\cX,d)$ and radius $r>0$, the subset $C$ of $\cX$ is said to be a $r$-covering (or $r$-net) if for any $x\in \cX, \ d(x,C) < r$.
\end{dfn}

\begin{dfn}[Packing]\label{def:packing}
    For a metric space $(\cX,d)$ and radius $r>0$, the subset $P$ of $\cX$ is said to be a $r$-packing (or $r$-separated set) if for any distinct $u,v\in P, \ d(u,v)\ge r$.
\end{dfn}

We call the individual elements of a packing and covering their \emph{centres}. If $\cX$ is compact, the minimum size of a covering and maximum size of a packing are well defined, and denoted by $\cC_r(\cX)$ and  $\cP_r(\cX)$. These are the \textit{covering number} and \textit{packing number}. These two quantities have been extensively studied for all sorts of metric spaces, and can be viewed as an `internal fractal dimension'. In this paper we focus on random trees, whether combinatorial trees, $\R$-trees or their scaling limits. \footnote{We make the distinction between combinatorial trees and $\R$-trees as the associated metric space is not the same. In a combinatorial tree, only nodes can be in a covering or packing, while in an $\R$-tree, points on edges can be included.}

\subsection{Our contribution}

Our main contribution is a law of large numbers for packing and covering numbers throughout the mesoscopic regime. We prove that, uniformly over arbitrary sequences $1\ll \gamma_n\ll n^{1/2}$, these quantities have deterministic first-order asymptotics of order $n/\gamma_n^2$. The leading constants depend on the offspring distribution only through its variance. We also show that the same constant is obtained by first taking the large-tree limit at fixed radius and then sending the radius to infinity. Finally, we transfer the result to the Brownian continuum random tree (CRT) and obtain exact small-radius asymptotics for its packing and covering numbers, where again the same universal constant appears. 

\begin{thm}\label{thm:bien}
For any sequence of radii $(\gamma_n)_{n\in\N}$ such that $\gamma_n \to \infty$ and $n^{1/2}/\gamma_n \to \infty$, and any offspring distribution $\mu$ with mean $1$ and finite variance $\sigma^2$, if $T_n$ is a $\mu$-Bienaymé tree conditioned to have size $n$, then, as $n\to\infty$

\begin{align*}
\frac{\sigma^2\gamma^2_n}{n}\mathcal{P}_{\gamma_n}(T_n)&\convp \fC \text{ and}\\
4\frac{\sigma^2\gamma^2_n}{n}\mathcal{C}_{\gamma_n}(T_n)&\convp \fC,
\end{align*}
where $\fC$ is a constant defined in \cref{subsec:def_const}.
\end{thm}

In the previous theorem we let the size of the tree and the radius grow simultaneously. We retrieve the same constant when we first let the tree size tend to infinity and then let the size of the balls grow.

\begin{thm}\label{thm:bien_reversed_limit}
For any $r>0$, any offspring distribution $\mu$ with mean $1$ and finite variance $\sigma^2$, if $T_n$ is a $\mu$-Bienaymé tree conditioned to have size $n$, then there is a constant $C_{\mu,r}$ such that as $n\to\infty$

\begin{align*}
\frac{1}{n}\mathcal{P}_{r}(T_n)&\convp C_{\mu,r}\text{ and}\\
\frac{1}{n}\mathcal{C}_{r}(T_n)&\convp C_{\mu,2r}.
\end{align*}
Moreover, as $r\to \infty$, \[r^2 \sigma^2 C_{\mu,r}\to \fC\] for  $\fC$ as in \cref{thm:bien}. 
\end{thm}

\begin{rem}
    We prove a more general version of the convergence in \cref{thm:bien_reversed_limit} for the integer radii case via a slightly more straightforward proof in a companion paper \cite{atamanchuk2026packingcoveringnumbertrees}. 
\end{rem}


The asymptotics given in the two previous theorems depend on $\mu$ only through its variance. It shows universality of our results across scales and offspring distributions. The universality of $\fC$ is further illustrated in our next theorem, where the same constant appears in the asymptotics of the packing and covering number of the Brownian CRT in the small ball limit.

\begin{thm}\label{thm:crt}
Let $\cT$ be the Brownian CRT conditioned to have mass $1$. Then, as $\delta\to 0$, 
\begin{align*}
\delta^2\mathcal{P}_{\delta}(\cT)&\convp  \fC
\text{ and}\\
4\delta^2\mathcal{C}_{\delta}(\cT)&\convp \fC,
\end{align*}
for $\fC$ as in \cref{thm:bien}.
\end{thm}

In \cite[Proposition 4]{aldousContinuumRandomTree1991}, Aldous conjectures that $\cC_{\delta}(\cT)\sim C\delta^{-2}\log(1/\delta)$. This was later disproved by \citet[Proposition 5.2]{duquesne2005probabilistic} who showed that $\cC_{\delta}(\cT)\lesssim C\delta^{-2}$. \cref{thm:crt} identifies the correct first-order behaviour, $\cC_{\delta}(\cT)\sim \fC\delta^{-2}$. 

It is remarkable that the same constant $\fC$ appears in the three theorems, highlighting the universality of our results. First because this constant does not depend on the scale $\gamma_n$ chosen, second because it does not depend on the offspring distribution $\mu$ and last because \cref{thm:bien_reversed_limit} and \cref{thm:crt} show that the same limit occurs when the radii and tree do not grow simultaneously. Thus, we may conclude that $\fC$ describes the universal internal fractal dimension of critical finite variance branching processes. 

\textbf{Application to graph burning} A direct Corollary of \cref{thm:bien} improves on the results of \citet{burnRTxf} on asymptotic burning number of size-conditioned Bienaymé trees. Graph burning is a discrete-time process on a graph, where vertices are either burning or unburned. In each round, a burning vertex causes all of its neighbours to burn and a new fire source is chosen: a previously unburned vertex is ignited. The process runs until all vertices are burning. The quantity of interest is the minimum number of rounds required to burn all vertices of the graph, denoted by $\cB(G)$. Using the vocabulary of our paper, the burning number is the smallest $k$ such that with $k$ balls of respective radii $1,\ 2,\dots,k$ it is possible to cover the graph. This problem has received increasing attention, in particular with regard to the Burning Number Conjecture, which states that it takes at most $\lceil \sqrt{n} \rceil$ steps to burn a connected graph on $n$ vertices (see for example \citet{land2016upper,bastide2021improved,norin2024burning}). More recently, it was proven by \citet{burnRTxf} that the burning number of size-conditioned Bienaymé tree with finite variance critical offspring distribution is of order $n^{1/3}$. Our result improves on \cite{burnRTxf} and answers one of their questions, proving that the burning number lies between $c n^{1/3}$ and $2c n^{1/3}$ with probability tending to $1$ for an explicit constant $c$. 

As recorded in \cite[Observation 2.1]{burnRTxf}, if a graph cannot be covered with $k$ balls of radius $k$, then its burning number is at least $k$. On the other hand, if it can be covered by $k$ balls of radius $k$, its burning number is at most $2k$. These facts and \cref{thm:bien} directly give the following corollary.

\begin{cor}\label{cor:burning_number}
    For any offspring distribution $\mu$ with mean $1$ and finite variance $\sigma^2$, if $T_n$ is a size-conditioned Bienaymé tree with offspring distribution $\mu$ then, as $n\to\infty$, with high probability,

    $$
    (1+o(1))\left(\frac{\fC}{4\sigma^2}\right)^{1/3}n^{1/3}
    \leq
    \cB(T_{n})
    \leq
    (1+o(1))\left(\frac{2\fC}{\sigma^2}\right)^{1/3}n^{1/3}.
    $$
\end{cor}

\subsection{Outline of the paper}

We now give an outline of the paper, as well as a sketch of proof for \cref{thm:bien}. First of all, we note that there is a direct, explicit relation between the packing and covering number of a tree. It was first established by \citet{meir1975relations} for combinatorial trees.

\begin{prop}[cf Moon Meir \cite{meir1975relations}]\label{prop:packink_is_covering}
    For any tree $T$ and positive integer $r$, it holds that $\mathcal{C}_{r}(T)=\mathcal{P}_{2r} (T)$. 
\end{prop}
We extend their result to $\R$-trees in \cref{app:covering=packing}, which allows us to exchange between studying the packing and covering number throughout the paper.

\begin{prop}\label{prop:R_tree_packink_is_covering}
    For any compact $\R$-tree $\cT$, it holds that, for any $r\geq 0$, $\mathcal{C}_{r}(\cT)=\mathcal{P}_{2r} (\cT)$. 
\end{prop}

Our second observation is that for $\eps$ arbitrarily small, the $\gamma_n$-packing number of $T_n$ is not much larger than the $(1-\eps)\gamma_n$-packing number of a much smaller subtree, which we introduce here. For $\lambda>0$, let $\skel$ be a subtree of $T_n$ spanned by the root and $\lfloor \lambda n/\gamma_n^2\rfloor $ uniformly random leaves. We call $\skel$ the \emph{skeleton} of $T_n$. 

\begin{prop}\label{prop:fringe}
Let $1\ll\gamma_n\leq 2\sqrt{n}$. For any $\eps>0$ it holds that \[\lim_{\lambda\to \infty}\limsup_{n\to \infty}\P{\frac{\gamma_n^2}{n}\left(\cP_{\gamma_n}(T_n)-\cP_{(1-\eps)\gamma_n}(\skel)\right)>\eps} =0.\]
\end{prop}

Analogously, for $\delta>0$ small, the $\delta$-packing number of the Brownian CRT $\cT$, is well-approximated by the $(1-\eps)\delta$-packing number of a much smaller subtree. Let $\cskel$ be the  sub-$\R$-tree of $\cT$ spanned by $\lfloor \lambda \delta^{-2}\rfloor $ uniformly random samples from the mass measure of $\cT$.

\begin{prop}\label{prop:fringe_CRT}
For any $\eps>0$ it holds that \[\lim_{\lambda\to \infty}\limsup_{\delta\to 0
}\P{\delta^2\left(\cP_{\delta}(\cT)-\cP_{(1-\eps)\delta}(\cskel)\right)>\eps} =0.\]
\end{prop}

\cref{prop:fringe} is proven in \cref{sec:covering_fringe} and \cref{prop:fringe_CRT} is proven in \cref{sec:covering_CRT_fringe}.

We denote by $\skelr$ the tree with edge lengths obtained from $\skel$ by replacing each maximal path $(v_0,v_1,\ldots,v_k)$ in $\skel$ all of whose internal vertices have degree two by a single edge $e$ joining $v_0$ and $v_k$, of length $\ell(e)=k$.  By a {\em tree with edge lengths} we mean a tree $T=(V,E)$ together with a function $\ell:E \to [0,\infty)$. Its coverings and packings are defined as in the corresponding $\R$-tree: we viewing edge $e$ as an isometric line segment of length $\ell(e)$ and allow for points on edges to be included in a packing or covering. Doing so does not change the packing and covering number on the leading order, indeed for any $r$

\begin{equation}\label{eq:pack_R_tree_vs_tree}\cP_r(\skel)\le \cP_r(\skelr) \le \cP_{r+1}(\skel).\end{equation}
We also view $\cskel$ as a tree with edge lengths in the natural way. We let $\cP_r(\skelr)$, $\cC_r(\skelr)$ and $\cP_r(\cskel)$, $\cC_r(\cskel)$ be the packing and covering numbers of $\skelr$ and $\cskel$ when viewed as trees with edge lengths, and we write $\#\skelr$, $\#\cskel$ for their respective number of vertices. 

Next, we aim to control the packing number of the skeletons $\skelr$ and $\cskel$. We start by proving that the skeletons have local limits, and later prove that the asymptotic packing number only depends on said local limit. It turns out that both $\skelr$ and $\cskel$ have the same local limit when appropriately rescaled, explaining the occurrance of the same constant in \cref{thm:bien} and \cref{thm:crt}. The precise definition of the marked local topology is discussed in \cref{sec:formalism}.

\begin{dfn}\label{dfn:kesten_tree}
  We use $T^*:=(T^*,P^*,\ell^*,o^*)$ to denote the binary Kesten tree with independent $\Exp(1)$ edge lengths, which we construct as follows. The marked vertex $o^*$ is the root, and $P^*$ is a one-ended infinite path starting at $o^*$. At each of the vertices of $P^*$ except $o^*$, we graft an independent planted Bienaymé tree with offspring distribution $\mu=\tfrac{1}{2}\delta_0+\tfrac{1}{2}\delta_2$. At $o^*$ we graft an independent unplanted Bienaym\'{e} tree with the same offspring distribution. Finally, the edge length function $l^*$ assigns each edge and independent $\Exp(1)$ random variable. We call the first edge on $P^*$, denoted by $\{o^*,v_1\}$, the root edge.
\end{dfn}

The local convergence of $\skelr$ is proven in \cref{sec:local_skeleton} while that of $\cskel$ is proven in \cref{sec:covering_CRT_fringe}.

\begin{restatable}{prop}{skelLocalLim}
\label{prop:skeleton_local_limit}
For any $\lambda>0$, as $n\to \infty$,
    \[\frac{\skelr}{\gamma_n/(\sigma \sqrt{2\lambda})}\xrightarrow{P} T^* \]
    in the marked local topology defined in \cref{subsec:def_marked_local_topology}, where $T^*$ is the binary Kesten tree with independent $\Exp(1)$ edge lengths. Moreover, $\frac{\gamma_n^2}{n}\#\skelr\convp 2\lambda$ as $n\to \infty$.
\end{restatable}

\begin{prop}\label{prop:crt_local_limit}
     For any $\lambda>0, $ as $\delta\to 0$,
    \[\frac{\cskel}{\delta/\sqrt{2\lambda}}\convp T^* \]
    in the local topology on trees with edge lengths, where $T^*$ is the binary Kesten tree with independent $\Exp(1)$ edge lengths. Moreover, $\#\cskel=2\lfloor \lambda\delta^{-2}\rfloor -1$ almost surely.
\end{prop}
Then, we establish that the asymptotic $\gamma_n$-packing number of $\skelr$ and the asymptotic $\delta$-packing number of $\cskel$ are functions of the local limit. We establish the following proposition in \cref{sec:covering_local_limit}, which is a specific case of a more general theorem from \cref{sec:packing_number_local} (see \cref{thm:packing_is_local}).

\begin{restatable}{prop}{packinglocal}\label{prop:packing_is_local}
    Let $(t_m)$ be any sequence of trees with edge lengths and $(\beta_m)$ any sequence of positive real numbers. For any $s>0$, there is $m_s\ge 0$ such that if $ \tfrac{1}{\beta_m}t_m\overset{\mathbb{P}}\to T^* $ in the local topology on trees with edge lengths, then 
\[\frac{1}{\#t_m}\mathcal{P}_{s\beta_m}(t_m)\convp m_s.\]
\end{restatable}

Then, we study the asymptotic behaviour of $m_s$ as $s \to \infty$.

\begin{restatable}{prop}{proplocalpacking}\label{prop:asymp_m_s}
    It holds that 
    \[s^2m_{s}\to \fC\]
    as $s\to \infty$, for $\fC$ as in \cref{thm:bien}.
\end{restatable}

The proofs of \cref{thm:bien} and \cref{thm:crt} follow from the above propositions.

\begin{proof}[Proof of \cref{thm:bien}]

Fix $\eps>0$ small enough such that 
\begin{equation}\label{eq:pf_1}\frac{\fC}{(1-\eps)^2\sigma^2}-\frac{\fC}{ \sigma^2}<\sqrt{\eps}.\end{equation}
Observe that $\cP_{\gamma_n}(\skel)\le \cP_{\gamma_n}(T_n) $ because $\skel$ is a subtree of $T_n$. Combined with \cref{prop:asymp_m_s} and \cref{prop:fringe}, we may choose $\lambda$ large enough such that \begin{equation}\label{eq:pf_2}\left|2\lambda m_{(1-\eps)\sigma \sqrt{2\lambda}}- \frac{\fC}{(1-\eps)^{2}\sigma^{2}}\right|<\sqrt{\eps}\quad\text{and}\quad \left|2\lambda m_{\sigma \sqrt{2\lambda}}- \frac{\fC}{\sigma^{2}}\right|<\sqrt{\eps}\end{equation} and such that for all $n$ large enough, 

\begin{equation}\label{eq:pf_3}
\P{\cP_{\gamma_n}(\skel)\le \cP_{\gamma_n}(T_n)\le   \cP_{(1-\eps)\gamma_n}(\skel)+\frac{\sqrt{\eps}n}{\gamma_n^2}}>1-\eps.\end{equation}
By \cref{prop:skeleton_local_limit}, we may apply \cref{prop:packing_is_local} to $\skelr$, with $\#\skelr=(1+o_p(1))2\lambda n/\gamma_n^2$, $\beta_n=\gamma_n/(\sigma\sqrt{2\lambda})$ and $  s=(1-\eps)\sigma \sqrt{2\lambda}$ or $s= \sigma \sqrt{2\lambda}$, to obtain that 
\[\frac{\gamma_n^2}{2\lambda n}\mathcal{P}_{(1-\eps)\gamma_n}(\skelr) \convp m_{(1-\eps)\sigma \sqrt{2\lambda}}\quad\text{and}\quad \frac{\gamma_n^2}{2\lambda n}\mathcal{P}_{\gamma_n}(\skelr) \convp m_{\sigma \sqrt{2\lambda}}
\]
as $n\to \infty$. By \eqref{eq:pack_R_tree_vs_tree} this also holds when we replace $\skelr$ by $\skel$.

Thus, for all $n$ large enough,

\begin{equation}\label{eq:pf_4}\P{\left|\frac{\gamma_n^2}{ n}\mathcal{P}_{(1-\eps)\gamma_n}(\skelr)-2m_{(1-\eps)\sigma \sqrt{2\lambda}}\right|\le \sqrt{\eps}\text{ and }\left|\frac{\gamma_n^2}{ n}\mathcal{P}_{\gamma_n}(\skelr)-2m_{\sigma \sqrt{2\lambda}}\right|\le \sqrt{\eps}}>1-\eps\end{equation}

Combining \eqref{eq:pf_1}, \eqref{eq:pf_2}, \eqref{eq:pf_3} and \eqref{eq:pf_4}, the union bound implies that for all $n$ large enough, 
\[\P{\left|\frac{\gamma_n^2}{n} \cP_{\gamma_n}(T_n)-\frac{\fC}{\sigma^2 }\right|\le 6\sqrt{\eps}}>1-2\eps,\]
which proves the theorem, because $\eps$ can be chosen arbitrarily small. 

\end{proof}

\begin{proof}[Proof for \cref{thm:crt}] 
The proof goes analogous to the proof of \cref{thm:bien} by applying \cref{prop:fringe_CRT} and \cref{prop:crt_local_limit} instead of \cref{prop:fringe} and \cref{prop:skeleton_local_limit} respectively.





\end{proof}

The proof of \cref{thm:bien_reversed_limit} follows from \cref{thm:bien} and a simple probabilistic argument presented in \cref{app:consistency_limit}.

\subsection{Packings in the Kesten tree and the limiting constant}\label{subsec:def_const}

In this section we give a heuristic explanation for the constant $\fC$. In \cref{prop:fringe} we prove that $\lambda$ can be chosen so large that the packing number of the size-conditioned Bienaymé tree is approximately the packing number of its skeleton $\skel$. When rescaled by $\gamma_n/(\sigma \sqrt{2\lambda})$, the skeleton converges locally to $T^*$. \cref{prop:packing_is_local} gives that the packing number is a local property by defining a canonical maximum packing. As balls of radii $\gamma_n$ rescaled by $\gamma_n/(\sigma \sqrt{2\lambda})$ have radii $\sigma\sqrt{2\lambda}$, the packing number is approximated with 
\begin{equation}\label{eq:sketch_universal_constant}
    \cP_{\gamma_n}(T_n)=\#\skelr \cdot m_{\sigma\sqrt{2\lambda}}(1+o(1)),
\end{equation}
where $m_s$ is the expected number of packing centres on the root edge in a canonical $s$-packing of $T^*$ that we define rigorously in \cref{sec:covering_local_limit}. To understand the definition of $\fC$, we need a heuristic understanding this canonical maximum $s$-packing, and how it behaves for large $s$.

Before looking close at $T^*$, let us look at finite trees. Informally, the canonical maximal $s$-packing of a finite tree is the one where all packing centres are placed as far away from the root as possible: one can imagine that all centres are pushed up in the tree towards the leaves, while respecting that they are all at least distance $s$ apart (in this paper we will view finite trees as growing up from the root). To understand why the packing number is local, it is key to observe that the canonical packing behaves well under merging trees at their roots: given two rooted trees and their canonical $s$-packings, if both trees are merged at their roots, then either the union of their packings is a canonical $s$-packing of the merged tree or we first need to remove the packing centre closest to the root. Furthermore, if we attach an edge below the root, we can construct a new canonical $s$-packing from the old one without moving the packing centres in the original tree. Thus, if we step-by-step construct the neighbourhood of an edge by merging trees on its ancestral line, then what happens at distance $s$ or more below the edge does not change its number of packing centres. These observations form the core of the proof of \cref{prop:packing_is_local}, and they are developed rigorously in \cref{sec:packing_number_local}. The canonical $s$-packing we just described can be algorithmically generated, and for the reader interested in getting a clearer picture, \cref{fig:placeholder} offers a preview of what they look like.

Returning to $T^*$, we can think of the canonical $s$-packing of $T^*$ as being derived by placing packing centres as far away as possible from some ``point at infinity'' which lies at the end of $P^*$ (in a sequence of finite trees converging locally to $T^*$, the root ``dissapears'' from view when we focus on a uniform vertex's neighbourhood). We can also take the same iterative perspective on its construction, except now instead of recursing along the ancestral line of a vertex, we do so along the path $P^*$. Because edge lengths are $\operatorname{Exp}(1)$ random variables but we consider packings of large radii $s$, it is unlikely that the root edge has more than $1$ packing centre on it. Moreover, if the tree grafted at $o^*$ is non-empty, then, since we push the packing centres upwards, the root edge will not have a packing centre unless the grafted tree is very large. This is unlikely for large $s$. Hence, the only event meaningfully contributing to $m_s$ is when $o^*$ is a leaf and a packing centre. Therefore, the expectation of $m_s$ is approximately the probability that $o^*$ is a leaf and a packing centre. On the event that $o^*$ is a leaf, which happens with probability $1/2$, $o^*$ is also a packing centre if none of the trees grafted at the vertices of $P^*$ have packing centres that would interfere with $o^*$ being a packing centre. Since the edge lengths are $\operatorname{Exp}(1)$, we can view the grafted subtrees as arriving according to a rate $1$ Poisson process. A tree grafted at distance $t$ from $o^*$ interferes with $o^*$ if and only if it has a packing centre at distance $(t,s-t)$ from its root  (distance larger than $t$ so that it is higher than $o^*$, and thus more preferable from the perspective of wanting to push packing centres as far away from infinity as possible; smaller than $s-t$ so that its distance to $o^*$ is too small for both centres to be in an $s$-packing together).
Therefore, trees with a packing centre that interferes with $o^*$ appear at the position-dependent rate $r(t)$, where $r(t)$ is the probability
$$r(t)=\P{\type_s(B)\in (t,s-t)}$$
for $B$ be a planted $\mu = \frac{1}{2}\delta_0 + \frac{1}{2}\delta_2$ offspring distribution Bienaym\'{e} tree with $\operatorname{Exp}(1)$ edge lengths, and $\type_s(B)$ the distance from its root to the nearest packing centre in its canonical maximal $s$-packing. Hence,  we get the approximation

$$m_s\approx \frac{1}{2}\exp\left( -\int_0^{s/2}  r(t)dt\right).$$

We compute the asymptotics of the above in \cref{sec:covering_local_limit}. We use that the canonical packing of a tree of height less than $s/2$ consists of a single packing centre at a height-attaining vertex. Hence $\type_s(B) = \height(B)$ when $\height(B) < s/2$. On the other hand, if we condition on $\height(B) \geq s/2$, $s^{-1}B$ converges to a random tree encoded by a Brownian excursion conditioned to have height at least $1/2$, denoted by $\cT^{(1/2)}$. Using that $\type_s(\cdot)$ is almost surely continuous at $\cT^{(1/2)}$ as a function on the set of compact rooted metric spaces (see \cref{prop:convergence_of_the_type}), we get that 

\begin{equation}
\begin{aligned}
    r(t)=&
    \P{\height(B)\geq s/2}
    \P{\tfrac{t}{s}< \frac{1}{s}\type_s(B)<1-\tfrac{t}{s} \ \mid \ \height(B)>s/2}\\
    &+
    \P{t<\height(B)<s/2}\\
    \approx&
    \frac{2}{2+s/2}\P{\tfrac{t}{s}<\type_1(\mathcal{T}^{(1/2)})<1-\tfrac{t}{s}}+
    \frac{2}{2+t}-\frac{2}{2+s/2}
\end{aligned}
\end{equation}
Integrating over $t\in(0,s/2)$ yields

$$s^2m_s\approx \frac{1}{2}\exp\left(2-4\int_0^{1/2} \prob\{v<\type_1(\mathcal{T}^{(1/2)})<1-v)\} \ dv\right)=:\fC.$$
This explains the origin of the constant $\fC$. \cref{thm:bien} is retrieved by replacing $s$ by $\sigma\sqrt{2\lambda}$ and $\#\skelr$ by $2\lambda n/\gamma_n^2$ in \eqref{eq:sketch_universal_constant}.

\begin{rem}
   Our use of local convergence and convergence under rescaling to the CRT, as well as the universality of our limiting constant, may raise the question whether convergence of the packing and covering numbers is simply a consequence of convergence of $T_n$ in the local Gromov--Hausdorff topology, when distances are rescaled by $\gamma_n$ and the tree is rooted at a uniformly random vertex. However, viewing the tree from the perspective of a uniformly random vertex is not enough to control a global quantity like the $\gamma_n$-covering number. For example, one may artificially append a star with $\eps_n n/\gamma_n$ arms of length $\gamma_n+1 $ somewhere in the tree, for $\eps_n\to0$ arbitrarily slow. This star contains a negligible proportion of the vertices and does not affect the limit under rescaling in the local Gromov--Hausdorff topology. (Its arms are also too short to affect the tree's global limit in the Gromov--Hausdorff--Prokhorov topology.) However, the star needs $\Theta(\eps_n n/\gamma_n)$ additional covering centres, which, for $\eps_n\gg 1/\gamma_n$, is much larger than the number of covering centres needed for the rest of the tree. This illustrates a subtlety of our problem: we need global control over the behaviour of local $\gamma_n$-neighbourhoods, and our combination of techniques and perspectives is needed to obtain this.
\end{rem}

\subsection*{AI statement}
We used a large language model to search the literature for relevant previous work and to assist with the proofs of the technical lemmas \cref{lem:steps_weak_convergence} and \cref{lem:height_continuous_Bienayme_tree}. All further mathematical reasoning and all writing are our own. 

\subsection*{Acknowledgements}

We would like to thank Louigi Addario-Berry for his continuous support and for fruitful discussions which greatly improved the quality of the paper.

CA and EH acknowledge the financial support of the NSERC Canada Graduate Scholarship -- Doctoral. SD acknowledges the financial support of the CogniGron research center and the Ubbo Emmius Funds (University of Groningen). Her research was also supported by the Marie Sklodowska-Curie grant GraPhTra (Universality in phase transitions in random graphs), grant agreement ID 101211705. This material is further based upon work supported by the National Science Foundation under Grant No.~DMS-1928930, while SD and SB were in residence at the Simons Laufer Mathematical Sciences Institute in Berkeley, California, during the spring semester of 2025. SD and SB also acknowledge the support of the Swiss National Foundation under grant No.~IZSEZ0\_238260 that supported a research visit in summer 2025.

\section{Tree formalisms and topology}\label{sec:formalism}

\subsection{Plane trees and \L ukasiewicz encodings}
In this paper, by \emph{tree}, we will mean plane tree. We follow Ulam--Harris' formalism to define such. 

Let $\mathcal{U}$ be the set of labels defined by
    \[\mathcal{U}=\bigcup_{k=0}^\infty\N^k,\]
where $\N=\set{1,2,3,\ldots}$ and $\N^0=\set{\emptyset}$. We endow $\mathcal{U}$ with the lexicographic order, and call $\mathcal{U}$ the \emph{Ulam--Harris tree}. We call the elements of $\mathcal{U}$ \emph{vertices} and each vertex is a (possibly empty) sequence $u=u_1\cdots u_j$ of positive integers. When $u=u_1\cdots u_j$ and $v=v_1\cdots v_k$ are vertices, we let $uv = u_1\cdots u_jv_1\cdots v_k$ be the \emph{concatenation} of $u$ and $v$. In particular, $u\emptyset = \emptyset u= u$. 
The length of vertex $u\in\mathcal{U}$, $|u|$, is called the \emph{generation} or \emph{height} of $u$. If $u=u_1\ldots u_k$, then the vertex $p(u)=u_1\ldots u_{k-1}$ is the \emph{parent} of $u$, and $u$ is a \emph{child} of $p(u)$. If $w=w_1\ldots w_k$ is such that $w_i=u_i$ for $1\leq i\leq k-1$ and $w_k\neq u_k$, then $u$ and $w$ are \emph{siblings}. Finally, $\mathrm{t}\subset \mathcal{U}$ is a \emph{plane tree} if $\mathrm{t}$ satisfies the following three conditions:
\begin{enumerate}[(i)]
\item $\emptyset \in\mathrm{t}$,
\item if $u\in\mathrm{t}$, then $p(u)\in\mathrm{t}$, 
\item there is a collection of non-negative integers $\set{c_{\mathrm{t}}(u) ~:~ u\in\mathrm{t}}$ so that for all $j\in\N$ and $u\in\mathrm{t}$, $uj\in\mathrm{t}$ if and only if $1\leq j\leq c_{\mathrm{t}}(u)$.
\end{enumerate}
The value $c_{\mathrm t}(u)$ is the number of children $u$ has in $\mathrm{t}$, or the \emph{out-degree} of $u$. Denote by $\mathcal{A}$ to be the set of all plane trees and $\mathcal{A}_n$ to be the set of all plane trees $\mathrm{t}$ so that $|\mathrm{t}|=n$.

Analogously, when we say \emph{forest} we mean \emph{plane forests}, which are an ordered collection of plane trees. For $k\geq 1$, denote by $\mathcal{A}^k =\set{\mathrm{f}=(\mathrm{t}_1, \ldots, \mathrm{t}_k) ~:~ \mathrm{t}_i\in\mathcal{A}}$ the collection forests consisting of $k$ trees. Lastly, $\mathcal{A}_n^k$ is the collection of forests $\mathrm{f}$ consisting of $k$ trees so that $|\mathrm{f}|=n$.

We recall that a sequence $(x_i)_{i=0}^n$ is a \emph{\L ukasiewicz path} of length $n$ if it satisfies:
\begin{enumerate}[(i)]
\item $x_0 =0,\ x_n=-1$,
\item $x_k> -1$ for all $0\leq k\leq n-1$,
\item and $x_k-x_{k-1}\geq -1$ for all $1\leq k\leq n$. 
\end{enumerate}
Let $\mathcal{W}_n$ be the set of \L ukasiewicz paths of length $n$, then there is a bijection between $\varphi:\mathcal{A}_n\rightarrow\mathcal{W}_n$. Let $\mathrm{t}\in\mathcal{A}_n$ and let $u_1,\ldots, u_n$ be the vertices of $\mathrm{t}$ in lexicographical order. Writing $(c_1,\ldots, c_n)=(c_{\mathrm t}(u_1),\ldots, c_{\mathrm t}(u_n))$, then  $\varphi(\mathrm{t})=(x_0, x_n,\ldots, x_n)$ is given by $x_0=0$ and for $k\geq 1$, \[x_k =\sum_{i=1}^k(c_i -1).\]
We refer to $W(\mathrm{t}) := \varphi(\mathrm{t})$ as the \L ukasiewicz path of the tree $\mathrm{t}$.

Extending this to forests, recall that $\mathrm{f}\in \mathcal{A}_n^k$ is a sequence of $k$ trees $\mathrm{f}=(\mathrm{t}_1, \ldots, \mathrm{t}_k)$ so $\sum_{i=1}^k\#\mathrm{t}_i=n$.  Letting $\mathcal{W}_n^k$ be the collection of sequences $(x_i)_{i=0}^n$ satisfying: 
\begin{enumerate}[(i)]
\item $x_0 =0,\ x_n=-k$,
\item $x_j> -k$ for all $0\leq j\leq n-1$,
\item and $x_k-x_{k-1}\geq -1$ for all $1\leq k\leq n$. 
\end{enumerate}
Then $\varphi$ is extended to a bijection between $\mathcal{A}_n^k$ and $\mathcal{W}_n^k$. Indeed, taking $u_1,\ldots u_n$ to be the vertices of $\mathrm{f}$ in lexicographical order of their trees and respecting the ordering of the trees, and writing $(c_1,\ldots, c_n)=(c_{\mathrm f}(u_1),\ldots, c_{\mathrm f}(u_n))$, then  $\varphi(\mathrm{f})=(x_0, x_n,\ldots, x_n)$ is given by $x_0=0$ and for $k\geq 1$, \[x_k =\sum_{i=1}^k(c_i -1).\]

\subsection{Degree statistics of trees and forests}\label{subsec:def_degree_stat}
Call a sequence of non-negative integers $\dstat=(\dstat_i)_{i\geq0}\subseteq\N_0^\N$ a \emph{degree statistic} if it satisfies $\sum_{i\geq 0}\dstat_i=1+\sum_{i\geq0}i\dstat_i$. A plane tree $\mathrm{t}$ has degree statistic $\dstat$ if for all $i\geq 0$, $\mathrm{t}$ contains exactly $\dstat_i$ vertices with $i$ children, and note that every degree statistic is realized by some tree. Furthermore, we call $\sum_{i\geq 0}\dstat_i$ the \emph{size} of $\dstat$ and denote this by $|\dstat|$. Observe that if $\mathrm{t}$ has degree statistic $\dstat$ then $\#\mathrm{t}=|\dstat|=\sum_{i\geq0}\dstat_i$.

We will be interested in plane forests and extend the above definition. For $k\in\N$, then $\dstat=(\dstat_i)_{i\geq0}\subseteq\N_0^\N$ is a \emph{degree statistic with $k$ components} or a \emph{degree statistic corresponding to $k$ components} if $\sum_{i\geq 0}\dstat_i=k+\sum_{i\geq0}i\dstat_i$. A plane forest $\mathrm{f}$ with $k$ components has degree statistic $\dstat$ if for all $i\geq 0$, $\mathrm{f}$ contains exactly $\dstat_i$ vertices with $i$ children, and similarly, every degree statistic is realized by some forest. As above, we call $\sum_{i\geq 0}\dstat_i$ the \emph{size} of $\dstat$ and indeed if $\mathrm{f}$ is a forest with degree statistic $\dstat$ then the number of vertices in $\mathrm{f}$ is the size of $\dstat$.

For $\mu$ a distribution on $\N_{0}$, a \emph{$\mu$-Bienaymé} tree is a random plane tree that is defined as follows. For each $u\in \mathcal{U}$, let $C_u$ be an independent sample from $\mu$ and let $T_\mu$ be the subset of $\mathcal{U}$ consisting of all $u_1\dots u_k$ with $k\ge 0$ for which $u_j\le C_{u_1\dots u_{j-1}}$ for all $1\le j \le k$. For any $n$ for which $\P{\#T_\mu=n}>0$, we let $T_n(=T_{\mu,n})$ have the law of $T_\mu$ conditioned on the event that $\#T_\mu=n$.

\subsection{$\R$-trees and the Gromov--Hausdorff topologies}

Let $(T,d)$ be a metric space. We call $(T,d)$ an $\R$-tree if for all $x,y \in X$ the following two conditions hold:

\begin{enumerate}
    \item There exists a unique isometric embedding $f:[0,d(x,y)] \to T$ such that $g(0) = x$ and $g(d(x,y)) = y$. We call $\llbracket x ,y  \rrbracket := g([0,d(x,y)])$ the $\emph{path}$ from $x$ to $y$.
    \item For any continuous and injective function $h:[0,1] \to T$ with $h(0) = x$ and $h(1) = y$, we have that $h([0,1]) = \llbracket x , y \rrbracket$.
\end{enumerate}

Let $(T_1,\rho_1,d_1)$, $(T_2,\rho_2,d_2)$ be two rooted compact $\R$-trees. A measurable relation $R \subseteq T_1 \times T_2$ is called a \emph{correspondence} betweeen $T_1$ and $T_2$ if the following three points holds:

\begin{enumerate}
    \item For all $x_1 \in T_1$ there exists $x_2 \in T_2$ such that $(x_1,x_2) \in R$.
    \item For all $y_2 \in T_2$ there exists $y_1 \in T_1$ such that $(y_1,y_2) \in R$.
    \item $(\rho_1,\rho_2) \in R$.
\end{enumerate}
We define the \emph{distortion} of $R$ by
$$
\dis(R) = \sup\big\{ |d_1(x_1,y_1) - d_2(x_2,y_2)| : (x_1,x_2), \ (y_1,y_2) \in R \big\}.
$$
Let $\cR(T_1,T_2)$ be the set of all correspondences between $T_1$ and $T_2$. We define the \emph{Gromov--Hausdorff} distance between $T_1$ and $T_2$ to be 
$$
\dgh(T_1,T_2) = \frac{1}{2}\inf_{R \in \cR(T_1,T_2)} \dis(R).
$$
The Gromov--Hausdorff distance defines a pseudometric on the space of rooted compact $\R$-trees, and so it has a unique quotient metric space associated with it, which we denote by $(\T,\dgh)$. $\T$ is exactly the set of equivalence classes of isometric rooted compact $\R$-trees. When the exact identity of the root or metric is clear from the context we shall abuse notation and just write $t \in \T$. There are a few related topologies that we will consider:

\textbf{Marked-Gromov--Hausdorff topology:} If we wish to add some finite number of extra marked points other than the root, there is a simple strengthening of the Gromov--Hausdorff topology that we consider. Let $(T_1,\rho_1,d_1,\cX_1)$ and $(T_2,\rho_2,d_2,\cX_2)$ be compact rooted $\R$-trees, each with collections of $m$ marked points $\cX_1 = (x_1^{(1)} , ... , x_m^{(1)})$ and $\cX_2 = (x_1^{(2)} , ... , x_m^{(2)})$. A correspondence $R$ between $T_1$ and $T_2$ is called an \emph{marked correspondence} if $(x_i^{(1)} , x_i^{(2)}) \in R$ for all $1 \leq i \leq m$. If $\cR^n(T_1 , T_2)$ denotes the set of all marked correspondences, then we define the \emph{$m$-marked-Gromov--Hausdorff} distance between the two marked trees as
$$
\dgh^m(T_1,T_2) = \frac{1}{2}\inf_{R \in \cR^m(T_1 , T_2)} \dis(R).
$$
As above, this defines a pseudometric, and we take $(\T^{(m)},\dgh^m)$ to denote its unique quotient metric space.

\textbf{Local Gromov--Hausdorff topology:} Let $\T_{\operatorname{loc}}$ be the set of all equivalence classes of isometric locally compact rooted $\R$-trees. For a tree $(T,\rho,d)$ and a fixed $s > 0$, we define $T^{(s)} = \{x \in T : d(x,\rho) \leq s\}$, which makes $(T^{(s)},\rho,d)$ an element of $\T$. We say that a sequence $((T_n,\rho_n,d_n))_{n \geq 1}$ convergences to $(T,\rho,d)$ in the \emph{local Gromov--Hausdorff topology} if for all $s > 0$, we have that $((T_n^{(s)} , \rho_n , d_n))_{n \geq 1}$ converges to $(T^{(s)},\rho,d)$ in the Gromov--Hausdorff topology on $\T$.

\textbf{Gromov--Hausdorff--Prokhorov topology:} There exists a natural strengthening of the Gromov--Hausdorff topology when our trees come equipped with a finite Borel measure. Let $(T_1,\rho_1,d_1,\mu_1)$ and $(T_2,\rho_2,d_2,\mu_2)$ be two compact $\R$-trees with $\mu_1$ and $\mu_2$ Borel probability measures. Let $\nu$ be a finite Borel measure on the product $T_1 \times T_2$. We define the \emph{discrepency} of $\nu$ as
$$
D(\nu,\mu_1,\mu_2) = \|\mu_1 - \nu \circ\pi_{1}^{-1}\| + \|\mu_2 - \nu \circ \pi_{2}^{-1}\|,
$$
where $\pi_{i}$ is the coordinate projection mapping and $\|\cdot\|$ denotes the total variation of a signed measure. Let $M(X)$ denote the set of all finite Borel measures on a metric space $(X,d)$. We then define the \emph{Gromov--Hausdorff--Prokhorov} distance (GHP) between the two measured trees as follows:
$$
\dghp = \frac{1}{2}\inf_{R \in \cR(T_1 , T_2), \ \nu \in M(T_1 \times T_2)} \Big(\dis(T) \vee D(\nu,\mu_1,\mu_2) \vee \nu(R^c)\Big).
$$
Again, the Gromov--Hausdorff--Prokhorov distance defines a pseudometric on the space of rooted compact measured $\R$-trees, and we take $(\T_{\mathrm{meas}},\dghp)$ to denote the corresponding quotient metric space.

\textbf{Extra notation:} Let $(t,\rho) \in \T_{\operatorname{loc}}$ and let $x \in t$. We let $t^{\uparrow x}$ denote the subtree consisting of all $y \in t$ such that $x \in \llbracket \rho ,y \rrbracket$. (This has the same meaning as for discrete trees.) We view $t^{\uparrow x}$ as an element of $\T_{\operatorname{loc}}$ by rooting it at $x$. We define $\height(x) = d(\rho,x)$, and define $\height(t) = \sup_{x \in t} \height(x)$. We let $\deg(x)$ denote the number of connected components in $t \setminus \{x\}$, which we call the \emph{degree}. We call the points with degree $1$  \emph{leaves}, and let $\cL(t)$ denote the set of all leaves in $t$. The points with degree at least three will be called \emph{branch points}, and we denote the set of all branch points by $\branch(t)$. If $t \in \T$, then $\branch(\cT)$ is at most countable, as the sets $\{b \in \branch(\cT) : \height(\cT^{\uparrow b}) > k^{-1}\}$ are finite for all $k \geq 1$ by compactness. If $S \subseteq t$, then we define $\Span(S) = \cup_{s,s' \in S} \llbracket s , s' \rrbracket$. When $S$ is finite, the space $t \setminus \Span(S)$ is a countable collection of connected pre-compact components $\{t_k : k \geq 1\}$. There is a unique vertex $\rho_k \in \Span(S)$ for each $k \geq 1$ such that $t_k \cup \{\rho_k\}$ is closed. We call $\{t_k \cup\{\rho_k\} : k \geq 1\}$ the subtrees \emph{pendent} to $\Span(S)$, and we view them as being rooted at the $\rho_k$.

\subsection{Discrete trees with edge lengths and the marked local topology}\label{subsec:def_marked_local_topology}

Throughout this paper we often work with rooted combinatorial trees $(t,\rho)$ that come equipped with an edge length function $\ell$. We want to stress that generation and height for a vertex differ in this case. For a vertex $v \in t$,  the \emph{generation} of $v$ is the number of vertices on the path between $\rho$ minus $1$, and the \emph{height} of $v$ is the sum of the edge lengths on the path between $\rho$ and $v$. In trees with edge lengths, distances between pairs of vertices where neither are the root are defined analogously. Moreover, for any edge $e=\{u,v\}$ in $t$, we will use $e$ to denote both the edge in the underlying combinatorial tree and the line segment connecting $u$ to $v$ in the tree when seen with its edge lengths.

Let $\fT$ be the set of all $(t,p,\ell,o)$ with $t$ a locally finite (plane) tree rooted at o, $\ell:E(t) \to [0,\infty)$ a function that assigns lengths to each edge, and $p$ a marked path than begins at $o$. In this paper we make use of a slight modification of the notion of local convergence for graphs with edge lengths introduced in \cite{aldous2004objective}. We call $r > 0$ a \emph{continuity point} of $t$ if there does not exist a vertex at distance exactly $r$ from $o$ in $p$. For such an $r$, let $p^{(r)}$ be all the vertices in $p$ that are distance at most $r$ from $o$. We define $F_r(o,t)$ the \emph{$r$-fringe of $t$} as follows. If all vertices of $p$ are at distance at most $r$ from $o$ then $F_r(o,t)=t$. Otherwise, pick any vertex $v$ in $p$ at distance greater than $r$ from $o$ and let $F_r(o,t)$ be the subtree spanned by all vertices in $t$ whose path towards $v$ enters the closed ball $B_r(o,t)$. Note that this construction does not depend on the choice of $v$ and $F_r(o,t)$ is the analogue of the ball $B_r(o,t)$ in the fringe topology. The fringe $F_r(o,t)$ maintains the marks on the vertices of $p$ that lie in $F_r(o,t)$.

For a sequence $(t_n,p_n,\ell_n,o_n)_{n \geq 1}$ in $\fT$, we say that $(t_n,p_n,\ell_n,o_n) \to (t,p,\ell,o)$ \emph{weakly locally} if for all continuity points $r > 0$ there exists an $N \geq 0$ such that for all $n \geq N$ there exists a graph isomorphism $\phi_n:F_r(o_n,t_n) \to F_r(o,t)$ which maps $p_n^{(r)}$ to $p^{(r)}$ and which satisfies $\ell_n(\phi_n^{-1}(e)) \to \ell(e)$ for all edges $e \in E(F_r(o,t))$. 

This notion of convergence determines a topology for which turns $\fT$ into a Polish space which we call the \emph{marked local topology}. We say that $T_n \to T$ in distribution with respect to the marked local topology if $\e[h(T_n)] \to \e[h(T)]$ for all bounded continuous functions $h:\fT \to [0,\infty)$. We extend the notion to pairs of trees in $\fT$ in the straightforward way.

Let $(T_n,\rho_n,\ell_n)_{n \geq 1}$  be a sequence of random rooted trees with edge lengths, and let $(T,P,\ell,o) \in \fT$ be a random tree. For each $n \geq 1$ be $O_n,V_n$ be two independent uniformly selected vertices in $T_n$, and let $P_n(O_n)$ and $P_n(V_n)$ be the path from $O_n$ to $\rho_n$ and $V_n$ to $\rho_n$ respectively. We say that  $(T_n,\rho_n,\ell_n) \to (t,p,\ell,o)$ \emph{in probability with respect to the marked local topology} if
$$
\Big( \big( T_n , P_n(O_n) , \ell_n , O_n \big) , \big( T_n , P_n(V_n) , \ell_n , V_n \big) \Big) \to \Big( \big( T_1 , P_1 , \ell_1 , o_1 \big) , \big( T_2 , P_2 , \ell_2 , o_2 \big) \Big)
$$
in distribution with respect to the marked local topology, where $( T_1 , P_1 , \ell_1 , o_1)$ and $( T_2 , P_2 , \ell_2 , o_2 )$ are independent copies of $(T,P,\ell,o)$.

Local convergence in probability for graphs is a familiar notion within the literature. We direct the reader to Sections 2 of \cite{aldous2004objective} or \cite{van2024random} or for some deeper discussion which motivates our definitions. 

\section{Packing numbers are local}\label{sec:packing_number_local}

In this section we develop some tools to discuss packings of discrete trees with edge lengths. (Here, we will often call trees with edge lengths simply `trees'.) We will allow ourselves to place packing centres on vertices or along edges. The key idea in this section is a recursive algorithm to generate a canonical maximum packing of a tree, and we use it to derive a law of large numbers for packing numbers with fixed radii of sequences of trees with a sufficiently nice local limit. The recursive definition and the notion of the type of a packing that we define below are inspired by Meir and Moon, who used similar notions for trees without edge lengths to obtain their asymptotic results in \cite{meir1973expected,meir1975expected} (for $r=1$) and \cite{meir1977packing,meir1977packingII} (for general $r$). Our use of local convergence is new.



\subsection{The type of a packing}

Let $(t,\rho,\ell)$ be a tree, and let $P$ be an $r$-packing of $t$ for some $r > 0$. We define the \emph{type of $P$} to be the smallest distance between a packing centre and the root of $t$, i.e., $\type(P) := \inf_{v \in P} d(\rho,v)$. Let $\cS_r(t)$ be the collection of all maximum $r$-packings of $t$. The \emph{$r$-type of $t$} is the largest possible type among packings in $\cS(t)$, i.e., $\type_r(t) = \sup_{P \in \cS_r(t)} \type(P)$. When there is no risk of confusion we will occasionally call the $r$-type `the type'. We call $P \in \cS_r(t)$ an \emph{$r$-type-attaining packing} (or $r$-TAP) if $\type(P) = \type_r(t)$. For convention, we take the $r$-type of an empty packing to be $r$. 

One reason that we choose to study $r$-TAPs is that they satisfy a helpful recursion at the root. Let $(t_1,\rho_1,\ell_1),...,(t_k,\rho_k,\ell_k)$ be trees, and let $P_1,...,P_k$ be $r$-TAPs of the respective trees. We define a new rooted tree $\bM(t_1,...,t_k)$ by gluing all the tree $t_1,...,t_k$ at their roots.
We call $\bM(t_1,...,t_k)$ the \emph{merger} of the trees $t_1,...,t_k$. Suppose without loss of generality that the trees are ordered from greatest type to least, $\type_r(t_1) \geq ... \geq \type_r(t_k)$.  Let $o_1 \in P_1,...,o_k \in P_k$ be such that $\height(o_j) = \type_r(P_j)$ for all $j$. Note that $d(o_i,o_j) = \type_r(t_i) + \type_r(t_j)$, so if this sum is smaller than $r$ for some $i,j$, then $P_1\cup \dots \cup P_k$ is not an $r$-packing. We now show how to turn $P_1\cup \dots \cup P_k$ into an $r$-TAP of $\bM(t_1,...,t_k)$ by removing some packing centres. Set 
$$
i^* = i^*_r(P_1,...,P_k) = \inf\left\{2 \leq i \leq k : \type_r(t_{i-1}) +\type_r(t_i) < r\right\}\wedge (k+1).
$$ 
We define the \emph{merger of the packings $P_1,...,P_k$}, written $\bM_r(P_1,...,P_k)$, as
$$
\bM_r(P_1,...,P_k) = (P_1 \cup ... \cup P_k) \setminus \{o_{i^*} , o_{i^* + 1} , ... , o_k\}.
$$
\cref{fig:merger_of_packings} illustrates these definitions. Since the ordering was not important to the definitions, we may analogously define $\bM_r(S)$ for unordered sets of packings. In the next lemma we observe that the merger of $r$-TAPs is an $r$-TAP of the merger.

\begin{lem}\label{lem:recursive_tap}
    Let $t_1,...,t_k$ be trees and let $P_1,...,P_k$ be defined as above. Then, $\bM_r(P_1,...,P_k)$ is an $r$-TAP of $\bM(t_1,...,t_k)$.
\end{lem}

\begin{proof}

Assume that $t_1,\dots, t_k$ are ordered by non-increasing type and let $o_1,\dots, o_k$ be as above. We first prove the statement when $k=2$ and then prove the more general statement inductively. 

To show that $\bM_r(P_1,P_2)$ is an $r$-packing, we need to show that for $x,y\in \bM_r(P_1,P_2)$, it holds that $d(x,y)\ge r$. If $x,y\in P_i$, then $d(x,y)\ge r$ because $P_i$ is an $r$-packing. Otherwise, if $x\in P_1$ and $y\in P_2$, then $d(x,y)=d(x,\rho)+d(y,\rho)\ge \type_r(t_1)+d(y,\rho)$. If $i^*=3$, then we use that 
\[\type_r(t_1)+d(y,\rho)\ge \type_r(t_1)+\type_r(t_2)\ge r.\]
Otherwise, $y\ne o_2$ so by the triangle inequality, \[d(y,\rho)\ge d(y,o_2)-d(\rho,o_2)\ge r-\type_r(t_2)\ge r-\type_r(t_1)\] 
so that $d(x,y)\ge r$. 

Now, we will show that $\bM_r(P_1,P_2)$ is an $r$-TAP. Let $B$ be a maximum $r$-packing of $\bM(t_1,t_2)$ and set $B_i = B \cap t_i$ for $i=1,2$. Since $B_i$ is an $r$-packing of $t_i$, we must have that

\begin{equation}\label{eq:packing_B_smaller_Pi}
    \forall \ i\in\{1,2\} \quad
    \#B_i\le \cP_r(t_i)=\#P_i.
\end{equation}

If $i^* = 3$ then $\bM_r(P_1,P_2)=P_1\cup P_2$ so $\# \bM_r(P_1,P_2)=\cP_r(t_i)+\cP_r(t_i)$. By maximality of $B$, we must have $\#B_i=\cP_r(t_i)$ and thus $\#\bM_r(P_1,P_2)=\#B$. Moreover,
\[\type(B)= \type(B_1)\wedge \type(B_2)\le \type_r(t_1)\wedge \type_r(t_2) = \type(\bM_r(P_1,P_2)),\]
proving that $\bM_r(P_1,P_2)$ is an $r$-TAP of $\bM(t_1,t_2)$.

We now consider $i^*=2$ so that $\#\bM_r(P_1,P_2)=\cP_r(t_1)+\cP_r(t_2)-1$. We claim that in this case $\type(\bM_r(P_1,P_2))=\type_r(t_1)$. First, $\type(\bM_r(P_1,P_2))\le d(\rho,o_1)=\type_r(t_1)$. Moreover, if $x\in \bM_r(P_1,P_2)\cap t_1$ then $d(x,\rho)\ge \type_r(t_1)$. If $x\in \bM_r(P_1,P_2)\cap t_2$ then $x\ne o_2$ so by the triangle inequality, \[d(x,\rho)\ge d(x,o_2)-d(\rho,o_2)\ge r-\type_r(t_2)>\type_r(t_1),\] because $\type_r(t_1)+\type_r(t_2)<r$. This implies that $\type(\bM_r(P_1,P_2))=\type_r(t_1)$. 

Next, we check that $\#B=\cP_r(t_1)+\cP_r(t_2)-1$ in the case $i^* = 2$. Because $\bM_r(P_1,P_2)$ is a packing, $\#B\geq \cP_r(t_1)+\cP_r(t_2)-1$. Suppose that the inequality is strict, then \eqref{eq:packing_B_smaller_Pi} gives for all $i\in\{1,2\}$, $\#B_i=\cP_r(t_i)$. Because $\type(B_i)\le \type_r(t_i)$ we have $\#B=\cP_i(t_1)+\cP_i(t_2)$, and so there exist $x\in B_1, y\in B_2$ such that \[d(x,y)=\type(B_1)+\type(B_2)\le \type_r(t_1)+\type_r(t_2)<r,\] contradicting that $B$ is an $r$-packing and proving that $\#B=\cP_r(t_1)+\cP_r(t_2)-1$.  Moreover, if $\#B_i=\cP_i(t_i)$ then $\type(B)\le \type_r(t_i)\le \type_r(t_1)$, so $\type(B)\le \type( \bM_r(P_1,P_2))$, concluding the proof of the lemma for $k=2$. 

Now consider general $k\ge 2$. Note that 
$$
\bM(\bM(t_1,\dots, t_{k-1}),t_k)=\bM(t_1,\dots,t_{k-1}, t_k),
$$ 
so the statement follows from induction if we show that 
$$
\bM_r(\bM_r(P_1,\dots, P_{k-1}),P_k)=\bM_r(P_1,\dots,P_{k-1}, P_k).
$$
Both sets contain $P_1\cup \dots \cup P_k\setminus\{o_1,\dots, o_k\}$ and are contained in $P_1\cup \dots \cup P_k$, so it is enough to show that they agree on $o_1,\dots, o_k$. For $j< k$, $o_j\in \bM_r(\bM_r(P_1,\dots, P_{k-1}),P_k)$ if and only if $o_j\in \bM_r(P_1,\dots, P_{k-1})$, which happens if and only if $\type(P_{j-1})+\type(P_j)\le r$. This is also the criterion for $o_j\in \bM_r(P_1,\dots,P_{k-1}, P_k)$. Moreover, $o_k\in \bM_r(\bM_r(P_1,\dots, P_{k-1}),P_k)$ if and only if $\type(\bM_r(P_1,\dots, P_{k-1}))+\type(P_k)\ge  r$. But this happens if and only if $i^*\ge k$, which is equivalent to $\type(P_{k-1})+\type(P_k)\ge r$. This is the criterion for $o_k\in \bM_r(P_1,\dots,P_{k-1}, P_k)$.

\end{proof}

\begin{figure}
    \centering
    {\begin{tikzpicture}[scale=1.2]

\colorlet{colA}{red}
\colorlet{colB}{blue}
\colorlet{colC}{green!60!black}


\begin{scope}[xshift=0cm]
    \node (b1) at (0,0) {};
    \node (b2) at (-0.3,0.7) {};
    \node (b3) at (0.4,0.8) {};
    \node (b4) at (0,1.6) {};
    \draw[line width=0.3mm,colB] (0,0) -- (-0.3,0.7);
    \draw[line width=0.3mm,colB] (0,0) -- (0.4,0.8);
    \draw[line width=0.3mm,colB] (0.4,0.8) -- (0.4,1.6);
    \node[draw=none, fill=none, below=0pt] at (0,-0.2) {\small{$(t_1,P_1)$}};

    \draw node [fill,circle,scale=0.35] at (0.4,0.9) {};
    \draw node [fill,circle,scale=0.35] at (0.4,1.6) {};
    \draw node [fill,circle,scale=0.35] at (-0.3,0.7) {};
    \draw node [fill,circle,scale=0.35] at (0.5 * 0.4 , 0.5 * 0.8) {};
\end{scope}

\begin{scope}[xshift=1.5cm]
    \draw[line width=0.3mm,colC] (0,0) -- (0,0.8);
    \draw[line width=0.3mm,colC] (0,0.8) -- (-0.3,1.3);
    \draw[line width=0.3mm,colC] (0,0.8) -- (0.4,1.6);
    \node[draw=none, fill=none, below=0pt] at (0,-0.2) {\small{$(t_2,P_2)$}};

    \draw node [fill,circle,scale=0.35] at (0.4,1.6) {};
    \draw node [fill,circle,scale=0.35] at (0.4 * 2 / 9,0.8 + 0.8 * 2 / 9) {};
    \draw node [fill,circle,scale=0.35] at (-0.3,1.3) {};
    \draw node [fill,circle,scale=0.35] at (0,0.3) {};
\end{scope}

\begin{scope}[xshift=3cm]
    \node (a1) at (0,0) {};
    \draw[line width=0.3mm,colA] (0,0) -- (0,0.9);
    \draw node [fill,circle,scale=0.35] at (0,0.9) {};
    \draw node [fill,circle,scale=0.35] at (0,0.2) {};
    \node[draw=none, fill=none, below=0pt] at (0,-0.2) {\small{$(t_3,P_3)$}};
\end{scope}
\begin{scope}[xshift=4.5cm]
\draw[line width=0.3mm,->] (-0.8,1) -- (0.2,1);
\end{scope}

\begin{scope}[xshift=6cm]

    \node (r) at (0,0) {};

   \node (a1) at (0,0) {};
    \draw[line width=0.3mm,colA] (0,0) -- (0.81,0.59*0.9);
    \draw node [fill,circle,scale=0.35] at (0.81,0.59*0.9) {};

    \draw[line width=0.3mm,colB] (0,0) -- (-0.76*0.81,0.76*0.59);
    \draw[line width=0.3mm,colB] (0,0) -- (-0.89*0.31,0.89*0.95);
    \draw[line width=0.3mm,colB] (-0.89*0.31,0.89*0.95) -- (-0.89*0.31-0.8*0.59,0.89*0.95+0.8*0.81);
    
    \draw node [fill,circle,scale=0.35] at (-0.89*0.31-0.1*0.59,0.89*0.95+0.1*0.81) {};
    \draw node [fill,circle,scale=0.35] at (-0.89*0.31-0.8*0.59,0.89*0.95+0.8*0.81) {};
    \draw node [fill,circle,scale=0.35] at (-0.76*0.81,0.76*0.59) {};
    \draw node [fill,circle,scale=0.35] at (-0.5*0.89*0.31,0.5*0.89*0.95) {};

    \draw[line width=0.3mm,colC] (0,0) -- (0.8*0.31,0.8*0.95);
    \draw[line width=0.3mm,colC] (0.8*0.31,0.8*0.95) -- (0.8*0.31-0.3,0.8*0.95+0.5);
    \draw[line width=0.3mm,colC] (0.8*0.31,0.8*0.95) -- (0.8*0.31+0.4,0.8*0.9+0.8);
    

     \draw node [fill,circle,scale=0.35] at (0.8*0.31+0.4,0.8*0.9+0.8) {};\draw node [fill,circle,scale=0.35] at (0.8*0.31+0.4*2/9,0.8*0.9+0.8*2/9) {};
     \draw node [fill,circle,scale=0.35] at (0.8*0.31-0.3,0.8*0.95+0.5) {};
     \draw node [fill,circle,scale=0.35] at (0.8*0.31*3/8,0.8*0.95*3/8) {};

    \node[draw=none, fill=none, below=0pt] at (0,-0.2)
    {\small{$(\mathbf{M}(t_1,t_2,t_3),\,\mathbf{M}_r(P_1,P_2,P_3))$}};
\end{scope}
\end{tikzpicture}}
    \caption{ Illustration of the merger of three trees and 
    the merger of their $r$-TAPs. The centre closest to the root in $t_3$ is deleted in the merger as it would be too close to the green centre.}
    \label{fig:merger_of_packings}
\end{figure}
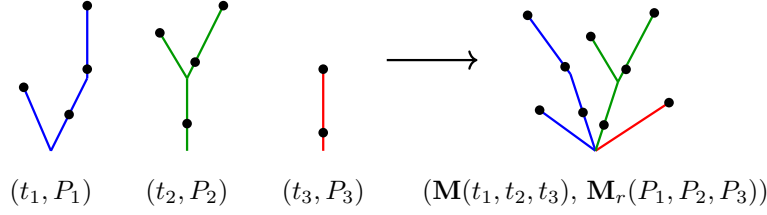

Accompanying the merger operation, we can also extend $r$-TAPs of a tree when we add edges below the root. Let $(t,\rho,\ell)$ be a tree, and let $\bA_L(t)$ be obtained from $t$ by adding a new root $\rho'$ below $\rho$ that is adjacent to $\rho$ with an edge of length $L$. Let $P$ be an $r$-TAP of $t$. Let $k^* = \sup\{k \geq 0 : kr \leq L + \bt_r(t)\}$. Then, we define $\bA_{r,L}(P)$ to be the packing obtained from $P$ by adding packing centres on the edge $\{\rho,\rho'\}$ at distances $r - \type_r(t), 2r-\type_r(t) , ... , k^* r - \type_r(t)$ from $\rho$.



\begin{lem}\label{lem:recursive_tap_2}
    Let $t$ and $P$ be as defined above. Then, $\bA_{r,L}(P)$ is an $r$-TAP of $\bA_L(t)$.
\end{lem}

\begin{proof}

    Observe that $\type(\bA_{r,L}(P))=L + \type_r(t) - k^* r$.  Let $B$ be a maximum packing of $\bA_L(t)$. First, $\#\{B\cap t\} \le \#P$ because $P$ is a maximum packing of $t$. Moreover, $\#\{B\setminus t\} \le k^*+1$ because by definition of $k^*$, we have that $(k^*+1)r > L +\type_r(t)\geq L$. Therefore, if $\#\{B\cap t\} < \#P$, then $\#\{B \setminus t\}=k^*+1 $ by maximality of $B$ and $\#B=\#\bA_{r,L}(P)$. Hence $\bA_{r,L}(P)$ is a maximal packing and
    $$
    \type(B)\le L-k^* r< \type(\bA_{r,L}(P)),
    $$
    where the first inequality follows from the fact that any $r$-packing of an edge of length $L$ with $k+1$ centres has type at most $L-rk$, and the second inequality follows from the first line of the proof. On the other-hand, if $\#\{B\cap t\} = \#P$ then the type of $B\cap t$ (viewed as a packing of $t$) is at most $\type_r(t)$. Thus, the centre on the edge $\{\rho,\rho'\}$ closest to $\rho$ is at distance at least $r-\type_r(t)$ from $\rho$, implying that $\#\{B\setminus t\}\le k^*$, so $\#B\leq \#\bA_{r,L}(P)$. Additionally, if $\#\{B\setminus t\}= k^*$ then $\type(B)\le \type(\bA_{r,L}(P))$. In both instances $\bA_{r,L}(P)$ is an $r$-TAP of $\bA_L(t)$.
\end{proof}

\subsection{The canonical $r$-TAP}

The recursions offered above give us a method to generate canonical $r$-TAPs in finite rooted trees with edge lengths that we will present in this subsection. First, we require some notation.

Let $(t = (V,E),\rho , \ell)$ be a tree and let $v \in V$. We let $C_v$ be the collection of children of $v$, let $c_v = \#C_v$, let $\ell(u)$ be shorthand for $\ell(\{u,v\})$ for all $u \in C_v$, for $v\neq \rho$, let $t^{\uparrow v}_*$ denote the union of $t^{\uparrow v}$ with the parent of $v$ and the edge connecting them, let $\fG_k \subseteq V$ be the set of vertices in generation $k$ of $t$, and let $\fm = \sup\{ k \geq 0 : \#\fG_k > 0 \}$. We recall that, since $t$ comes equipped with edge lengths, its maximum non-empty generation, $\fm$, is not the same as its height.

One way to think about the algorithm is that we initialize with packing centres at every leaf in the last generation of the tree (not necessarily the height attaining leaf), and iteratively extend and merge the partial packings as we make our way down the tree. This is the description given in Algorithm~\ref{alg:pack}. 

\begin{algorithm}[htbp]

    \SetAlgoLined
    
    $\pack_r(v) \leftarrow $ $\emptyset$ for each $v \in V$ \tcp*[f]{canonical $r$-TAP of $t^{\uparrow v}$}\;
    
    $\pack_r^*(v) \leftarrow $ $\emptyset$ for each $v \in V \setminus \{\rho\}$ \tcp*[f]{canonical $r$-TAP of $t^{\uparrow v}_*$}\;
    
     \For {$0 \leq i \leq \fm$} {
            \For {$v \in \fG_{\fm-i}$}{
                \If {$c_v = 0$} {
                $\pack_r(v) \leftarrow \{v\}$ \tcp*[f]{creates the trivial $r$-TAP of $t^{\uparrow v}$ for leaves}
                }
                \Else {
                    \For {$u \in C_v$} {
                        $\pack_r^*(u) = \bA_{r,\ell(u)}(\pack_r(u))$ \tcp*[f]{extends each packing corresponding to children of $v$}
                    }
                $\pack_r(v) = \bM_r\left\{ \pack_r^*(u)   : u \in C_v\right\}$ \tcp*[f]{merges the extended packings}
                }
            }
        }
    \Return $\pack_r(\rho)$ \tcp*[f]{output the canonical packing} \;
    
    \caption{\label{alg:pack} The definition of a canonical $r$-TAP of a rooted tree with edge lengths $(t = (V,E) , \rho , \ell)$.}
\end{algorithm}

We can also compress the algorithm into a recursively defined function:
$$
\pack_r(v) =
\begin{cases}
    \bM_r\Big(\big\{ \bA_{r,\ell_u}\big(\pack_r(u)\big) : u \in C_v\big\}\Big) , & \text{if $c_v \geq 1$} \\
    \{v\}, & \text{if $c_v = 0$.}
\end{cases}
$$

If the underlying tree is ever unclear, we will write $\pack_r(v;t)$ in place of $\pack_r(v)$. All of the packings at the top of the recursion, when $c_v = 0$, are maximum and type-attaining of $t^{\uparrow v}$. Then, all other steps in the algorithm involves extending and merging packings of trees, which by \cref{lem:recursive_tap} and \cref{lem:recursive_tap_2} preserves the property of being an $r$-TAP. We can conclude from this that $\pack_r(v)$ is an $r$-TAP of $(t^{\uparrow v},v,\ell)$, and in particular, that $\pack_r(\rho)$ is an $r$-TAP of $t$. We call $\pack_r(\rho)$ \emph{the canonical $r$-TAP of $t$}.

The following lemma offers the most important feature of the canonical $r$-TAP. It specifies that we only need to look at a bounded region around $v$ in $t$ to determine $\pack_r(\rho)\cap t^{\uparrow v}$.

\begin{figure}
    \centering

    {\begin{tikzpicture}[scale=1.2]

\colorlet{colA}{red}
\colorlet{colB}{blue}
\colorlet{colC}{green!60!black}

\begin{scope}[xshift = 0cm]

    \node (r) [fill,circle,scale=0.35] at (0,0) {};

    \node (r1) at (-1,0.8) {};
    \draw[line width=0.3mm] (0,0) -- (-0.9*0.866,0.9*0.5);


    \draw node [fill,circle,scale=0.35] at (-0.9*0.866,0.9*0.5) {};

    \draw[line width=0.3mm] (0,0) -- (-0.3,0.7);
    \draw[line width=0.3mm] (0,0) -- (0.4,0.8);
    \draw[line width=0.3mm] (0.4,0.8) -- (0.4,1.6);

    \draw node [fill,circle,scale=0.35] at (0.4,0.8) {};

    \draw node [fill,circle,scale=0.35,colA] at (0.4,1.6) {};
    \draw node [fill,circle,scale=0.35] at (-0.3,0.7) {};

    \draw[line width=0.3mm] (0,0) -- (0.866,0.5);
    \draw[line width=0.3mm] (0.866,0.5) -- (0.566,1.0);
    \draw[line width=0.3mm] (0.866,0.5) -- (1.266,1.3);

    \draw node [fill,circle,scale=0.35] at (0.866,0.5) {};

    \draw node [fill,circle,scale=0.35,colA] at (1.266,1.3) {};
    \draw node [fill,circle,scale=0.35,colA] at (0.566,1.0) {};

    \node at (0,-0.25) {$k=0$};

\end{scope}


\begin{scope}[xshift = 3cm]

    \node (r) [fill,circle,scale=0.35] at (0,0) {};

    \node (r1) at (-1,0.8) {};
    \draw[line width=0.3mm] (0,0) -- (-0.9*0.866,0.9*0.5);


    \draw node [fill,circle,scale=0.35,colA] at (-0.9*0.866,0.9*0.5) {};

    \draw[line width=0.3mm] (0,0) -- (-0.3,0.7);
    \draw[line width=0.3mm] (0,0) -- (0.4,0.8);
    \draw[line width=0.3mm] (0.4,0.8) -- (0.4,1.6);

    \draw node [fill,circle,scale=0.35] at (0.4,0.8) {};

    \draw node [fill,circle,scale=0.35,colA] at (0.4,0.9) {};
    \draw node [fill,circle,scale=0.35,colA] at (0.4,1.6) {};
    \draw node [fill,circle,scale=0.35,colA] at (-0.3,0.7) {};

    \draw[line width=0.3mm] (0,0) -- (0.866,0.5);
    \draw[line width=0.3mm] (0.866,0.5) -- (0.566,1.0);
    \draw[line width=0.3mm] (0.866,0.5) -- (1.266,1.3);

    \draw node [fill,circle,scale=0.35] at (0.866,0.5) {};

    \draw node [fill,circle,scale=0.35,colA] at (1.266,1.3) {};
    \draw node [fill,circle,scale=0.35,colA] at (0.566,1.0) {};
    \draw node [fill,circle,scale=0.35,colA] at (0.866 + 0.4 * 2 / 9,0.5 + 0.8 * 2 / 9) {};

    \node at (0,-0.25) {$k=1$};

\end{scope}

\begin{scope}[xshift = 6cm, yshift = 0cm]

    \node (r) [fill,circle,scale=0.35] at (0,0) {};

    \node (r1) at (-1,0.8) {};
    \draw[line width=0.3mm] (0,0) -- (-0.9*0.866,0.9*0.5);

    \draw node [fill,circle,scale=0.35,colA] at (-0.9*0.866,0.9*0.5) {};
    \draw node [scale=1,colA] at (-0.9*0.866*0.2,0.9*0.5*0.2) {$\mathbin{\rotatebox[origin=c]{-30}{$\times$}}$};

    \draw[line width=0.3mm] (0,0) -- (-0.3,0.7);
    \draw[line width=0.3mm] (0,0) -- (0.4,0.8);
    \draw[line width=0.3mm] (0.4,0.8) -- (0.4,1.6);

    \draw node [fill,circle,scale=0.35] at (0.4,0.8) {};

    \draw node [fill,circle,scale=0.35,colA] at (0.4,0.9) {};
    \draw node [fill,circle,scale=0.35,colA] at (0.4,1.6) {};
    \draw node [fill,circle,scale=0.35,colA] at (-0.3,0.7) {};
    \draw node [fill,circle,scale=0.35,colA] at (0.5 * 0.4 , 0.5 * 0.8) {};

    \draw[line width=0.3mm] (0,0) -- (0.866,0.5);
    \draw[line width=0.3mm] (0.866,0.5) -- (0.566,1.0);
    \draw[line width=0.3mm] (0.866,0.5) -- (1.266,1.3);

    \draw node [fill,circle,scale=0.35] at (0.866,0.5) {};


    \draw node [fill,circle,scale=0.35,colA] at (1.266,1.3) {};
    \draw node [fill,circle,scale=0.35,colA] at (0.566,1.0) {};
    \draw node [fill,circle,scale=0.35,colA] at (0.866 * 3 / 8,0.5 * 3 / 8) {};
    \draw node [fill,circle,scale=0.35,colA] at (0.866 + 0.4 * 2 / 9,0.5 + 0.8 * 2 / 9) {};

    \node at (0,-0.25) {$k=2$};

\end{scope}

\end{tikzpicture}}
    
    \caption{A figure depicting the the three main steps of the canonical $r$-TAP algorithm as described in Algorithm~\ref{alg:pack}. The vertices of $t$ are black, the packing centres at each step are red, and the single red cross depicts a packing centre that was added during an extension step, but then deleted in the final merger. At each step, the partial packings up to generation $3-k$ are determined by extending and merging the packings above non leaves, and placing packing centres at leaves.}
    \label{fig:placeholder}
\end{figure}
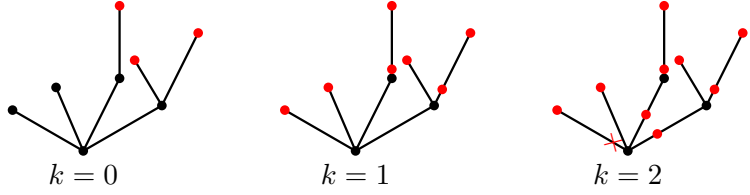

\begin{lem}\label{lem:packing_centres_fringe}
    Let $(t,\rho,\ell)$ be a tree and let $v$ be a vertex in $t$. Let $u$ be the lowest height ancestor of $v$ such that $d(u,v) \leq r \wedge \height(v)$, where we consider $v$ as one of its own ancestors. Then, $\pack_r(\rho)\cap t^{\uparrow v} = \pack_r(u)\cap t^{\uparrow v}$.
\end{lem}

\begin{proof}
Let $0 \leq k \leq \fm$, where we recall that $\fm$ is the number of generations in the tree (excluding the generation of the root). The proof is trivial if $u = \rho$, so suppose otherwise. Let $v \in \fG_k$, let $ Q = (v = v_0 , ... , \rho = v_k)$ be the path from $v$ to $\rho$ traversed in order, and let $1 \leq j < k$ such that $u = v_{j}$. For each $0 \leq i \leq k$, let $t_i = t^{\uparrow v_i} \setminus t^{\uparrow v_{i-1}}_*$ (we note that the edge $\{v_{i-1},v_i\}$ is not included in $t_i$). Observe from the definition of the canonical $r$-TAP that 
$$
\pack_r(v_i) = \bM_r\left(\bA_{r,\ell(v_{i-1})}(\pack_r(v_{i-1})),\pack_r(v_{i};t_i)\right),
$$ 
for all $0 \leq i \leq k$. This directly yields
\begin{equation}\label{eq:bounded_distance_property}
\pack_r(v_i) \cap t^{\uparrow v} = \bM_r\left(\pack_r(v_{i-1};t^{\uparrow v_{i-1}}_*) \cap t^{\uparrow v} ,\pack_r(v_{i};t_i)\right) \cap t^{\uparrow v},
\end{equation}
for all $0 \leq i \leq k$. By definition, every element of $t^{\uparrow v}$ is necessarily at distance strictly greater than $r$ from $v_{j + 1},...,v_k$, and so for all $j+1 \leq m \leq k$, 
$$
\bM_r\left(\pack_r(v_{m-1};t^{\uparrow v_{m-1}}_*) \cap t^{\uparrow v} ,\pack_r(v_{m};t_m)\right) \cap t^{\uparrow v} = \pack_r(v_{m-1};t^{\uparrow v_{m-1}}_*) \cap t^{\uparrow v}.
$$
This combines with \eqref{eq:bounded_distance_property} to give,
$$
\pack_r(v_m) \cap t^{\uparrow v} = \pack_r(v_{m-1};t^{\uparrow v_{m-1}}_*) \cap t^{\uparrow v},
$$
for all $j+1 \leq m \leq K$. Chaining this identity together and recalling that $v_{j} = u$ we get,
$$
\pack_r(\rho) \cap t^{\uparrow v} = \pack_r(v_{j} ; t_*^{\uparrow v_{j}}) \cap t^{\uparrow v} = \pack_r(v_{j}) \cap t^{\uparrow v} = \pack_r(u) \cap t^{\uparrow v},
$$
which is the desired result.
\end{proof}

\subsection{The packing number is local}\label{subsec:local_limit}

We close off this section with a law of large numbers for packing numbers of random trees that have a nice limit in the marked local topology defined in \cref{sec:formalism}. 

Before proceeding, we pause to recall the definition of the \emph{ends} of an infinite tree $t$. Let $p$ and $p'$ be two infinite paths in $t$. We call $p$ and $p'$ equivalent if they differ in at most finitely many places. The ends of $t$ are exactly the equivalence classes of the different infinite paths under this equivalence relation. We call $t$ \emph{$k$-ended} if there are $k$ such equivalence classes.

Let $(t,\ell,\rho)$ be a finite rooted tree with edge lengths, with a distinguished vertex $o$ and path $p=(o=v_0,v_1,\dots, v_k)$. We define $\ff(t,p,\ell,o)=\ff_r(t,p,\ell,o)$ to be number of packing centres on the edge $\{o,v_1\}$ in $\pack_r(v_k)$ if $k \ne 0$, and let $\ff(t,p,\ell,o) = 0$ if $k=0$. To avoid double counting, we view the edge as including $o$, but not including $v_1$.

We will apply the function $\ff$ to special choices of paths that go in the direction of the root. For $v \in (t,\ell,\rho)$, let $v_1$ be its parent in $t$, let $a_v$ be lowest height ancestor of $v$ such that $d(a_v,v_1) \leq r \wedge \height(v)$, and let $p_v$ be the path from $v$ to $a_v$. Note that, by \cref{lem:packing_centres_fringe}, we have that $\pack_r(\rho) \cap t^{\uparrow v_1}= \pack_r(a_v) \cap t^{\uparrow v_1}$, and so the number of packing centres on the edge $\{v,v_1\}$ in $\pack_r(\rho)$ is the same as the number in $\pack_r(a_v)$. Since the size of a packing is the sum over all edges of the number of packing centres on the edge, we get the following identity:
\begin{equation}\label{eq:packing_number_as_sum}
\cP_r(t) = \sum_{v \in t} \ff(t,p_v,\ell,v).
\end{equation}
With a local limit theorem in mind, we would like to define $\ff$ for infinite trees $(t,p,\ell,o) \in \fT$. First, if $t$ has more than one end or $p$ is a finite path, then we simply define $\ff(t,p,\ell,o) = 0$. Suppose now that $t$ is one-ended and that $p = (v_0, v_1 , ... )$ is the unique infinite path in $t$ that begins at $o$. Let $F(v_i) = F(v_i;t)$ denote the tree, rooted at $v_i$, which contains all vertices whose unique infinite path goes through $v_i$, i.e.,  it contains the fringe subtrees rooted at $v_0,...,v_i$. we will also use $F^*(v_k)$ to denote the tree $F(v_k)$ in combination with $v_{k+1}$ and the edge $\{v_k,v_{k+1}\}$. \cref{lem:packing_centres_fringe} ensures that the set $\pack_r(v_k ; F(v_k))\cap F(v_1)$ is eventually constant, and we can define $\ff(t,p,\ell,o)$ be the number of packing centres on the edge $\{o,v_1\}$ in the limiting set.

The continuity of the function $\ff$ at a one-ended tree (where $p$ is an infinite path) with respect to the marked local topology depends on the edge lengths. For example, if $\ell(\{o,v_1\}) = r$, then $\ff$ will not necessarily be continuous at $(t,p,\ell,o)$. Moreover, edges having zero or infinite length could also pose issues. However, taking the edge lengths to be independent and atomless turns out to be sufficient for continuity, and we prove this in the next two lemmas.

\begin{lem}\label{lem:type_continuity_for_finite_trees}

Let $(t,\ell,\rho)$ be a tree with independent atomless edge lengths. Then $\type_r(t)$ is atomless and $\type_r:\fT \to [0,r)$ is almost surely continuous at $(t,\rho,\ell)$ with respect to the marked local topology. Moreover, for any vertex $v \in t$, $\ff(t,\cdot,\rho,v):\R^{\#t-1} \to \R$ is continuous at $\ell$ almost surely.

\end{lem}

\begin{proof}
    We can prove that the type is atomless with a simple induction argument. First suppose that $t$ is a single edge with length $L$. The type is precisely $L \ (\text{mod } r)$, and so for any $x \in [0,r)$, by the atomlessness assumption,
    $$
    \prob\{\type_r(t) = x\} \leq \sum_{k=0}^\infty \prob\{L = kr + x\} = 0.
    $$
    From here, if we can show that the atomlessness is preserved under extensions and mergers we are done by \cref{lem:recursive_tap} and \cref{lem:recursive_tap_2}.

    If we extend a tree $t$ by an edge of independent atomless length $L$ below the root, then the new, extended tree has type $\type_r(t) + L \ (\text{mod } r)$. Since the edge lengths are independent, $\type_r(t) + L$ is atomless, and so $\type_r(t) + L \ (\text{mod } r)$ is as well by another union bound.

    Now consider $k$ trees $t_1,...,t_k$ with atomless types. Recall  the type of the merger of the $k$ trees has type $\type_r(t_I)$, where $I$ is a random index between $1$ and $k$. By applying a union bound we get,
    $$
    \prob\{\type_r(\bM_r(t_1,...,t_k)) = x\} \leq \sum_{k=1}^k \prob\{\type_r(t_i) = x\} = 0.
    $$

    Now we turn to the continuity of $\bt_r$ with respect to the marked local topology. Since $t$ is finite, if $t_n \to t$ with respect to the marked local topology, then there is some $N \geq 0$ such that $t_n$ is isomorphic as a graph to $t$ for all $n \geq N$. Because of this, it is enough to take $t_n = t$ for all $n \geq 0$ and just study the convergence of the edge lengths $(\ell_n)_{n=1}^\infty$. Moreover, we can assume that $t$ has at least one edge as the proof is trivial otherwise.

    Let $\fm$ be largest non-empty generation of $t$. Fix $0<\eps < 1$ and take 
    $$
    \delta = \frac{\eps}{4\fm}\left(\min_{u \ne v} |r-\type_r(t^{\uparrow u}_*) + \type_r(t^{\uparrow v}_*)| \wedge \min_w\big( \type_r(t_*^{\uparrow w}) \wedge (r - \type_r(t_*^{\uparrow w}))\big)\right).
    $$ 
    By the atomlessness of the type for all finite trees, we know that $\delta$ is non-zero almost surely. We will work on this event. Suppose that $\|\ell_n - \ell\|_\infty < \delta$. We will argue inductively that for $i=0,\dots, \fm$, $|\type_r(t^{\uparrow v}) - \type_r(t_n^{\uparrow v})| < i\delta < \eps$ for each vertex $v$ in generation $\fm-i$ of $t$.

    The base case of generation $\fm$ is trivial, so it suffices to just complete the inductive step. Assume the claim for all vertices in generation $\fm-i+1$, and let $v$ be in generation $\fm-i+1$. We first prove that $|\type_r(t_*^{\uparrow v}) - \type_r((t_n)_*^{\uparrow v})| < i\delta$. Indeed, by \cref{lem:recursive_tap_2} we have that
    \begin{equation}\label{eq:i_still_hate_continuity}
    \type_r(t_*^{\uparrow v}) = \ell(v) - k(v)r + \type_r(t^{\uparrow v}),
    \end{equation}
    where $k(v) = \sup\{k \geq 0 : kr \leq \ell(v) + \type_r(t^{\uparrow v})\}$. We also similarly have
    \begin{equation}\label{eq:i_hate_continuity}
    \type_r((t_n)_*^{\uparrow v}) =\ell_n(v) - k_n(v)r + \type_r((t_n)^{\uparrow v}),
    \end{equation}
    where $k_n(v)$ is defined analogously to $k(v)$. By using the inductive hypothesis and the fact that $\|\ell_n-\ell\|_\infty < \delta$ with the triangle inequality we get
    \begin{align}
    \Big|\big(\ell_n(v) + \type_r((t_n)^{\uparrow v})\big) - \big(\ell(v) + \type_r(t^{\uparrow v})\big)\Big| &\leq i\delta \notag \\
    &\leq \frac{\eps}{2}\Big(\type_r(t_*^{\uparrow v}) \wedge (r-\type_r(t_*^{\uparrow v}))\Big). \label{eq:error_bound_extended_type}
    \end{align}
    We will now show that $k(v) = k_n(v)$. Combining \eqref{eq:i_still_hate_continuity} and \eqref{eq:error_bound_extended_type} we get that
    \begin{align*}
        \ell_n(v) + \type_r((t_n)^{\uparrow v}) - k(v)r &\geq \ell(v) + \type_r(t^{\uparrow v}) - k(v)r - \frac{1}{2}\type_r(t^{\uparrow v}_*) \\
        &= \type_r(t^{\uparrow v}_*) - \frac{1}{2}\type_r(t^{\uparrow v}_*) > 0,
    \end{align*}
    and 
    \begin{align*}
        \ell_n(v) + \type_r((t_n)^{\uparrow v}) - (k(v)+1)r &\leq \ell(v) + \type_r(t^{\uparrow v}) - (k(v)+1)r + \frac{1}{2}(r-\type_r(t^{\uparrow v}_*)) \\
        &= -(r-\type_r(t^{\uparrow v}_*)) + \frac{1}{2}(r-\type_r(t^{\uparrow v}_*)) < 0,
    \end{align*}
    Together, the inequalities imply that $k(v) = k_n(v)$.
    Inputting this identity into \eqref{eq:i_hate_continuity}, using the inductive hypothesis, and using the bound $\|\ell_n-\ell\|_\infty < \delta$ once again we get
    \begin{equation}\label{eq:hatred_towards_continuity_unwavering}
    |\type_r((t_n)_*^{\uparrow v})-\type_r(t_*^{\uparrow v})| < i\delta 
    \end{equation}
    If we show that the error does not get worse upon completing a merger step, then the claim will be proven for all vertices in generation $\fm-i$, finishing the proof. Let $v$ be a vertex in this generation, with set of children $v_1,...,v_k$ ordered such that $\type_r(t_*^{\uparrow v_1}) \geq ... \geq \type_r(t_*^{\uparrow v_k})$. By \cref{lem:recursive_tap}, it suffices to show that $i_n(v) = i(v)$, where
    $$
    i(v) = \inf\left\{2 \leq j \leq k :  \type_r\left(t_*^{\uparrow v_j}\right) + \type_r\left(t_*^{\uparrow v_{j-1}}\right) < r\right\}\wedge (k+1),
    $$
    and $i_n(v)$ is defined completely analogously for $t_n$. Indeed, for any $2 \leq j \leq k$ we have by \eqref{eq:hatred_towards_continuity_unwavering} and the definition of $\delta$,
    \begin{align*}
        &\left|\type_r\left((t_n)_*^{\uparrow v_j}\right) + \type_r\left((t_n)_*^{\uparrow v_{j-1}}\right) - \type_r\left(t_*^{\uparrow v_j}\right) - \type_r\left(t_*^{\uparrow v_{j-1}}\right)\right| \\
        &\leq \frac{1}{2}\left| r - \type_r\left(t_*^{\uparrow v_j}\right) - \type_r\left(t_*^{\uparrow v_{j-1}}\right) \right|.
    \end{align*}
    Then, for all $2 \leq j \leq k$, the following implications hold:
    $$
    \type_r\left((t_n)_*^{\uparrow v_j}\right) + \type_r\left((t_n)_*^{\uparrow v_{j-1}}\right) < r \iff \type_r\left(t_*^{\uparrow v_j}\right) + \type_r\left(t_*^{\uparrow v_{j-1}}\right) < r.
    $$ 
    This fact immediately implies that $i_n(v) = i(v)$ almost surely and completes the proof of continuity.

    To see the final fact about the continuity of $\ff$ at any given vertex, it suffices to note that in the just completed proof of the continuity for the type we showed that the number of packing centres on each edge $e \in t$ is the same for length $\ell'$ as for length $\ell$ whenever $\|\ell-\ell'\|_\infty < \delta$.
\end{proof}

\begin{lem}\label{lem:continuity_local_top}
    Let $(T,p,\ell,o)$ be a (random) one-ended tree with i.i.d.\ atomless edge lengths. Then, $\ff$ is almost surely continuous at $(T,p,\ell,o)$ in the marked local topology on $\mathfrak{T}$.
\end{lem}

\begin{proof}

    The sum of infinitely many i.i.d.\ atomless random variables is infinite almost surely, so there exists some finite $K$ such that $d^* := \sum_{i=1}^K \ell\{v_i,v_{i+1}\} > r$. This implies that the number of packing centres on the edge $\{v_0,v_1\}$ in $\pack_r(v_K;F_K)$ is equal to $\ff(T,p,\ell,o)$. Suppose that $(T_n,p_n,\ell_n,o_n) \to (T,p,\ell,o)$ in the marked local topology. By definition, for any $\delta > 0$, we can choose $N \geq 0$ large enough that for all $n \geq N$ there is some isomorphism 
    $$
    \phi_n:F_{d^*}(T_n,p_n,\ell_n,o_n) \to F_{d^*}(T,p,\ell,o),
    $$
    where $|\ell_n(e) - \ell(\phi_n(e))| < \delta$ for all $e \in F_{d^*}(T_n,p_n,\ell_n,o_n)$. This reduces the problem of continuity back to the problem of continuity for finite trees with i.i.d.\ atomless edge lengths, allowing us to directly apply \cref{lem:type_continuity_for_finite_trees}.
\end{proof}

With this fact proven we can prove our desired limit law.

\begin{thm}\label{thm:packing_is_local}
    Let $(T_n,\rho_n,\ell_n)_{n \geq 1}$ be a sequence of random rooted trees on $n$ vertices with edge lengths, let $(T,p,\ell,o) \in \fT$ be one-ended with i.i.d.\ atomless edge lengths, and let $(\beta_n)_{n \geq 1}$ be a sequence of positive real numbers. Let $O_n$ be a uniformly chosen vertex from $T_n$ for all $n \geq 1$. If $(T_n,\rho_n,\beta_n^{-1}\ell_n) \to (T,p,\ell,o)$ in probability with respect to the marked local topology and $(\ell_n(O_n)^2)_{n \geq 1}$ is uniformly integrable, then
    $$
    \frac{1}{n}\cP_{r\beta_n}(T_n) \convp \e[\ff(T,p,\ell,o)].
    $$
\end{thm}

\begin{proof}
    Most of the important work was done above already. Before moving forward with the final proof there are a couple things to note. First, convergence locally in probability immediately implies that $p$ must be infinite, as two neighbourhoods that are a bounded distance apart cannot have independent edge lengths. Second, scaling edge lengths by $\beta_n$ is equivalent to changing the radius of the packing by $\beta_n^{-1}$ for packing numbers, and so it suffices to complete the proof for $\beta_n =1 $. To simplify the following computations we will compress our notation by writing $\ff(T_n,O_n)$ in place for $\ff(T_n,p_{O_n},\ell_n,O_n)$, where, as before, $p_{O_n}$ is the path from $O_n$ to $\rho$.
    
    Let $K \geq 0$. Using \eqref{eq:packing_number_as_sum} we get that $\e[\cP_r(T_n)] = n\e[\ff(T_n,O_n)]$. Then, since $\ff(T_n,O_n)$ is at most $\ell(O_n)/r$ (all the packing centres on the edge must be distance at least $r$ from each other) we get
    $$
    \Big|\e[\ff(T_n,O_n)] - \e\left[\ff(T_n,O_n) \wedge K\right]\Big| \leq  \e\left[\frac{\ell_n(O_n)}{r} \I{\ell_n(O_n) \geq rK}\right].
    $$
    Taking a limit in $n$ and then $K$ we get, by the uniform integrability of $\ell_n(O_n)$ and the fact that $\ff(\cdot,\cdot) \wedge K$ is bounded and almost surely continuous at $(T,p,\ell,o)$,
    $$
    \lim_{n \to \infty} n^{-1}\e[\cP_r(T_n)] =  \lim_{n \to \infty} \e[\ff(T_n,O_n)] = \e[\ff(T,o)].
    $$
    If we can show that $\e[\cP_r(T_n)^2] = (1+o(1))(n\e[\ff(T,o)])^2$ then we are done by Chebyshev's inequality. Let $(V_n)_{n \geq 1}$ be another independent collection of uniform vertices in $T_n$. For the second moment, just as above, we can write $\e[\cP_r(T_n)^2] = n^2\e[\ff(T_n,O_n)\ff(T_n,V_n)]$. Then, by a similar argument to the one just performed we get,
    \begin{align*}
    &\lim_{n \to \infty} \left|n^{-2}\e[\cP_r(T_n)^2] - \e[\ff(T,o)]^2\right| \\
    &\leq \lim_{K \to \infty}\lim_{n \to \infty} 2\e\left[ \frac{\ell_n(O_n)\ell_n(V_n)}{r^2}\I{\ell_n(O_n) \geq \sqrt{rK}} \right] \\
    &\leq \frac{2}{r^2}\lim_{K \to \infty}\sup_n \e[\ell_n(V_n)^2]\e\left[\ell(O_n)^2\I{\ell_n(O_n) \geq \sqrt{rK}}\right],
    \end{align*}
    where the final bound is due to the Cauchy-Schwarz inequality. The last line is zero by uniform integrability.
\end{proof}

For this whole section we have allowed variation in the edge lengths. However, if we change our perspective and focus on trees with fixed edge lengths, i.e.,
    $$
    \fT_z = \{ (t,p,\ell,o) \in \fT : \ell(e) = z \text{ for all edges } e \in t  \},
    $$
then the continuity of $\ff$ becomes much clearer (to see this, note that $\fT_z$ is a closed set in the topology we have been using and that this topology is equivalent to the standard local topology on $\fT_z$). Moreover, the edge lengths are uniformly bounded by definition. Hence, we could perform a substantially simplified version of our arguments in this section to derive a corresponding limit law in the fixed edge length case which we will omit. Using the established local convergence in probability of Bienaym\'{e} trees with finite-variance offspring distribution, we can then derive a useful limit theorem for their packing numbers (see \cite[Theorem 7.12]{janson_simpygen}).

\begin{thm}\label{thm:unscaled_local_packing_number}
    Let $\mu$ be a distribution on $\N_0$ with mean $1$ and variance $0 < \sigma^2 < \infty$, and, for each $n \geq 1$ such that $\mu^{*n}\{n-1\} > 0$, let $T_n$ be a Bienaym\'{e} tree conditioned to have size $n$. Then, for each $r > 0$, there exists $C(r,\mu) > 0$ such that $n^{-1}\cP_r(T_n) \convp C(r,\mu)$. Moreover, as $r \to \infty$ $C(r,\mu) \to \fC\sigma^{-2}$, where $\fC$ is the absolute constant from~\cref{subsec:def_const}
\end{thm}

\begin{proof}
    The first part was explained before the statement of this theorem, and
    the second part is a consequence of a general probability theory fact that we will provide in the appendix (see \cref{app:consistency_limit} \cref{fact:basic_prob}).
\end{proof}

\section{Packing the local limit}\label{sec:covering_local_limit}

This section is dedicated to using \cref{thm:packing_is_local} to analyse packings of $\cS_n^\lambda$ in the large $n$ and $\lambda$ limit. For this, we introduce some notation concerning canonical $r$-TAPs of one-ended trees. We will define $\pack_r(t) = \pack_r(t,p,\ell,o)$ to be countable set of points which satisfy the property that, for any $i \geq 0$, $\pack_r(v_k ; F(v_k))\cap F(v_i)$ is equal to $\pack_r(t,p,\ell,o)\cap F(v_i)$ for large $k$. In accordance with the finite case, we call this set the \emph{canonical $r$-TAP} of $t$. One can see this set as the result of running the $r$-TAP algorithm on a tree rooted at $\infty$ `at the end of $p$'. Observe that $\ff_r(t,p,\ell,o)$ as defined in the previous section is the number of packing centres on $\{0,v_1\}$ in $\pack_r(t,p,\ell,o)$.

We start by proving \cref{prop:packing_is_local}, with $m_s=\e[M_s]$, for $M_s=\ff_s(T^*,P^*,\ell^*,o^*)$. That is, $M_s$ is the number of packing centres on the root-edge in the canonical $s$-TAP of the binary Kesten tree with independent $\mathrm{Exp(1)}$ edge lengths, $(T^*,P^*,\ell^*,o^*)$.

\packinglocal*


\begin{proof}
    The proof is an immediate application of \cref{thm:packing_is_local} at $T^*$.
\end{proof}

In the rest of the section, we will prove \cref{prop:asymp_m_s}. For an informal overview of the proof, we refer the reader to \cref{subsec:def_const}.

\proplocalpacking*



We aim to study the asymptotic growth of $\e[M_s]$, so we will first study its value for fixed $s$.


\begin{lem}\label{lem:expectation_Ms}
Let $E\sim \Exp(1)$, $0<X_1<X_2<\dots$ be the points in a rate $1$ Poisson point process conditional on $X_1<s$. Independently, let $B$ be  a binary Bienaymé tree with offspring distribution $\tfrac{1}{2}\delta_0+\tfrac{1}{2}\delta_2$ and independent $\Exp(1)$ edge lengths, and let $(B_i)_{i=1}^\infty$ be independent \emph{planted} binary Bienaymé trees with the same offspring and edge length distributions. Then, 
\begin{align*}\e[M_s]=&\P{\bt_s(B)=0}\cup \{E>s-\bt_s(B)\}\\
&\quad \times\left(\frac{e^{-s}}{1-e^{-s}}+\P{\forall k\geq 1:\bt_s(B_k)\not\in (X_k, s-X_k)}\right),\end{align*}
with the convention that the set $(a,b)$ is empty if $a\geq b$.
\end{lem}

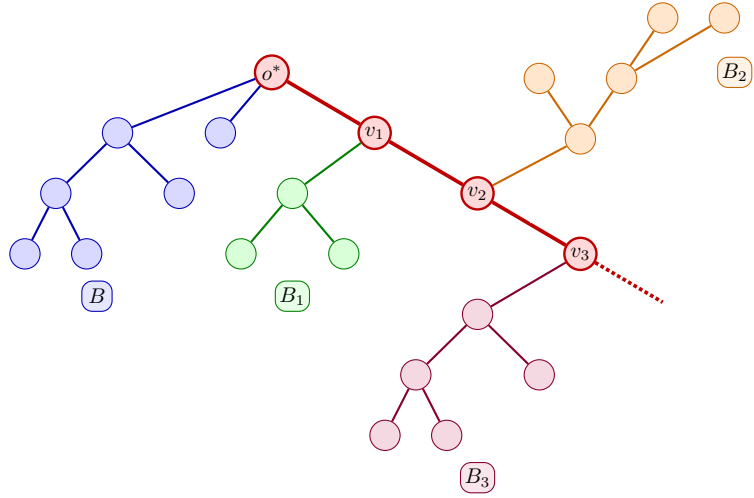
\begin{figure}
    \centering
\begin{tikzpicture}[
    scale=0.8, transform shape,
    x=1.7cm,
    y=1cm,
    every node/.style={
      circle,
      draw,
      inner sep=1.5pt,
      minimum size=5mm,
      font=\small
    },
    spine vertex/.style={
      circle,
      draw=red!75!black,
      fill=red!15,
      line width=0.9pt
    },
    spine edge/.style={
      draw=red!75!black,
      line width=1.4pt
    },
    gwblue/.style={
      draw=blue!70!black,
      fill=blue!15
    },
    gwgreen/.style={
      draw=green!50!black,
      fill=green!15
    },
    gworange/.style={
      draw=orange!80!black,
      fill=orange!20
    },
    gwpurple/.style={
      draw=purple!70!black,
      fill=purple!15
    },
    tree edge/.style={
      line width=0.8pt
    },
    edge label/.style={
      rectangle,
      draw=none,
      fill=white,
      inner sep=1pt,
      font=\small
    },
    tree label/.style={
      rectangle,
      rounded corners,
      inner sep=2pt,
      font=\small,
      draw
    }
    ]
    
    
    \node[spine vertex] (v0) at (0,0)   {$o^*$};
    \node[spine vertex] (v1) at (1,-1)  {$v_1$};
    \node[spine vertex] (v2) at (2,-2)  {$v_2$};
    \node[spine vertex] (v3) at (3,-3)  {$v_3$};
    
    \draw[spine edge]
      (v0) -- node[edge label,above right] {} (v1);
    
    \draw[spine edge]
      (v1) -- node[edge label,above right] {} (v2);
    
    \draw[spine edge]
      (v2) -- node[edge label,above right] {} (v3);
    
    
    \draw[spine edge,densely dotted]
      (v3) -- ++(0.8,-0.8);
    
    
    
    \node[gwblue] (t0a) at (-1.5,-1) {};
    \node[gwblue] (t0aright) at (-0.5,-1) {};
    \node[gwblue] (t0b) at (-2.1,-2) {};
    \node[gwblue] (t0c) at (-0.9,-2) {};
    \node[gwblue] (t0d) at (-2.4,-3) {};
    \node[gwblue] (t0e) at (-1.8,-3) {};
    
    \draw[tree edge,blue!70!black] (v0) -- (t0a);
    \draw[tree edge,blue!70!black] (v0) -- (t0aright);
    \draw[tree edge,blue!70!black] (t0a) -- (t0b);
    \draw[tree edge,blue!70!black] (t0a) -- (t0c);
    \draw[tree edge,blue!70!black] (t0b) -- (t0d);
    \draw[tree edge,blue!70!black] (t0b) -- (t0e);
    
    \node[
      tree label,
      draw=blue!70!black,
      fill=blue!10
    ] at (-1.7,-3.7) {$ B$};
    
    
    \node[gwgreen] (t1a) at (0.2,-2) {};
    \node[gwgreen] (t1b) at (-0.3,-3) {};
    \node[gwgreen] (t1c) at (0.7,-3) {};
    
    \draw[tree edge,green!50!black] (v1) -- (t1a);
    \draw[tree edge,green!50!black] (t1a) -- (t1b);
    \draw[tree edge,green!50!black] (t1a) -- (t1c);
    
    \node[
      tree label,
      draw=green!50!black,
      fill=green!10
    ] at (0.2,-3.7) {$ B_1$};
    
    
    \node[gworange] (t2a) at (3,-1.1) {};
    \node[gworange] (t2b) at (2.6,-0.1) {};
    \node[gworange] (t2c) at (3.4,-0.1) {};
    \node[gworange] (t2d) at (3.8,0.9) {};
    \node[gworange] (t2e) at (4.4,0.9) {};
    
    \draw[tree edge,orange!80!black] (v2) -- (t2a);
    \draw[tree edge,orange!80!black] (t2a) -- (t2b);
    \draw[tree edge,orange!80!black] (t2a) -- (t2c);
    \draw[tree edge,orange!80!black] (t2c) -- (t2d);
    \draw[tree edge,orange!80!black] (t2c) -- (t2e);
    
    \node[
      tree label,
      draw=orange!80!black,
      fill=orange!10
    ] at (4.5,0) {$ B_2$};
    
    
    \node[gwpurple] (t3a) at (2,-4) {};
    \node[gwpurple] (t3b) at (1.4,-5) {};
    \node[gwpurple] (t3c) at (2.6,-5) {};
    \node[gwpurple] (t3d) at (1.1,-6) {};
    \node[gwpurple] (t3e) at (1.7,-6) {};
    
    \draw[tree edge,purple!70!black] (v3) -- (t3a);
    \draw[tree edge,purple!70!black] (t3a) -- (t3b);
    \draw[tree edge,purple!70!black] (t3a) -- (t3c);
    \draw[tree edge,purple!70!black] (t3b) -- (t3d);
    \draw[tree edge,purple!70!black] (t3b) -- (t3e);
    
    \node[
      tree label,
      draw=purple!70!black,
      fill=purple!10
    ] at (2,-6.7) {$B_3$};
    
\end{tikzpicture}
    \caption{Sketch of the Kesten tree constructed in \cref{lem:expectation_Ms}. The spine is red. The trees $B$, $B_1, \ B_2$ and $B_3$ are in blue, green yellow and pink.}
    \label{fig:kesten_sketch}
\end{figure}

\begin{proof}

    Recall that $M_s=\#\{\pack_s(T^*)\cap \{o^*,v_1\}\}$, where $\{o^*,v_1\}$ is the line segment between the root $o^*$ and its first ancestor $v_1$. Recall from \cref{subsec:local_limit} that $F^*(o^*)$ denotes the subtree rooted at $v_1$ containing $F(o^*)$ and the edge $\{o^*,v_1\}$. For this proof, we let $\cM_s = \pack_s(T^*)\cap \{o^*,v_1\}$, $\cM_s' = \pack_s(F^*(o^*)) \cap \{o^*,v_1\}$, and $M_s' = \#\{\pack_s(F^*(o^*)) \cap \{o^*,v_1\}\}$. 

    We start by giving a convenient construction of $T^*$ using the Poisson process $(X_i)_{i=1}^\infty$, which will allow us to deduce $\e[M_s]$. Let $P^* = (v_0,v_1,\dots)$ be an infinite path and set $o^* := v_0$. We will build $T^*$ by defining the fringe trees pendent to $P^*$ (which come equipped with edge lengths) as well as the edge lengths in $P^*$. Set $F(o^*;T^*)=B$, and for $i\geq 1$ set $B_i$ to be the tree rooted at $v_i$, pendent to $P^*$ (that is, $B_i = F(v_i;T^*) \setminus F^*(v_{i-1};T^*)$). For each $i \geq 1$, let $\ell^*\{v_i,v_{i+1}\} = X_{i+1}-X_i$. The only thing left to construct is $\ell^*\{v_0,v_1\}$. For each $0 < t < s$, let $B_{s,t}$ be an independent $\operatorname{Ber}(e^{-(s-t)})$ random variable, and let $E_{s,t}$ be an independent $\operatorname{Exp}(1)$ random variable conditioned to be less than $s-t$. Additionally, let $G_s$ be an independent geometric random variable with support $\{1,2,\ldots\}$ and success probability $e^{-s}$. We set
    $$
    \ell^*\{v_0,v_1\} =
    \begin{cases}
        sG_s - t_s(B) + X_1, & \text{if } B_{s,t_s(B)} = 1 \\
        E_{s,t_s(B)}, & \text{if } B_{s,t_s(B)} = 0
    \end{cases}
    .
    $$
    By the memoryless property of the exponential distribution, $\ell^*\{v_0,v_1\}$ is an $\operatorname{Exp}(1)$ random variable that is independent of $t_S(B)$. (Indeed, first, $B_{s,t_s(B)}$ samples whether $\ell^*\{v_0,v_1\}$ exceeds $s-t_s(B)$. If so, then $G_s$ samples how many multiples of $s$ fit in $\ell^*\{v_0,v_1\}-(s-t_s(B))$, and $X_1$ samples the remainder of $\ell^*\{v_0,v_1\}-(s-t_s(B))$ under division by $s$.)

    On the event $A^c := \{B_{s,t_s(B)} = 0\}$, we have that $M'_s=0$ by \cref{lem:recursive_tap_2}. Hence, since $M_s \le M_s'$, we have that $M_s = 0$ on $A^c$.

    On the event $A$, again by using \cref{lem:recursive_tap_2}, we know that $M_s' = G_s$ and that $\type_s(F^*(v_0)) = X_1$. Then, to determine $M_s$ on the event $A$, we only need to determine if the packing centre closest to $v_1$ in $\pack_s(F^*(v_0))$, call it $p$, gets removed in a merger step. Using the structure of $T^*$ and \cref{lem:recursive_tap}, we see that this happens if and only if there is some $k \geq 1$ so that $\type_s(B_k) \in (d_{T^*}(p,v_k),s-d_{T^*}(p,v_k))$. By construction we have $d_{T^*}(p,v_k) = X_k$, which gives us
    \begin{align*}
        M_s = & \I{A}(M'_s-\I{\exists k\geq 1:\bt_s(B_k)\in (X_k, s-X_k)})\\
        =&
        \I{A}(M'_s-1+\I{\forall k\geq 1:\bt_s(B_k)\notin (X_k, s-X_k)}).
    \end{align*}
    Because $M'_s$ is equal to $G_s$ on $A$, the claimed expression for $\e[M_s]$ follows.
\end{proof}

We now examine the asymptotics of $\e[M_s]$ to prove \cref{prop:asymp_m_s}.

\begin{proof}[Proof of \cref{prop:asymp_m_s}]
We use the expression for $\e[M_s]$ given by \cref{lem:expectation_Ms}. First, remark that with probability $1/2$ the tree $B$ is an isolated vertex, whose type is $0$. This fact, in addition to a union bound, yields

\begin{align*}
    \frac{1}{2}
    \leq
    \P{\bt_s(B)=0}\cup \{E>s-\bt_s(B)\}
    \leq& \frac{1}{2}+\P{\bt_s(B)=0 }\cap \{ \#B>1 \}\\
    &+
    \P{E>s-\bt_s(B)}.
\end{align*}
If $B$ is not an isolated vertex and $\bt_s(B)=0$ then $\height(B)\geq s$, and so 
$$
\P{\bt_s(B)=0}\cap \{
\#B>1\}\leq \P{\height(B)\geq s}.
$$ Moreover, since 
$$
\P{E>s-\bt_s(B)}\leq \P{\bt_s(B)\geq s/2}+\P{E>s/2}
$$ 
and $\bt_s(B)\leq \height(B)$, we have 

\begin{align*}
    \frac{1}{2}
    \leq
    \P{\bt_s(B)=0}\cup \{E>s-\bt_s(B)\}
    \leq &
    \frac{1}{2}+2\P{\height(B)\geq s/2}
    +
    \P{E>s/2},
\end{align*}
so

\begin{equation}\label{eq:lim_first_term_MS}
    \lim_{s\to\infty}\P{\bt_s(B)=0}\cup \{E>s-\bt_s(B)\}=\frac{1}{2}.
\end{equation}
To evaluate $\P{\forall 1\le k \le K:\bt_s(B_k)\not\in (X_k, s-X_k)}$, recall that the $X_i$ are sampled from a rate 1 Poisson point process on $\R^+$ conditioned on $X_1<s$. We start by eliminating this conditioning, let $(\Tilde{X_i})_{i\geq 0}$ be sampled from a rate 1 Poisson point process on $\R^+$. A decomposition on the event $\Tilde{X}_1<s$ and its complement yields

\begin{equation}\label{eq:poisson_process_cdn_UB}
\begin{aligned}
    &\P{\forall 1\le k \le K:\bt_s(B_k)\not\in (\Tilde{X}_k, s-\Tilde{X}_k)}\\
    &\hspace{2cm}\leq e^{-s}+ \P{\forall 1\le k \le K:\bt_s(B_k)\not\in (X_k, s-X_k)}
\end{aligned}
\end{equation}

and 

\begin{equation}\label{eq:poisson_process_cdn_LB}
\begin{aligned}
    &\P{\forall 1\le k \le K:\bt_s(B_k)\not\in (\Tilde{X}_k, s-\Tilde{X}_k)}\\
    &\hspace{2cm}\geq (1-e^{-s})\P{\forall 1\le k \le K:\bt_s(B_k)\not\in (X_k, s-X_k)}.
\end{aligned}
\end{equation}
In the following, we evaluate $\P{\forall 1\le k \le K:\bt_s(B_k)\not\in (\Tilde{X}_k, s-\Tilde{X}_k)}$. We adopt a continuous point of view on the construction of this event. Equivalently, on the infinite path below the root, branch points appear at rate $1$, so trees appear in position $v$ with $\bt_s(B)\in (v,s-v)$ at rate

$$\P{ \bt_s(B)\in (v,s-v) },$$
which is non-zero only if $v\in[0,s/2)$. As a consequence, if we let $\Gamma_s$ be the integrated rate,

\begin{equation*}
    \Gamma_s=\int_{v=0}^{s/2} \p\left\{ \bt_s(B)\in (v,s-v) \right\} dv,
\end{equation*}
the number of trees attached to the infinite path below $\rho$ at a position $v$ and of type in $(v,s-v)$ is distributed as a Poisson random variable of parameter $\Gamma_s$. Hence, 

\begin{equation}\label{eq:proba_kill_as_poisson_process}
    \P{\forall 1\le k \le K:\bt_s(B_k)\not\in (\Tilde{X}_k, s-\Tilde{X}_k)}=e^{-\Gamma_s}.
\end{equation}
It remains to evaluate $\Gamma_s$. Remark that, if the height of a tree is less than $s/2$, then its height and type coincide. Splitting the event in the integral defining $\Gamma_s$ depending on the height of the tree yields

\begin{equation}\label{eq:gamma_s_int}
    \Gamma_s
    =
    \int_{v=0}^{s/2} \p\left\{ v<\height(B)<s/2 \right\} dv +\int_{v=0}^{s/2} \p\left\{ v<t_s(B)<s-v; \height(B)>s/2 \right\} dv.
\end{equation}
The first term is controlled using \cref{lem:height_continuous_Bienayme_tree}, that is,

\begin{equation}\label{eq:gamma_s_limit1}
    \int_{v=0}^{s/2} \P{v<\height(B)<s/2 } dv
    =
    \int_{v=0}^{s/2} \left( \frac{2}{2+v}-\frac{2}{2+s/2} \right) dv
    =
    2\log (s) -2.
\end{equation}
The second term of \eqref{eq:gamma_s_int} is controlled in the following lemma, which we pause to prove.

\begin{lem}\label{lem:integral_convergence}
 Let $T$ be a planted $\mu = \frac{1}{2}\delta_0 + \frac{1}{2}\delta_2$ offspring distribution Bienaym\'{e} tree with $\operatorname{Exp}(1)$ edge length and $\cT^{(1/2)}$ the tree encoded by a Brownian excursion conditioned to have height at least $1/2$. Then,
\begin{align*}
    &\int_{v=0}^{s/2}\prob\{v< \type_s(T)<s-v\; ; \;\height(T)>s/2\} \ dv \\
    &\quad \longrightarrow  4\int_0^{1/2} \prob\{v<\type_1(\mathcal{T}^{(1/2)})<1-v\} \ dv
\end{align*}
as $s\to \infty$.
\end{lem}

\begin{proof}
    First, the change of variable $v=v/s$ followed by a simple decomposition of probability yields

    \begin{align*}
        &\int_{v=0}^{s/2}\prob\{v< \type_s(T)<s-v\; ; \;\height(T)>s/2\} \ dv \\
         & \quad =
         s \int_{v=0}^{1/2}\prob\left\{v< \frac{1}{s}\type_s(T)<1-v\; ; \;\height(T)>s/2\right\} \ dv\\
         & \quad =
         s \prob\left(\height(T)>s/2\right)
         \int_{v=0}^{1/2} \prob\left\{v< \frac{1}{s}\type_s(T)<1-v \ \mid \ \height(T)>s/2\right\} \ dv.
    \end{align*}
    Next, we adapt the proof of \cite[Corollary 1.13]{legall2005} to establish convergence in distribution with respect to the Gromov-Hausdorff topology of height-conditioned Bienaymé trees with edge lengths. The proof of \cite[Corollary 1.13]{legall2005}, which goes through the height process, takes for input the convergence of the local time at $0$ and of the height process of a unconditioned Bienaymé forest to prove the convergence of height conditioned Bienaymé trees. In \cite[Theorem 1]{deraphelis2015}, the joint convergence of the height process of a Bienaymé forest with or without edge length is proven. The convergence of the local time at $0$ is also proven and does not depend on the edge lengths. Hence, combining \cite[Theorem 1]{deraphelis2015} and \cite[Corollary 1.13]{legall2005} proves convergence of the height process for height conditioned Bienaymé tree with edge lengths. By \cite[Lemma 2.4]{legall2005} and the familiar encoding of trees by their height processes, we get distributional convergence with respect to the Gromov-Hausdorff topology.

    Together with \cref{prop:convergence_of_the_type}, which establishes that the type is almost surely continuous at $\cT^{(1/2)}$, we have that

    \begin{align*}
        \lim_{s\to \infty}\P{v< \frac{1}{s}\type_s(T)<1-v \ \mid \ \height(T)>s/2} 
        \ = \ 
        \P{v<\type_1 (\mathcal{T}^{(1/2)})<1-v}.
    \end{align*}
    By \cref{lem:height_continuous_Bienayme_tree} we have that
    $$
    \lim_{s\to \infty}s \prob\left\{\height(T)>s/2\right\}
    =
    4.
    $$
    Applying the dominated convergence we can conclude the proof.
    
\end{proof}

Combining the results of \eqref{eq:gamma_s_limit1} and \cref{lem:integral_convergence} into \eqref{eq:gamma_s_int} leads to

\begin{equation}
\begin{aligned}
    \Gamma_s
    =&
    2\log(s)-2+4\int_0^{1/2} \prob\{v<\type_1(\mathcal{T}^{(1/2)})<1-v\} \ dv+o_s(1).
\end{aligned}
\end{equation}
Let

$$\fC=\frac{1}{2}\exp\left(2-4\int_0^{1/2} \prob\{v<\type_1(\mathcal{T}^{(1/2)})<1-v\} \ dv\right)$$
and plug the above into \eqref{eq:proba_kill_as_poisson_process} to get

$$\P{\forall 1\le k \le K:\bt_s(B_k)\not\in (\Tilde{X}_k, s-\Tilde{X}_k)}= (1+o_s(1))\cdot s^2 \fC .$$
Using the bounds \eqref{eq:poisson_process_cdn_UB} and \eqref{eq:poisson_process_cdn_LB} gives 

$$\P{\forall 1\le k \le K:\bt_s(B_k)\not\in (X_k, s-X_k)}= (1+o_s(1))\cdot 2 s^2 \fC.$$
Recalling \cref{lem:expectation_Ms} in combination with \eqref{eq:lim_first_term_MS} concludes the proof of \cref{prop:asymp_m_s}.

\end{proof}

\section{Packing the fringe}\label{sec:covering_fringe}
In this section, we prove the covering analogue of \cref{prop:fringe}, that a vanishing proportion of balls are needed to cover the remainder of the tree $T_n$ once the skeleton $\skel$ is already covered. \cref{prop:fringe} follows from using the relation between packing and covering number of \cref{prop:packink_is_covering}. We begin this section by proving the \cref{prop:fringe} from a number of structural lemmas about the forest obtained by removing $S_n^\lambda$ from $T_n$, and the remainder is devoted to proving those lemmas.


We write $k_{\operatorname{out}}(\skel)$ for the total number of neighbours in $(\skel)^c$ of the vertices in $\skel$. I.e. 
\[k_{\operatorname{out}}(\skel)=\sum_{v\in \skel}(k_{T_n}(v)-k_{\skel}(v))= \sum_{v\in \skel}(k_{T_n}(v)-1)+1 \]
where the last equality holds because $\skel$ is a tree.

The asymptotic growth of $\#\skel$ and $k_{\operatorname{out}}(\skel)$ can be deduced from Lemma 3.3 in \cite{addarioberry2025discretesnakesgloballycentered} by Addario-Berry, Goldschmidt, Mitchell and the third author.

\begin{restatable}{lem}{sizeskel}\label{lem:size_skel}
As $n\to \infty,$ for $1\ll \gamma_n\ll n^{1/2}$
\begin{equation}\label{eq:roots_fringe}
\frac{\gamma_n \# \skel }{n}\convp \frac{\sqrt{2\lambda}}{\sigma}
\end{equation}
and for $\gamma_n=n^{1/2}$, 
\begin{equation}\label{eq:roots_fringe_sqrt}
\frac{ \# \skel }{n^{1/2}}\convd \tau_{\lfloor \lambda \rfloor }
\end{equation}
where $\tau_{k}$ is the time of the $k$th increment in a counting process on $[0,\infty)$ with intensity $\tfrac{\sigma^2}{2}tdt$ at time $t$.  
In both cases,
\begin{equation}\label{eq:binaryskeleton}
\frac{ k_{\operatorname{out}}(\skel)}{\# \skel } \convp \sigma^2,
\end{equation}
 and it holds that $\frac{\gamma_n^2 \dstat_2(\skel)}{n}\convp \lambda $ as $n\to \infty$.
\end{restatable}

Morally, the edges enumerated by $k_{\operatorname{out}}(\skel)$ form the roots in the forest that makes up the fringe. We can guarantee a large number of roots by taking $\lambda$ large enough, and consequently, a large number of roots makes diameters of connected components in the fringe shorter and reduces its covering number. We will make this precise in this section.

Denote by $\fforest$ the \emph{fringe forest} obtained by removing the skeleton from the size conditioned $\mu$-Bienaym\'e tree, $\fforest = T_n\setminus\skel$. There are $k_{\operatorname{out}}(\skel)$ tree components of $\fforest$ which are rooted at the neighbours of $\skel$. Order the tree components using the lexicographic order of the roots in $T_n$. Furthermore, denote the distance of a vertex $v\in\fforest$ from $\skel$ by \[d(\skel, v) = \min\set{d(u, v): u\in\skel}\]
and the subtree rooted at $v$ is ${\fforest}^{\uparrow v}$.

The following is the main result of this section. In essence, it says that once we cover the skeleton $\skel$ using open balls of radius $\gamma_n$, then we may cover the fringe forest $\fforest$ (and thus the remainder of the tree $T_n$) using only an additional $\eps n/\gamma_n^2$ open balls of radius $\eps\gamma_n$. 

\begin{prop}\label{prop:fringe_output}
Let $1\ll\gamma_n\leq n^{1/2}$. For any $\eps,\delta>0$, there exist $R$ so that for $\lambda>R^2/\sigma^2$, then $\fforest\setminus\set{v\in\fforest~:~d(v,\skel)<\eps\gamma_n }$ can be covered with $\eps n/\gamma_n^2$ balls of radius $\eps\gamma_n$ with probability at least $1-\delta$ for all sufficiently large $n$. 
\end{prop}

To prove \cref{prop:fringe_output}, we require the following input, which bounds the proportion of vertices in $\fforest$ with subtrees taller than $\frac{\eps\gamma_n}{2}$.

\begin{lem}\label{lem:prop_tall_trees}
For any $\varepsilon,\delta>0$, there exists
$R>0$ such that, for every
$\lambda>R^2/\sigma^2$ and every sequence
$\gamma_n\to\infty$ with $\gamma_n\le n^{1/2}$,
\[\P{\#\set{v\in \fforest:\height\left({\fforest}^{\uparrow v}\right)>\frac{\eps\gamma_n}{2}}>\frac{\eps^2n}{2\gamma_n}} <\delta,\]
for all sufficiently large $n$.
\end{lem}

Assuming \cref{lem:prop_tall_trees} and employing a deterministic strategy provided in \cite[\S 4]{burnRTxf}, we may cover the fringe forest $\fforest$ with open balls of radius $\eps\gamma_n$ whenever the degree statistic of $\fforest$ is succeeded by a degree statistic in the set $N'$, where the appropriate set $N'$ will be given in the proof of \cref{lem:prop_tall_trees}. In the application of this strategy, we assume that some vertices in $\fforest$ fall within the covering of the skeleton $\skel$, specifically, we assume for $\varepsilon>0$, the vertices that are less than $\varepsilon\gamma_n$ distance of the skeleton are already covered.

First, we introduce some notation. Write, for $i\in \N$, the collection of vertices in the fringe forest that are distance $i$ away from the skeleton 
\[\ell_i = \set{v\in \fforest : d\left(v, \skel\right) = i}.\] 
Note that the roots of the fringe trees are at distance one from the skeleton. Second, we fix $k$ and we will capture the vertices in $\fforest$ that are at distance $j$ modulo $k$ from $\skel$ whose subtrees have height at least $k$ by the following sets for $j=0,\ldots, k-1$:
    \[\cC_k^j = \bigcup_{i=0}^\infty \set{v\in\ell_{ik+j} : \height\left({\fforest}^{\uparrow v}\right)\geq k}.\]
The next lemma shows that all vertices that are farther than $2k$ from the skeleton are contained in the $2k$-balls of $\cC_k^j$ .

\begin{lem}
    Let $k\in\N$ and $j\in\set{0,1,\ldots,k-1}$, then
    \[V\left(\fforest\right) =\left( \bigcup_{v\in \cC_k^j} B(v, 2k)\right)\cup\set{v\in\fforest~:~ d\left(v,\skel\right)< 2k}.\]
\end{lem}
\begin{proof}
Fix $k\in\N$ and $j\in\set{0,1,\ldots, k-1}$. Let $u\in\fforest$ and take 
$$i = \min\set{m\in\N_0 : d\left(u,\skel\right)< mk+j}.$$ If $i\geq 2$, then take $v\in\ell_{(i-2)k+j}$ to be the ancestor of $u$ that is distance $(i-2)k +j$ from $\skel$ so that $\height\left( {\fforest}^{\uparrow v} \right)\geq k$. Since $d(u,v)< 2k$, then $u\in B(v, 2k)$. If $i\in
\set{0,1}$, then $d\left(u,\skel\right)<2k$.
\end{proof}

Fix $k$ and $j$ and assume we have covered $\skel$. Suppose that for each tree in $\fforest$, the vertices at distance less than $2k$ from $\skel$ are captured in the covering of $\skel$. The previous lemma says that to additionally cover the rest of $\fforest$, we need only to cover the vertices in $\bigcup_{v\in \cC_k^j} B(v, 2k)$ for some $j$, which can be obtained by placing the centres of the covering at each of the vertices in $\cC_k^j$. Therefore, for an $\ell$ of interest, to ensure that at most $\ell$ additional balls are placed, we would need $\min_{j\in\set{0,1\ldots,k-1}}\#\cC_k^j$ to be at most $\ell$. We immediately provide the proof of \cref{prop:fringe_output} where $k=\eps\gamma_n/2$ and $\ell=\eps n/\gamma_n^2$:

\begin{proof}[Proof of \cref{prop:fringe_output}]
Let $\eps>0$ and for each $j\in\set{0, 1,\ldots,  \lfloor\frac{\eps\gamma_n}{2}\rfloor-1}$, set $X_j = \#\cC_{\lfloor\frac{\eps\gamma_n}{2}\rfloor}^j$. Then we have $\#\set{v\in \fforest:\height\left({\fforest}^{\uparrow v}\right)>\frac{\eps\gamma_n}{2}}=\sum_{j=0}^{\lfloor\frac{\eps\gamma_n}{2}\rfloor-1} X_j$ and that \[\min_{j\in\set{0,1,\ldots, \lfloor\frac{\eps\gamma_n}{2}\rfloor}}X_j\leq \frac{1}{\lfloor\frac{\eps\gamma_n}{2}\rfloor}\sum_{j=0}^{\lfloor\frac{\eps\gamma_n}{2}\rfloor-1} X_j = \frac{\#\set{v\in \fforest:\height\left({\fforest}^{\uparrow v}\right)>\frac{\eps\gamma_n}{2}}}{\lfloor\frac{\eps\gamma_n}{2}\rfloor}.\]
Thus for $\delta>0$, \cref{lem:prop_tall_trees} allows us to appropriately fix $\lambda$ so that
\begin{align*}\P{\min_{j\in\set{0,1,\ldots, \lfloor\frac{\eps\gamma_n}{2}\rfloor-1}}X_j > \frac{\varepsilon n}{\gamma_n^2}} &\leq \P{\frac{\#\set{v\in\fforest:\height\left({\fforest}^{\uparrow v}\right)>\frac{\eps\gamma_n}{2}}}{\lfloor\frac{\eps\gamma_n}{2}\rfloor} >\frac{\varepsilon n}{\gamma_n^2}}\\
&\leq \P{\#\set{v\in \fforest:\height\left({\fforest}^{\uparrow v}\right)>\frac{\eps\gamma_n}{2}}> \left\lfloor\frac{\varepsilon^2 n}{2\gamma_n}\right\rfloor}\\
&< \delta,
\end{align*}
for all $n$ large.
\end{proof}

From \cref{prop:fringe_output}, we immediately get the covering analogue of \cref{prop:fringe}.

\begin{prop}\label{prop:fringe_cover}
    Let $1\ll\gamma_n\leq \sqrt{n}$. For any $\eps>0$ it holds that \[
    \lim_{\lambda\to \infty}\limsup_{n\to \infty}\P{\cC_{\gamma_n}(T_n)>\cC_{(1-\eps)\gamma_n}(\skel)+\frac{\eps n}{\gamma_n^2}} =0.
    \]
\end{prop}
\begin{proof}
For brevity, write $F =\fforest\setminus\set{v\in\fforest~:~d(v,\skel)<\eps\gamma_n }$, and we note the following deterministic bound,
\[\cC_{\gamma_n}(T_n)\leq \cC_{(1-\eps)\gamma_n}(\skel) + \cC_{\eps\gamma_n}(F).\]
Indeed, take a minimum $(1-\eps)\gamma_n$-covering of $\skel$ consisting of $\cC_{(1-\eps)\gamma_n}(\skel)$ balls. Enlarging each ball to have radius $\gamma_n$ yields a $\gamma_n$-covering of $\skel$ that also covers every vertex of  $\fforest$ that is within $\eps\gamma_n$ distance of $\skel$. The remaining vertices in $F$ can be covered with $\cC_{\eps \gamma_n}(F)$ balls of radius $\eps\gamma_n$, and hence also with $\cC_{\eps \gamma_n}(F)$ balls of radius $\gamma_n$.

Now, for $\delta>0$, \cref{prop:fringe_output} yields there exists $R$ so that for all $\lambda\geq \frac{R^2}{\sigma^2}$, then for all $n$ sufficiently large we obtain
\[\P{\cC_{\gamma_n}(T_n)-\cC_{(1-\eps)\gamma_n}(\skel)>\frac{\eps n}{\gamma_n^2}} \leq \P{\cC_{\eps\gamma_n}(F)>\frac{\eps n}{\gamma_n^2}}< \delta,\]
giving the desired result.
\end{proof}

\begin{proof}[Proof of \cref{prop:fringe}]
    The proof follows from \cref{prop:fringe_cover} evaluated at at $\gamma_n/2$ and the relation between packing and covering, that is, $\cC_{\gamma_n/2}(T_n)=\cP_{\gamma_n}(T_n)$ and $\cC_{(1-\eps)\gamma_n/2}(\skel)=\allowdisplaybreaks\cP_{(1-\eps)\gamma_n}(\skel)$.
\end{proof}

\bigskip
\noindent {\bf Outline:}\\
The remainder of \cref{sec:covering_fringe} will comprise of the ingredients to prove \cref{lem:prop_tall_trees}, with its proof found at the end of \cref{sec:transfer_bounds}. We provide a brief overview of the steps going into this proof:
\begin{itemize}
    \item[\cref{sec:stochastic_domination}:] The fringe forest $\fforest$ with degree statistic $\dstat$ has the same law as $\rforest_\dstat$, a uniform sample among all plane forests with prescribed degree statistic $\dstat$. The aim of this section is to show that we may use $\rforest_{\dstat'}$ as our object of interest, where $\dstat'$ is an appropriate sub-binary degree statistic.
    
     \cref{dfn:partial_order} provides a partial order $\preceq$ on degree statistics and \cref{lem:degree_order} shows that $\dstat$ precedes a sub-binary degree statistic $\dstat'$ by constructively showing how to move up the partial order. We then utilize the tools showing stochastic domination in random trees with fixed degree statistic, particularly the ideas in \citet{height} and \citet{donderwinkel2026tight}, to achieve \cref{cor:stochastic_order_forest} that shows that $\preceq$ induces a stochastic ordering on the heights of the subtrees of the corresponding uniform forests with a given degree statistic. \cref{lem:more_tall_trees_when_subbinary} consequently shows that the expected number of tall subtrees in $\fforest$ is bounded above by the expected number of tall subtrees in $\rforest_{\dstat'}$.\\
    
    \item[ \cref{sec:number_tall_trees}:] To show that the probability $\rforest_{\dstat'}$ has many tall subtrees is small, it is enough to apply Markov's inequality and show that the probability a uniformly rooted subtree is tall is sufficiently small. The objective of this section is to show that a uniformly rooted subtree of $\rforest_\bdstat$ is unlikely to be tall, where $\bdstat$ is a binary degree statistic obtained from $\dstat'$. 
    
    We reduce the sub-binary degree statistic $\dstat'$ to a binary degree statistic $\bdstat$ by removing all out-degree one vertices: $\bdstat = \dstat'-(0, \dstat'_1, 0, \ldots)$. For a uniform vertex $V$ of $\rforest_\bdstat$, we shift our focus on the probability that $\rforest_\bdstat^{\uparrow V}$ is tall. 
    
    Observation \eqref{eq:uniform_binary} says that $\rforest_\bdstat$ has the same distribution as $F$, a $\mu$-Bienaym\'e forest conditioned to have the same size and number of components as $\bdstat$ whenever $\mathrm{supp}(\mu)=\{0,2\}$. Then, \cref{lem:bridge_shift} allows us to encode $\rforest_\bdstat^{\uparrow V}$ using a cyclic shift of the \L ukasiewicz encoding of $F$. Afterwards, within \cref{lem:height_binary_ds}, we compare functions of $\rforest_\bdstat^{\uparrow V}$ to functions of a Bienaym\'e tree using a change of measure argument. Setting the function of interest to be the height of the encoded tree,  \cref{lem:sub_critical_height_profile} (shows that the probability a sub-binary Bienaym\'e tree is tall decays as the offspring distribution becomes slightly subcritical) along with \cref{lem:first_exc_cbridge} (controls the length of the excursion encoding $\rforest_\bdstat^{\uparrow V}$) allow us to obtain the desired bound on $\P{\height\left( \rforest_{\bdstat}^{\uparrow V}\right)>k}$.\\
    
    \item[\cref{sec:transfer_bounds}:] The tail bound for $\height\left( \rforest_{\bdstat}^{\uparrow V}\right)$ in \cref{lem:height_binary_ds} plus Markov's inequality is strong enough to derive a result like \cref{lem:prop_tall_trees} for uniformly rooted subtrees in  $\rforest_\bdstat$. Much of the work in this last subsection goes towards converting the bound to the setting of uniformly rooted subtrees in $\rforest_{\dstat'}$.We finish with the proof of \cref{lem:prop_tall_trees}.
    
    \cref{lem:height_subbinary_ds} allows us to transfer the tail bounds by looking at vertex-to-leaf paths in $\rforest_\bdstat$ and adding in out-degree one vertices. The analysis of the length of the vertex-to-leaf paths in $\rforest_{\dstat'}$ is obtained by viewing the addition of out-degree one vertices as an urn experiment.

    \cref{fact:non_uniform} is a non-uniform-in-$\gamma_n$ version of \cref{lem:prop_tall_trees}. This is proved by conditioning on whether we may find a suitable sub-binary degree statistic $\dstat'$ that succeeds $\dstat$ in $\preceq$, in which case we may obtained the desired bound by combining \cref{lem:more_tall_trees_when_subbinary}, \cref{lem:height_binary_ds}, and \cref{lem:height_subbinary_ds}. Furthermore, we show that the case where there is no suitable $\dstat'$ is a low probability event using \cref{lem:bounded_deg1} and \eqref{eq:roots_fringe}.

    Lastly, \cref{lem:prop_tall_trees} is proved using \cref{fact:non_uniform} and a monotonicity argument on the proportion of tall subtrees of the fringe forest when the skeleton is instead spanned by the root and $\lfloor \lambda n/g^2\rfloor$ uniform leaves, for some large $g\leq \sqrt n$.
\end{itemize}

\subsection{Stochastic domination}\label{sec:stochastic_domination}
We aim to bound the proportion of tall subtrees in the fringe forest $\fforest$ and we will do so by deriving a bound on the proportion of tall subtrees in a sub-binary forest. Then, we can transfer the bound because subtree heights in sub-binary forests stochastically dominate the subtree heights in $\fforest$. This stochastic domination is a consequence of a partial order on the degree statistics, which is developed in this subsection.

Denote by $\dsForest(\dstat)$ the set of plane forests with degree statistic $\dstat$, and denote by $\rforest_\dstat\in_u\dsForest(\dstat)$ a uniformly chosen plane forest with fixed degree statistic $\dstat$. Furthermore, for a vertex $v$, denote $\rforest_\dstat^{\uparrow v}$ to be the induced subtree of $\rforest_\dstat$ rooted at $v$. When the degree statistic $\dstat$ corresponds to one component, then we instead write $\rtree_\dstat\in_u\dsTree(\dstat)$ and $\rtree_\dstat^{\uparrow v}$ respectively.

We will define a partial order $\preceq$ on degree statistics by specifying a covering relation \footnote{
	For a partially ordered set $(\mathcal{P},\preceq)$, $y \in\mathcal{P}$ covers $x \in \mathcal{P}$ if $x \prec y$ and for all $z \in \mathcal{P}$, if $x \preceq z \preceq y$ then $z \in \{x,y\}$. The term \emph{covering relation} is standard nomenclature and will only be used in defining the partial order $\preceq$. Outside of this, the usage of cover or covering will refer to \cref{def:covering}.}.
\begin{dfn}\label{dfn:partial_order}
    Let $\dstat = (\dstat_i)_{i\geq0},\ \dstat'=(\dstat_i')_{i\geq0}$ be degree statistics corresponding to plane forests with the same number of components, so $\sum_{i\geq 0}\dstat_i-\sum_{i\geq0}i\dstat_i = \sum_{i\geq 0}\dstat_i'-\sum_{i\geq0}i\dstat_i'$. Then $\dstat'$ covers $\dstat$ if either of the following holds:
\begin{enumerate}
\item  It holds that $\dstat_1'=\dstat_1+1$ and $\dstat_i'=\dstat_i$ for $i\neq 1$, or
\item There exists $k,\ell$ with $1\leq k\leq \ell$ such that
    \[\dstat_i'= \dstat_i +\I{i=k}-\I{i=k-1}+\I{i=\ell}-\I{i=\ell+1}.\]
\end{enumerate}
\end{dfn}
From this definition of $\preceq$, we can obtain $\dstat'$ from $\dstat$ by either adding a vertex with one child, or for $1\leq k\leq \ell$, we replace one vertex with $k-1$ children and one vertex with $\ell+1$ children by a vertex with $k$ children and one with $\ell$ children. In the latter, this informally says that the out-degrees in $\dstat'$ are less skewed than the out-degrees in $\dstat$. We will compare degree statistics of the same size and same number of components, so we will focus on moving along the partial order using $(2)$.

The next theorem states that the partial order $\preceq$ induces a stochastic ordering on the
heights of uniformly rooted subtrees of random plane forests with a given degree statistic.

\begin{thm}\label{cor:stochastic_order_forest}
Fix $r\geq 1$ and let $\preceq$ be the partial order on degree statistics as in \cref{dfn:partial_order}.
Take $\dstat=(\dstat_i)_{i\geq 0}$, $\dstat'=(\dstat_i')_{i\geq 0}$ to be degree statistics corresponding to $r$ components and have the same size.
Let $\rforest_\dstat, \rforest_{\dstat'}$ to be uniform forests from $\dsForest(\dstat),\ \dsForest(\dstat')$ respectively, and take $V, V'$ to be uniformly chosen vertices from $\rforest_\dstat$ and $\rforest_{\dstat'}$ . Then,
    \[\dstat\preceq\dstat' \implies \height\left(\rforest_\dstat^{\uparrow V}\right) \preceq_{\mathrm{st}}\height\left(\rforest_{\dstat'}^{\uparrow V'}\right).\]
\end{thm}
\begin{proof}
    This result is an immediate consequence of \cref{thm:stochastic_order_tree}, which can be found in \cref{sec:postponed}. We can put this in the setting of trees by adding the unit vector $e_r$ with $1$ in the $r$-th coordinate to the degree statistics: $\dstat+e_r,\ \dstat'+e_r$. Take $\rtree,\ \rtree'$ to be uniform samples from the respective sets $\dsTree(\dstat+e_r),\ \dsTree(\dstat'+e_r)$ restricted to trees with root degree $r$, then \cref{thm:stochastic_order_tree} gives the desired result.
\end{proof}

Take the fringe forest $\fforest$ and let $\dstat$ be its degree statistic. The next lemma gives a function that outputs a sub-binary degree statistic $\dstat'$ that corresponds to a forest with the same size, same number of components, and fewer leaves so that $\dstat$ precedes $\dstat'$ in the partial order.

\begin{lem}\label{lem:degree_order}
Let $\dstat = (\dstat_i)_{i\geq 0}$ be a degree statistic corresponding to $r$ components and with size $n$. From $\dstat$, we may obtain a sub-binary degree statistic $\dstat'$ with the same size, corresponding to the same number of components and $\dstat'_0\leq\dstat_0$, so that $\dstat\preceq\dstat'$.

\end{lem}
\begin{proof}
We introduce two operations on degree statistics. Let $\mathbf m=(\mathbf m_i)_{i\geq0}$ be a degree statistic and in the following operations assume that $\mathbf m_k, \mathbf m_\ell>0$:
\begin{itemize}[]
\item {\bf Operation A:} Take $k=1$ and $\ell\geq 3$. For all $i$, set 
\[\mathbf m_i' = \mathbf m_i +\I{i=2} - \I{i=1}+\I{i=\ell-1}-\I{i=\ell}.\]
\item {\bf Operation B:} Take $k=0$ and $\ell\geq 3$. For all $i$, set 
\[\mathbf m_i' = \mathbf m_i +\I{i=1} - \I{i=0}+\I{i=\ell-1}-\I{i=\ell}.\]
\end{itemize}
By \cref{dfn:partial_order} (2), each of the two operations produces a degree statistic $\mathbf m'=(\mathbf m_i')_{i\geq 0}$ that is higher up the partial order while maintaining the same size $n$ and same number of components. Note that Operation B reduces the number of out-degree zero vertices by one, while Operation A leaves this value unchanged. Operation A reduces the number of out-degree one vertices by one, while Operation B will increase them by one. 

Let $\dstat=(\dstat_i)_{i\geq 0}$ be a degree statistic corresponding to $r$ components and has size $n$. We assume $\dstat$ is not sub-binary, otherwise we may take $\dstat'=\dstat$. 
We see that $\sum_{k\geq 3}\dstat_k$ is the number of vertices with more than two children and $m:=\sum_{k\geq 3}(k-2)\dstat_k$ is the number of children that need to be transferred to a vertex with exactly one child or to a leaf in order to obtain a sub-binary degree statistic.

We construct $\dstat'$ by applying Operation A as many times as necessary and then applying Operation B if there still remain vertices with more than two children, where each application of the operations reduces $m$ by one. Note that it will always be possible to apply Operation B since any degree statistic realizes to a forest, which has leaves. 

If $m\leq \dstat_1$, then apply Operation A $m$ times and let $\dstat'$ be the resulting degree statistic. Note that $\dstat'$ has the same size and corresponds to the same number of components as $\dstat$. Furthermore, $\dstat'$ is sub-binary since each of the $m$ children have been transferred to out-degree-one vertices. Lastly, since only Operation A was applied, then $\dstat_0'=\dstat_0$.

If $m> \dstat_i$, then obtain the degree statistic $\dstat'$ by applying Operation A $\dstat_1$ times and applying Operation B $m-\dstat_1$ times. Again, $\dstat'$ has the same size and corresponds to the same number of components as $\dstat$. As well, $\dstat'$ is sub-binary since the $m$ children have been transferred to vertices with one child or to leaves. Lastly, $\dstat_0'<\dstat_0$ since Operation $B$ was applied.
\end{proof}

Lastly, the forests with degree statistic $\dstat'$ will produce more tall subtrees than the forests with degree statistic $\dstat$.
\begin{lem}\label{lem:more_tall_trees_when_subbinary}
Let $\dstat =(\dstat_i)_{i\geq0},\ \dstat'=(\dstat_i')_{i\geq0}\in \N_0^\N$ be degree statistics corresponding to the same number of components and with the same size such that $\dstat'$ is sub-binary and $\dstat\preceq\dstat'$. With $\rforest_\dstat\in_u\dsForest(\dstat)$ and $\rforest_{\dstat'}\in_u\dsForest(\dstat')$, then
\[\E{\#\set{v\in \rforest_\dstat ~:~\height\left(\rforest^{\uparrow v}_{\dstat}\right)>k} } \leq\E{\#\set{v\in \rforest_{\dstat'} ~:~\height\left(\rforest^{\uparrow v}_{\dstat'}\right)>k} }.\]
\end{lem}
\begin{proof}
Immediately from \cref{cor:stochastic_order_forest}, for any $k\in\R$ and uniformly chosen $V\in_u\rforest_\dstat$, $V'\in_u\rforest_{\dstat'}$,
\begin{align*}
\E{\#\set{v\in \rforest_\dstat ~:~\height\left(\rforest^{\uparrow v}_{\dstat}\right)>k)} } 
&= n\P{\height\left(\rforest^{\uparrow V}_{\dstat}\right)>k}\\
&\leq n\P{\height\left(\rforest^{\uparrow V'}_{\dstat'}\right)>k}\\
&=\E{\#\set{v\in \rforest_{\dstat'}~:~\height\left(\rforest^{\uparrow v}_{\dstat'}\right)>k} \label{eq:more_tall_trees}}.
\end{align*}
\end{proof}

In the next section, we will obtain a bound on the proportion of vertices with tall subtrees in a Bienyam\'e forest with $k_{\operatorname{out}}(\skel)$ roots, then \cref{lem:more_tall_trees_when_subbinary} will be used to extend this bound to the fringe forest $\fforest$.

\subsection{The likelihood of a tall subtree}\label{sec:number_tall_trees}
We will analyse the proportion of vertices with tall subtrees of a Bienaym\'e forest conditioned to have size $\#\fforest$ and $k_{\operatorname{out}}(\skel)$ roots.

Note that, conditional on its degree statistic, a Bienaym\'e tree is uniformly distributed over the set of trees with that degree statistic. Furthermore, conditioning on the skeleton fixes the degree statistic of the fringe forest. The plane forests with a given degree statistic are in bijection with trees having the corresponding skeleton: the bijection is obtained by assigning the trees of the forest to the outgoing edges of the skeleton in depth-first order. Consequently, conditional on its degree statistic $\dstat = (\dstat_i)_{i\geq 0}$, the fringe forest $\fforest$ is uniformly distributed over the set of plane forests with degree statistic $\dstat$.

Then, from \cref{sec:stochastic_domination}, we can stochastically dominate the heights of subtrees in $\fforest$ by looking at a sub-binary degree statistic $\dstat'=(\dstat_i')_{i\geq 0}$ with the same size, same number of components, and $\dstat_0'\leq \dstat_0$. In particular, \cref{lem:more_tall_trees_when_subbinary} will allow us to extend a bound on the proportion of tall subtrees in a forest with degree statistic $\dstat'$ to the proportion of tall subtrees observed in $\fforest$.

To obtain the bound for plane forests with degree statistic $\dstat'$, we will reduce $\dstat'$ to a binary degree statistic by removing all out-degree one vertices: $\bdstat = \dstat'-(0,\dstat_1', 0, \ldots)$. So, for $\dstat$ and $\dstat'$ having size $n$, then $\bdstat$ is a binary degree statistic corresponding to the same number of components and has size $n-\dstat_1'$. Motivated by \cref{lem:size_skel}, we first focus on the case where we have a binary forest with $n-\dstat_1'$ vertices and $\frac{Rn}{\gamma_n}$ roots. 

Starting with the degree statistic $\dstat$ of $\fforest$, let $\dstat'$ be a sub-binary degree statistic obtained from \cref{lem:degree_order}. The next lemma shows that both degree statistics $\dstat,\dstat'$ observes $\dstat_1$ and $\dstat_1'$ bounded away from $\#\fforest$.

\begin{lem}\label{lem:bounded_deg1}
    Let $\mu$ be an offspring distribution with mean $1$ and finite variance $0<\sigma^2<\infty$, and fix $\lambda>0$. Let $\fforest$ be the fringe forest obtained from removing the skeleton $\skel$ from the size-conditioned $\mu$-Bienaym\'e tree $T_n$. Let $\dstat$ be the degree statistic of $\fforest$ and $\dstat'$ be the sub-binary degree statistic obtained from \cref{lem:degree_order}. There exists $\beta = \beta(\mu)<1$ so that with high probability, $\dstat_1 \leq \beta\#\fforest$ and $\dstat_1'\leq \beta\#\fforest$.
\end{lem}
\begin{proof}
Take $T_n$ to be a size-conditioned Bienaym\'e tree with offspring distribution $\mu=(\mu_i)_{i\geq0}$. From \cref{lem:size_skel}, we deduce that for $\eps>0$, 
\begin{equation}\label{eq:fringe_size}
\liminf_{n\rightarrow\infty}{\#\fforest/n} =  \liminf_{n\rightarrow\infty}(1-\#\skel/n) \geq (1-\eps).
\end{equation}

Let $\mathbf m= (\mathbf m)_{i\geq 0}$ and $\dstat=(\dstat_i)_{i\geq 0}$ be the respective degree statistics of $T_n$ and $\fforest$, and observe that for each $i\geq0$ we have $\dstat_i\leq \mathbf m_i$. Furthermore, a standard result for Bienaym\'e trees (for example see \cite[Theorem 7.11(ii)]{janson_simpygen}) states for each $i\geq0$, 
\[\frac{\mathbf m_i}{n} \overset{p}{\longrightarrow} \mu_i.\]
In particular, we have that 
\begin{equation}\label{eq:empirical_dist1}
\frac{\mathbf m_1}{n} \overset{p}{\longrightarrow} \mu_1
\end{equation}
and using the fact that $\mathbf m$ is the degree statistic of a tree, 
\begin{equation}\label{eq:empirical_dist_sum}
    \frac{1}{n}\left( \sum_{i\geq 3}(i-2)\mathbf m_i\right) = \frac{1}{n}\left( 2\mathbf m_0 + \mathbf m_1 - (n+1) \right)  \overset{p}{\longrightarrow} 2\mu_0+\mu_1-1. 
\end{equation}
Since $\mu$ is critical and non-degenerate, then we observe that $M:=\max\set{\mu_1,\ 2\mu_0+\mu_1-1}<1$ and, we further have $\frac{M}{1-\eps}<1$ for $\eps>0$ chosen small enough. Take $\beta \in \left(\frac{M}{1-\eps},\ 1 \right)$ and note that $\beta(1-\eps)\in(M,1)$. 

Immediately from \eqref{eq:fringe_size}, \eqref{eq:empirical_dist1}, and observing that $\dstat_1 \leq \mathbf m_1\leq \beta(1-\eps)n$, we have $\P{\dstat_1 \leq \beta\#\fforest}\underset{n\rightarrow\infty}{\longrightarrow} 1$.

Now take the sub-binary degree statistic  $\dstat'$ obtained from \cref{lem:degree_order}, then from the construction of $\dstat'$, we have \[\dstat'_1 \leq \max\set{\dstat_1,\ \sum_{i\geq3}(i-2)\dstat_i-\dstat_1}.\] 
Applying \eqref{eq:fringe_size}, \eqref{eq:empirical_dist1} and \eqref{eq:empirical_dist_sum} with the observations $\dstat_1 \leq \beta(1-\eps)n$ 
and 
\[\sum_{i\geq3}(i-2)\dstat_i-\dstat_1\leq\sum_{i\geq3}(i-2)\dstat_i\leq\sum_{i\geq3}(i-2)\mathbf m_i\leq\beta(1-\eps)n,\] 
yields $\P{\dstat_1' \leq \beta\#\fforest}\underset{n\rightarrow\infty}{\longrightarrow} 1$.
\end{proof}

To bound the number of tall subtrees within a binary forest with given size, we can union bound over all possible subtrees. Thus, it is enough to zoom in on a uniformly rooted subtree and analyse how likely that that particular tree is tall.

Let $F_{\mu,n,k}$ denote a $\mu$-Bienaym\'e forest conditioned to have $n$ vertices and $k$ components. In particular, in the binary case where $\mathrm{supp}(\mu)=\set{0,2}$, this conditional distribution does not depend on $\mu$. Moreover, since the binary degree statistic $\bdstat$ is uniquely determined by its size $n-\dstat_1'$ and number of components $\frac{Rn}{\gamma_n}$, then $F_{\mu,n-\dstat_1',\frac{Rn}{\gamma_n}}$ is uniform on $\dsForest(\bdstat)$. It follows for $\rforest_\bdstat\in_u\dsForest(\bdstat)$, we have that \begin{equation}\label{eq:uniform_binary}
    \rforest_\bdstat \overset{d}{=}F_{\mu,n-\dstat_1',\frac{Rn}{\gamma_n}}.
\end{equation}
    
Take a uniform vertex $V\in\rforest_\bdstat$. The subtree $\rforest_\bdstat^{\uparrow V}$ can 
then be studied by sampling a \L ukasiewicz path encoding $F_{\mu,n-\dstat_1',\frac{Rn}{\gamma_n}}$ and examining the excursion above the running infimum that starts at a uniformly chosen time. Equivalently, we can examine the first excursion of a uniform cyclic shift of that same \L ukasiewicz path. The following lemma lets us study the first excursion of a conditioned random walk instead.

\begin{lem}\label{lem:bridge_shift}
Let $\mu$ be an offspring distribution and let $F_{\mu, n,k}$ be a $\mu$-Bienaym\'e forest conditioned to have $n$ vertices and $k$ components. Denote the \L ukasiewicz path of $F_{\mu,n,k}$ by \[W:=W(F_{\mu,n,k})=(W_0, W_1, \ldots, W_n)\] so that $W_0=0$ and $W_n=-k$. Let $i\in[n]$ and define the cyclic shift 
\[W^{(i)} = (0, W_{i+1}-W_i,\ W_{i+2}-W_i,\ \ldots,\ W_{i+n}-W_i),\]
with indices taken modulo $n$.

Let $X=(X_j)_{j\geq0}$ be random walk with $X_0=0$ and steps independently distributed as $J-1$ where $J\sim\mu$. Then for $I\in_u[n]$, the cyclic shift $W^{(I)}$ is distributed as the random walk $(X_0, X_1, \ldots, X_n)$ conditioned on $X_n=-k$. 
\end{lem}
\begin{proof}
Let $I\in_u[n]$ and $w=(w_0,w_1,\ldots, w_n)$ be a sequence satisfying $w_0=0$ and $w_n=-k$. Then,
\begin{align*}
\P{W^{(I)}=w} &=\frac{1}{n}\sum_{i=1}^n\P{W^{(i)}=w}=\frac{1}{n}\sum_{i=1}^n\P{W=w^{(-i)}}\\
&=\frac{k}{n}\P{W=w^{(-i^*)}} \tageq\label{eq:cycle_lemma},
\end{align*}
where the last line follows from the cycle lemma \cite[Lemma 6.1]{pitman2006combinatorial} and the shift $-i^*$ is one of the $k$ shifts so that $w^{(-i^*)}$ hits $-k$ for the first time at time $n$.

Denote the hitting time of $-k$ for the random walk $X$ by
    \[\tau_{-k}(X) = \min\set{i\geq0 ~:~ X_i=-k}.\]
For each $j\in[n]$, write $\Delta w^{(-i^*)}_j= w_{i^*+j}-w_{i^*+j-1}$. Using the fact that $W$ is distributed as $(X_0,\ldots,X_{\tau_{-k}(X)})$ conditioned on $\tau_{-k}(X)=n$ (\cite[Proposition 1.4]{kortchemski2012invariance}), then \eqref{eq:cycle_lemma} becomes
\begin{align*}
  \P{W^{(I)}=w}
  &=
  \frac{k}{n}\cdot\frac{\prod_{j=1}^n\mu(\Delta w^{(-i^*)}_j +1)}{\P{\tau_{-k}(X)=n}}
  =
  \frac{\frac{k}{n}\prod_{j=1}^n\mu(\Delta w^{(-i^*)}_j +1)}{\frac{k}{n}\P{X_n=-k}}\\
  &=
  \probC{(X_0, X_1, \ldots, X_n) =w}{X_n=-k},
\end{align*}
where the second equality applies Kemperman's formula \cite[(6.3)]{pitman2006combinatorial}.
\end{proof}

We have reduced studying the uniformly rooted subtree $\rforest_\bdstat^{\uparrow V}$ to studying the first excursion of a conditioned random walk. The following lemma controls the length of this excursion whose proof is deferred to \cref{sec:postponed}.

\begin{restatable}{lem}{firstexcbridge}\label{lem:first_exc_cbridge}
    Let $X=(X_k)_{k\geq 0}$ be a simple symmetric random walk and denote the hitting time of $-1$ by
    \[\tau_{-1}(X) = \min\set{i\geq0 ~:~ X_i=-1}.\] 
    Take $\gamma_n\rightarrow\infty$ so that $\gamma_n=O(\sqrt{n})$. Let $\eps>0$ and fix $\alpha\in[0,1)$ bounded away from $1$. For fixed $R$ large enough, then conditioning on $X_{n(1-\alpha)}=-\frac{Rn}{\gamma_n}$ yields
    \[\probC{\tau_{-1}(X)>\left\lceil\frac{n(1-\alpha)}{2}\right\rceil}{ X_{n(1-\alpha)}=-\frac{Rn}{\gamma_n}}\leq \eps n^{-1/2},\]
    for all $n$ large enough.
\end{restatable}

The following lemma will be used to control the height of $\rforest_\bdstat^{\uparrow V}$. It shows that the probability a sub-binary Bienaym\'e tree is tall decays as the offspring distribution becomes slightly subcritical.
\begin{lem}\label{lem:sub_critical_height_profile}
Let $\alpha\in[0,1)$, $t>0$ and $m$ large enough so that
\begin{equation}\label{eq:m_big}
    tm^{-1}< 1-\alpha.
\end{equation} Let $T_\mu$ be a $\mu$-Bienaym\'e tree with offspring  distribution
\[\mu=\frac{1}{2}(1-\alpha+tm^{-1})\delta_0+\alpha\delta_1+\frac{1}{2}(1-\alpha-tm^{-1})\delta_2.\] 
Then,\[\limsup_{m\rightarrow\infty}m\P{\height( T_\mu)>m} \leq \frac{e^{-t}2t}{(1-\alpha)(1-e^{-t})}\]
and thus \[\lim_{t\rightarrow\infty}\limsup_{m\rightarrow\infty}m\P{\height( T_\mu)>m}=0.\]

\end{lem}
\begin{proof}
We closely follow the proof of Theorem 1 in \cite{kesten1966galton} which deals with critical offspring distributions, and adapt to our setting with a sub-critical offspring distribution. The probability generating function for $\mu$ is
\[f(x)= \frac{1}{2}(1-\alpha+tm^{-1})+\alpha x+\frac{1}{2}(1-\alpha-tm^{-1})x^2,\]
and its iterates are
\[f_0(x)=x,\quad f_{k}(x)=f(f_{k-1}(x)) \textrm{ for } k\geq1.\]
In particular, $f_{m}(0) = \P{\textrm{generation $m$ is empty}}$ (for example, see \cite[Theorem 4.1]{harris1963theory}) so
$\P{\height(T_\mu)\geq m} = 1 - f_{m}(0)$. Setting $u_k = 1- f_k(0)$, we write
\begin{align*}
f_k(0) &= \frac{1}{2}(1-\alpha+tm^{-1})+\alpha f_{k-1}(0)+\frac{1}{2}(1-\alpha-tm^{-1})f_{k-1}(0)^2\\
&= \frac{1}{2}(1-\alpha+tm^{-1})+\alpha (1 - u_{k-1})+\frac{1}{2}(1-\alpha-tm^{-1})(1- u_{k-1})^2\\
&= 1-(1-tm^{-1})u_{k-1} +\frac{1}{2}(1-\alpha-tm^{-1})u_{k-1}^2,
\end{align*}
to obtain the recursion
\[u_k = (1-tm^{-1})u_{k-1} -\frac{1}{2}(1-\alpha-tm^{-1})u_{k-1}^2.\]
To estimate $u_{k}$, set
\begin{align*}
v_{k} &= (1-tm^{-1})^{-k}u_k\\
&=(1-tm^{-1})^{-k}\left( (1-tm^{-1})u_{k-1} -\frac{1}{2}(1-\alpha-tm^{-1})u_{k-1}^2\right)\\ 
&= v_{k-1} -\frac{1}{2}(1-\alpha-tm^{-1})(1-tm^{-1})^{k-2}v_{k-1}^2,
\end{align*}
and consider the difference
\begin{align*}
\frac{1}{v_{m}}-\frac{1}{v_0}  &= \sum_{k=1}^{m} \pran{\frac{1}{v_k}-\frac{1}{v_{k-1}}}\\
&= \sum_{k=1}^{m}\frac{v_{k-1}-\left( v_{k-1} -\frac{1}{2}(1-\alpha-tm^{-1})(1-tm^{-1})^{k-2}v_{k-1}^2\right)}{v_kv_{k-1}}\\
&=\sum_{k=1}^{m}\frac{\frac{1}{2}(1-\alpha-tm^{-1})(1-tm^{-1})^{k-2}}{1-\frac{1}{2}(1-\alpha-tm^{-1})(1-tm^{-1})^{k-2}v_{k-1}}\\
&=\frac{1}{2}\cdot\frac{1-\alpha-tm^{-1}}{1-tm^{-1}}\sum_{k=0}^{m-1}\frac{(1-tm^{-1})^{k}}{1-\frac{1}{2}(1-\alpha-tm^{-1})(1-tm^{-1})^{k-1}v_{k}}. \tageq\label{eq:diff}
\end{align*}

Furthermore, observing that $x:=\frac{1}{2}(1-\alpha-tm^{-1})(1-tm^{-1})^{k-1}v_{k}\in(0,1)$ under \eqref{eq:m_big}, applying $1/(1-z)\geq 1$ for $z \in (0,1)$, and using $1-z\leq e^{-z}$ for $z\geq-1$, we bound \eqref{eq:diff} from below by
\begin{align*}
\frac{1}{2}\cdot\frac{1-\alpha-tm^{-1}}{1-tm^{-1}}\sum_{k=0}^{m-1}(1-tm^{-1})^{k}&=\frac{1}{2}\cdot\frac{1-\alpha-tm^{-1}}{1-tm^{-1}}\cdot\frac{1-(1-tm^{-1})^{m}}{tm^{-1}}\\
&\geq\frac{1}{2}\cdot\frac{1-\alpha-tm^{-1}}{1-tm^{-1}}\cdot\frac{1-e^{-t}}{tm^{-1}}.
\end{align*}

Noting that $v_0= u_0=1$, then
\begin{align*}
\P{\height(T_\mu)>m}&=u_{m} 
=
\pran{1-tm^{-1}}^{m}\cdot v_{m}\\
&\leq\pran{1-tm^{-1}}^{m}\cdot\frac{1}{1+\frac{1}{2}\cdot\frac{1-\alpha-tm^{-1}}{1-tm^{-1}}\cdot\frac{1-e^{-t}}{tm^{-1}}}\\
&=\frac{1}{m}\cdot\frac{e^{-t}2t}{2tm^{-1}+\frac{(1-\alpha-tm^{-1})}{1-tm^{-1}}(1-e^{-t})}.
\end{align*}
Hence,
\begin{align*}
\limsup_{n\rightarrow \infty} m\P{\height(T_\mu)>m} 
&\leq \limsup_{n\rightarrow\infty} \frac{e^{-t}2t}{2tm^{-1}+\frac{(1-\alpha-tm^{-1})}{1-tm^{-1}}(1-e^{-t})}\\
& = \frac{e^{-t}2t}{(1-\alpha)(1-e^{-t})},
\end{align*}
and so,
\[\lim_{t\rightarrow\infty}\limsup_{n\rightarrow\infty}m\P{\height( T_\mu)>m} =\lim_{t\rightarrow\infty}\frac{e^{-t}2t}{(1-\alpha)(1-e^{-t})} =0.\]
\end{proof}

The key output of this subsection towards the proof of \cref{lem:prop_tall_trees} is a tail bound controlling the height and size of $\rforest_\bdstat^{\uparrow V}$. Equipped with the previous two lemmas we are now prepared to conclude the subsection by obtaining the desired bound.

\begin{lem}\label{lem:height_binary_ds}
Let $\gamma_n\rightarrow\infty$ so that $\gamma_n=O(\sqrt{n})$. Fix $\alpha\in[0,1)$ bounded away from $1$ and let $R>0$. Let $\bdstat$ be a binary degree statistic with size $n(1-\alpha)$ and $\frac{Rn}{\gamma_n}$ components. For $\rforest_\bdstat\in_u\dsForest(\bdstat)$ and uniformly chosen vertex $V$, we have:
\begin{enumerate}[(i)]
\item Let $\eps,\delta>0$, then for sufficiently large $R$ we have 
\[\P{\height(\rforest_\bdstat^{\uparrow V})>\delta\gamma_n}<\frac{\eps}{\delta}\gamma_n^{-1},\]
for all $n$ large enough.
\item Let $\eps>0$. There exists constant $c>0$ so that for large enough $y$ (and in the case $\gamma_n=\sqrt{n}$, we additionally, independent of $n$, need $R$ sufficiently large), 
    \[\P{\#\rforest_\bdstat^{\uparrow V}\geq y}<cy^{-1/2} + \eps n^{-1/2},\]
for all $n$ large enough.
\end{enumerate}
\end{lem}
\begin{proof} 
Let $\mu$ be the offspring distribution 
\[\mu= \frac{1}{2}\left(1+\frac{R}{(1-\alpha)\gamma_n}\right)\delta_0 + \frac{1}{2}\left(1-\frac{R}{(1-\alpha)\gamma_n}\right)\delta_2.\] 
Let $\bdstat=(\bdstat_i)_{i\geq0}$ be a binary degree statistic of size $n(1-\alpha)$ corresponding to $\frac{Rn}{\gamma_n}$ roots. From the observation made in \eqref{eq:uniform_binary}, we may sample $\rforest_\bdstat\in_u\dsForest(\bdstat)$ by sampling a \L ukasiewicz path $W=(W_0, W_1,\ldots, W_{n(1-\alpha)})$ of a $\mu$-Bienaym\'e forest conditioned to have size $n(1-\alpha)$ and $\frac{Rn}{\gamma_n}$ components. Let $V$ be a uniform vertex in $\rforest_\bdstat$, then the subtree rooted at $\rforest_\bdstat^{\uparrow V}$ is encoded by the first excursion of a uniform cyclic shift of $W$, that is, for a uniformly chosen time $I\in_{u}[n(1-\alpha)]$, $\rforest_\bdstat^{\uparrow V}$ is encoded by the first excursion of $W^{(I)}$. Denote this excursion by $e$. 

Take $Y=(Y_k)_{k\geq0}$ to be a simple random walk where $X_0=0$ and the steps are independently distributed as $K-1$ where $K\sim\mu$. Let $e'$ be the first excursion of $Y$ above its running infimum. Then $e'$ encodes a $\mu$-Bienaym\'e tree $T_{\mu}$.

By \cref{lem:bridge_shift},
    \[W^{(I)} \overset{d}{=} (Y_0,Y_1,\ldots, Y_{n(1-\alpha)})\ \bigg|\ Y_{n(1-\alpha)}=-\frac{Rn}{\gamma_n}.\]
Both parts of the proof will rely on the following change of measure. We compare the behaviour of $\rforest_\bdstat^{\uparrow V}$ to an unconditioned Bienaym\'e tree by taking $f$ a bounded measurable function with respect to $Y$, and estimating $\E{f(e)},\ \E{f(e')}$. First observe,
\begin{align*}
    \E{f(e)}
    &=
    \e\left[f(e') \ \Big| \ Y_{n(1-\alpha)}=-\frac{Rn}{\gamma_n}\right]
    =
    \E{f(e')\cdot \frac{\I{Y_{n(1-\alpha)}=-\frac{Rn}{\gamma_n}}}{\P{Y_{n(1-\alpha)}=-\frac{Rn}{\gamma_n}}}}\\
    &=
    \E{f(e')\frac{\probC{Y_{n(1-\alpha)}=-\frac{Rn}{\gamma_n}}{Y_1,\ldots,Y_{|e'|}}}{\P{Y_{n(1-\alpha)}=-\frac{Rn}{\gamma_n}}}}.
\end{align*}
Now, given the length of the first excursion $e'$, then 
\[\Cprob{Y_{n(1-\alpha)}=-\frac{Rn}{\gamma_n}}{Y_1,\ldots,Y_{|e'|}}=\P{Y_{n(1-\alpha)-|e'|}=-\frac{Rn}{\gamma_n}+1}.\] 
We condition on whether $|e'|\leq n(1-\alpha)/2$, and in particular, \cref{lem:meas_change_const} yields a constant $C>0$ so that for all $n$ large enough, all $\ell\leq n(1-\alpha)/2$

$$
\frac{\P{Y_{n(1-\alpha)-\ell}=-\frac{Rn}{\gamma_n}+1}}{\P{Y_{n(1-\alpha)}=-\frac{Rn}{\gamma_n}}}
\leq C.
$$
Therefore,

\begin{align*}
    \E{f(e)}&=\E{f(e)\I{|e|\leq \frac{n(1-\alpha)}{2}}} + \E{f(e)\I{|e|>\frac{n(1-\alpha)}{2}}}\\
    &\leq C\E{f(e')} + \|f\|_\infty\P{|e|>\frac{n(1-\alpha)}{2}}.\tageq\label{eq:measure_change}
\end{align*}
We use \eqref{eq:measure_change} to prove the two statements of the lemma.

For part (i), let $\eps,\delta>0$ and take $f(\cdot)$ to be the indicator function that the inputted excursion encodes a tree with height bigger than $\delta\gamma_n$. Note that excursions $e,\ e'$ encode the respective trees $\rforest_\bdstat^{\uparrow V},\ T_\mu$. Then for all $n$ large enough, \eqref{eq:measure_change} becomes
\begin{align*}
\P{\height(\rforest_\bdstat^{\uparrow V})>\delta\gamma_n} &\leq C\P{\height(T_\mu)>\delta\gamma_n} +1\cdot \probC{\tau_{-1}(Y)>\frac{n(1-\alpha)}{2}}{Y_{n(1-\alpha)}=-\frac{Rn}{\gamma_n}}\\
&< \frac{\eps}{2\delta}\gamma_n^{-1} + \frac{\eps}{2\delta} n^{-1/2}<\frac{\eps}{\delta}\gamma_n^{-1},
\end{align*}
for $R$ and $n$ large enough by \cref{lem:sub_critical_height_profile} and \cref{lem:first_exc_cbridge}, where the conditioning allows us to apply the latter.

Now for part (ii), let $\eps>0$ and set $R$ to satisfy \cref{lem:first_exc_cbridge}. Consider the offspring distribution $\nu = \frac{1}{2}\delta_0 + \frac{1}{2}\delta_2$. Then $\mu\preceq_{st}\nu$, so there exists a coupling $(A,B)$ so that $A\sim \mu$ and $B\sim\nu$ and $A\leq B$ almost surely. We can simultaneously construct Bienaym\'e trees $T_\mu$ and $T_\nu$ by taking independent copies of $(A,B)$ and assign it as the respective offspring counts for every potential vertex in the Ulam--Harris tree. By induction on the generations, every vertex present in $T_\mu$ also appears in $ T_\nu$, so $T_\mu\subseteq T_\nu$ almost surely. Thus, we get stochastic domination on the sizes of the trees, $\P{\#T_\mu\geq y}\leq \P{\#T_\nu\geq y}$ for all $y\in\R$.

Let $f(\cdot)$ be the indicator function that the inputted excursion encodes a tree with size at least $y$. Again, taking $e,\ e'$ which encode trees $\rforest_\bdstat^{\uparrow V},\ T_\mu$, then \eqref{eq:measure_change} becomes
\begin{align*}
\P{\#\rforest_\bdstat^{\uparrow V}\geq y} &\leq C\P{\#T_\mu\geq y} +1\cdot \probC{\tau_{-1}(Y)>\frac{n(1-\alpha)}{2}}{Y_{n(1-\alpha)}=-\frac{Rn}{\gamma_n}}\\
&\leq C\P{\#T_\nu\geq y} + \eps n^{-1/2}\\
&< C'y^{-\frac{1}{2}} + \eps n^{-1/2},
\end{align*}
where we go from the first to the second line using stochastic domination and \cref{lem:first_exc_cbridge}, which can be applied due to the conditioning. The last line follows from \cite[Proposition 2.3]{angel2025largest} for large enough $y$. 

\end{proof}

\subsection{Transferring bounds and the proof of \cref{lem:prop_tall_trees}}\label{sec:transfer_bounds}

The next lemma allows us to transfer the tail bounds for uniform forests with a fixed binary degree statistic to uniform forests with fixed sub-binary degree statistic.
\begin{lem}\label{lem:height_subbinary_ds}
    Fix a sub-binary degree statistic $\dstat = (\dstat_0, \dstat_1, \dstat_2, 0, 0,\ldots)$ with size $n$ that corresponds to $k$ components. Further assume there is a $\beta<1$ so that $\dstat_1\leq \beta n$. Set $\bdstat=(\dstat_0,0,\dstat_2,0,0,\ldots)$ to be the binary degree statistic obtained from $\dstat$ by replacing $\dstat_1$ with 0. Then, for $\rforest_\dstat\in_u\dsForest(\dstat),\ \rforest_\bdstat\in_u\dsForest(\bdstat)$ and respectively chosen uniform vertices $V, V'$ we have for $h>0$,
    \begin{align*}
        \P{\height\left( \rforest_\dstat^{\uparrow V}\right)> h}
        < & \P{\height\left(\rforest_\bdstat^{\uparrow V'}\right)>\frac{1-\beta}{4}h}
        +
        \P{\# \rforest_\bdstat^{\uparrow V'}>h^3}\\
        &+
        h^3\exp\set{-\frac{h(1-\beta)}{8}}.
    \end{align*}
\end{lem}
\begin{proof}
We closely follow the argument in \cite[Lemma 15]{height}. Consider the following sets of plane forests with a marked vertex:
\[\mathcal{S}=\set{(\rf ,v): \rf\in\dsForest(\dstat), v\in \rf} \text{ and } \mathcal{R}=\set{(\rf' ,v'): \rf\in\dsForest(\bdstat), v'\in \rf'},\]
and the set of weak compositions of $\dstat_1$ into $n-\dstat_1+1$ parts:
    \[\mathcal{W} =\set{w=(w_1, \ldots, w_{n-{\dstat_1}+1}) ~: w_i\in\N_0, \sum_{i=1}^{n-\dstat_1+1}=\dstat_1}.\]
We show there is a bijection between $\mathcal{S}$ and $\mathcal{R}\times\mathcal{W}$.

Given a plane forest with a marked vertex $(\rf, v)\in\mathcal{S}$, we can obtain a weak composition $w_{(\rf,v)}=(w_1, \ldots, w_{n-\dstat_1+1})$ and a plane forest with a marked vertex $(\rf^-, v^-)$ from $(\rf, v)$ by \emph{collapsing} all out-degree-one vertices as follows.
Let $v_1$ be the first weak descendant of $v$ (i.e a descendent of $v$ or $v$ itself) whose out-degree is not one, and fix an ordering $v_1\prec v_2\prec\cdots\prec v_{n-\dstat_1}$ of the non-out-degree-one vertices so that $v_2\prec\cdots\prec v_{n-\dstat_1}$ respects the lexicographic order in $\rf$. 

To obtain $w_{(\rf,v)} = (w_1,\ldots, w_{n-\dstat_1+1})$, let $w_1=d_\rf(v, v_1)$ and take $w_{n-\dstat_1+1}$ to be the length of the longest path from $v$ to an ancestor so that all vertices on this path have out-degree-one, expect possibly $v$. For $i\in\set{2,\ldots, n-\dstat_1}$, then $w_i$ is the length of the longest path from $v_i$ to an ancestor that, after the first vertex, only contains out-degree-one vertices. To obtain the marked collapsed tree $(\rf^-,v^-)$, mark $v^-$ to be $v_1$. Next, obtain 
 $\rf^-$ by identifying every out-degree-one vertex with its first descendant that does not have out-degree one. This results in $(\rf^-, v^-, w_{(\rf,v)})$. See \cref{fig:collapse} for an example. 
 
\begin{figure}
    \centering
\includegraphics[width=.7\linewidth]{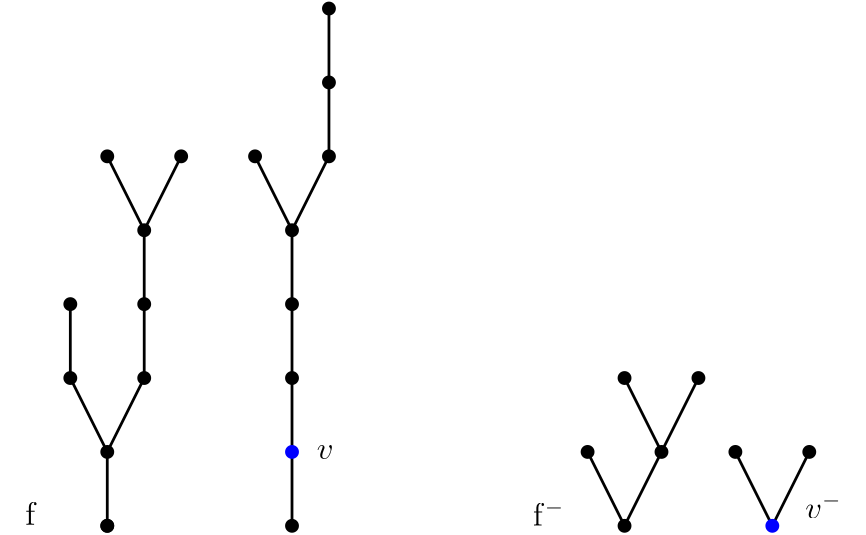}
    \caption{Left: A marked plane forest $(\rf, v)$. Right: The marked plane forest $(\rf^-, v^-)$ obtained from $(\rf, v)$. The marked vertex $v^-$ is the first non-degree-one weak descendant of $v$, the forest $\rf^-$ is obtained from $\rf$ by collapsing all degree-one vertices, and the weak composition is $w_{(\rf,v)}=(3,1,1,2,0,0,0,2,1)$.}
    \label{fig:collapse}
\end{figure}

It is clear that this is a bijection. Given $(\rf^-, v^-, w_{(\rf,v)})$ we may recover $(\rf, v)$ as follows. Set $v_1=v^-$ and order the remaining vertices $v_2,\ldots, v_{n-\dstat_1}$ in lexicographic order. We use $w_{(\rf,v)}=(w_1, \ldots, w_{n-\dstat_1+1})$ to place out-degree-one vertices. Place $w_1+w_{n-\dstat_1+1}$ vertices on the ancestral edge of $v_1$ and instead place the mark on the $w_1$-th descendant of $v_1$. 
For all other vertices $v_i$, place $w_i$ vertices on the ancestral edge of $w_i$. If any of the vertices $v_i$ is a root, then simply attach the $w_i$ new vertices as ancestors of $v_i$.

Note that sampling uniformly from $\mathcal{S}$ is the same as taking a uniform sample $\rf\in\dsForest(\dstat)$ and then sampling a uniform vertex $v\in\rf$, since every forest in $\dsForest(\dstat)$ has the same number of vertices. Similarly, a uniform sample from $\mathcal{R}$ has the same distribution as uniformly sampling from $\dsForest(\bdstat)$ and then uniformly sampling a vertex.

Now, for a marked forest $(\rf',v')\in\mathcal{R}$, there are exactly $\#\mathcal{W}=\binom{n}{n-\dstat_1}$ marked forests $(\rf,v)\in\mathcal{S}$ so that $(\rf^-,v^-)=(\rf',v')$. It follows that if $(F,V)\in_u\mathcal{S}$, then its collapse $(F^-,V^-)$ is also a uniform sample from $\mathcal{R}$. Furthermore, conditioning on $(F^-,V^-) =(\rf',v')$, then the distribution of $(F,V)$ is described by choosing a uniform composition $w\in\mathcal{W}$ and assigning out-degree-one vertices to ancestral edges as done in the above, which will consequently determine the marked vertex $V$. 

We are interested in bounding the height of $F^{\uparrow V}$, where $(F,V)$ is a uniform sample from $\mathcal{S}$. Let $h\geq 1$ and choose $(\rf',v')\in\mathcal{R}$ so that $\height\left(\rf'^{\uparrow v'} \right) \leq \frac{(1-\beta)h}{4}$ and $\#\rf'^{\uparrow v'}\leq h^3$. Consider a $v'$-to-leaf path $U=(v' = u_1,\ldots, u_m)$ in $\rf'^{\uparrow v'}$, then it follows from the above construction that conditional on $(F^-, V^-) = (\rf',v')$, the total number of vertices lying on the corresponding $V$-to-leaf path in $F^{\uparrow V}$ is determined by how many vertices get placed between $u_1$ and $V$ and on the ancestral edges of  $u_2, \ldots, u_m$. Thus, the number of vertices added to the corresponding $V$-to-leaf path in $F^{\uparrow V}$ is distributed as the sum of the $m$ parts of a uniformly sampled weak composition $W\in\mathcal{W}$. Note that we may take the first $m$ parts of $W$ by exchangeability of the $W_i$s. Thus, the total number of vertices on the corresponding $V$-to-leaf path in $F^{\uparrow V}$ has the same distribution as the random variable $X_m$ described in the following experiment: consider an urn with $n-\dstat_1$ white balls and $\dstat_1$ black balls. Repeatedly sample without replacement from the urn until it is empty. For $0\leq i\leq n-1$, write $Y_i$ for the total number of white balls drawn after the $i$-th sample, and set $Y_{n}=n-\dstat_1$. Then let $X_m=\min\set{i~:~Y_i=m}$. 

We also introduce the with replacement counterpart of $Y_i$, where for $0\leq i\leq n$ let $Y_i^*\sim \text{Binomial}(i, \frac{n-\dstat_1}{n})$, and note that $Y_i$ and $Y_i^*$ observe the same mean. Furthermore by \cite[Proposition 20.6]{MR883646}, $Y_i^*$ is a dilation of $Y_i$, which consequently implies that $\E{f(Y_i)}\leq \#{f(Y_i^*)}$ for convex functions $f:\R\rightarrow\R$. Particularly, concentration inequalities obtained via bounding exponential moments for $Y_i^*$ can be applied to $Y_i$. 

Now, for $x\geq 2$ we have
\begin{align*}
\P{X_m>\frac{xmn}{n-\dstat_1}}
&\leq
\P{Y_{\left\lceil \frac{xmn}{n-\dstat_1}\right\rceil}\leq m}
\leq
\P{Y_{\left\lceil \frac{xmn}{n-\dstat_1}\right\rceil}\leq \frac{1}{2}\E{Y_{\left\lceil \frac{xmn}{n-\dstat_1}\right\rceil}}}\\
&\leq\exp\set{-\frac{xm}{8}},
\end{align*}
where the second line follows from $\e\Big[Y_{\left\lceil \frac{xmn}{n-\dstat_1}\right\rceil}\Big] = xm$. The last line uses the fact that $\E{Y_i}=\E{Y_i^*}$, so applying a Chernoff bound \cite[Theorem 2.3(c)]{mcdiarmid98concentration} to \\$\p\Big\{Y^*_{\left\lceil \frac{xmn}{n-\dstat_1}\right\rceil}\leq \frac{1}{2}\e\Big[{Y^*_{\left\lceil \frac{xmn}{n-\dstat_1}\right\rceil}}\Big]\Big\}$ also gives the same bound for $\p\Big\{ Y_{\left\lceil \frac{xmn}{n-\dstat_1}\right\rceil}\leq \frac{1}{2}\e \Big[Y_{\left\lceil \frac{xmn}{n-\dstat_1}\right\rceil}\Big]\Big\}$.

In particular, for $y\geq\frac{2mn}{n-\dstat_1}$, then
$\P{X_m\geq y}\leq \exp\set{-\frac{y(n-\dstat_1)}{8n}}$.
Since the number of vertices on a fixed $v'$-to-leaf path is one larger than the distance between $v'$ and that leaf, then it follows that for any leaf $l$ in $\rf'^{\uparrow v'
}$ and for $y\geq\frac{2(d_{\rf'^{\uparrow v'}}(v', l)+1)n}{n-\dstat_1}$, we have
\begin{align*}
\probC{d_{F^{\uparrow V}}(V, l)\geq y}{(F^-, V^-)= (\rf',v')} &\leq \exp\set{-\frac{(y+1)(n-\dstat_1)}{8n}}\\
&\leq\exp\set{-\frac{y(n-\dstat_1)}{8n}},
\end{align*}
where we abuse notation and use $l$ to denote both the leaf in a $v'$-to-leaf path in $\rf'^{\uparrow v'}$ and the corresponding leaf in $F^{\uparrow V}$.

As we chose $(\rf',v')$ so that $\height\left(\rf'^{\uparrow v'} \right) \leq \frac{(1-\beta)h}{4}$, then setting $y=h$ satisfies
\[\frac{2(d_{\rf'^{\uparrow v'}}(v', l)+1)n}{n-\dstat_1}\leq \frac{4n}{n-\dstat_1}\cdot \frac{(1-\beta)h}{4}\leq  \frac{(1-\beta)hn}{n-\dstat_1}\leq h,\] where we use the assumption $\dstat_1\leq \beta n$. Furthermore, the choice of $(\rf',v')$ satisfies $\#\rf'^{\uparrow v'}\leq h^3$, so a union bound over all leaves gives
\begin{equation}\label{eq:dilation_argument}
    \Cprob{\height\left(F^{\uparrow V}\right)\geq h}{(F^-, V^-)= (\rf',v')} <h^3\exp\set{-\frac{h(n-\dstat_1)}{8n}}.
\end{equation}
Going back to estimating the height of $F^{\uparrow V}$ and applying a union bound on the height and size of its collapse $(F^-, V^-)$, 

\begin{equation}\label{eq:proba_high_fringe_forestdecompo}
\begin{aligned}
    \P{\height\left(F^{\uparrow V}\right)>h}
    \leq& \P{\height\left({F^-}^{\uparrow V^-} \right)>\tfrac{(1-\beta)h}{4}} + \P{\#{F^-}^{\uparrow V^-}>h^3} \\
    &  \hspace{-3cm}+ \probC{\height\left(F^{\uparrow V}\right)>h}{\height\left({F^-}^{\uparrow V^-} \right)\leq\tfrac{(1-\beta)h}{4},\ \#{F^-}^{\uparrow V^-}\leq h^3}.
\end{aligned}
\end{equation}
We let $\cA=\{ (\rf', v'):~\height\left(\rf'^{\uparrow v'}\right)\leq\frac{(1-\beta)h}{4}, \# \rf'^{\uparrow v'}\leq h^3 \}$ and decompose
\begin{equation}
\begin{aligned}
    &\probC{\height\left(F^{\uparrow V}\right)>h}{\height\left({F^-}^{\uparrow V^-} \right)\leq\tfrac{(1-\beta)h}{4},\ \#{F^-}^{\uparrow V^-}\leq h^3}\\
    &=
    \sum_{(f',v')\in \cA}  \Cprob{\height\left(F^{\uparrow V}\right)\geq h}{(F^-, V^-)= (\rf',v')}
    \cdot
    \P{(F^-,V^-)=(\rf',v')}.
\end{aligned}
\end{equation}
Bounding the first probability in the sum with \eqref{eq:dilation_argument} and plugging the obtained control into \eqref{eq:proba_high_fringe_forestdecompo} yields
\begin{align*}
    \P{\height\left(F^{\uparrow V}\right)>h}
    \leq &
    \P{\height\left({F'}^{\uparrow V'} \right)>\frac{(1-\beta)h}{4}} + \P{\#{F'}^{\uparrow V'}>h^3}\\
    &+ h^3\exp\set{-\frac{h(n-\dstat_1)}{8n}}.
\end{align*}
The assumption $\dstat_1\leq\beta n$ concludes the proof of the lemma.
\end{proof}


We can now prove a non-uniform-in-$\gamma_n$ version of \cref{lem:prop_tall_trees}, that is:
\begin{lem}\label{fact:non_uniform}
    Let $1\ll\gamma_n\leq n^{1/2}$. For any $\eps,\delta>0$, there exists $R$ so that for $\lambda>\frac{R^2}{\sigma^2}$,  then 
    \[\P{\#\set{v\in \fforest:\height\left({\fforest}^{\uparrow v}\right)>\frac{\eps\gamma_n}{2}}>\frac{\eps^2n}{2\gamma_n}} <\delta,\]
    for sufficiently large $n$.
\end{lem}

\begin{proof}
For brevity let $\cA= \left\{ \#\left\{v\in\fforest~:~ \height\big({\fforest}^{\uparrow v}\big)>\frac{\eps\gamma_n}{2}\right\}>\frac{\eps^2n}{2\gamma_n} \right\}$. We aim to upper bound $\P{\cA}$.

Let $\eps,
\delta>0$. Take $R$ a large enough constant to satisfy the forthcoming applications of \cref{lem:height_binary_ds}, and let $\lambda>\frac{R^2}{2\sigma^2}$ with $\fforest$ being the corresponding fringe forest.

Let $N'$ be the following collection of sub-binary degree statistics:
\[N'= \left\{\begin{aligned}
\mathbf m'=(\mathbf m_i')_{i\geq 0}~:~& \sum_{i\geq 0}\mathbf m'_i=n(1-\alpha_n) =r'+\sum_{i\geq0}i\mathbf m_i,\\ &\mathbf m_i'=0 \text{ for }i\geq 3,\ \mathbf m_1'\leq \beta n(1-\alpha_n),\ r'\geq \frac{Rn}{\gamma_n}\end{aligned}\right\},\]
where $\beta$ is the constant in \cref{lem:bounded_deg1} and  $\alpha_n\in(0,1)$ is some small constant.

Let $\dstat$ be the random degree statistic of $\fforest$ and throughout this proof, fix the size of $\dstat$ to be $n(1-\alpha_n)$ and denote the number of components by $r_n$. Taking $\rforest_\dstat\in_u\dsForest(\dstat)$ and $\rforest_{\mathbf m'}\in_u\dsForest(\mathbf m')$, then,

\begin{equation}
    \P{\cA}
    =
    \P{\#\set{v\in\rforest_\dstat~:~ \height\left({\rforest_\dstat}^{\uparrow v}\right)>\frac{\eps\gamma_n}{2}}>\frac{\eps^2n}{2\gamma_n}}.
\end{equation}
Decomposing on the event $\set{\exists \mathbf m'\in N' : \dstat\preceq \mathbf m'}$ and its complement, we get
\begin{equation}\label{eq:split_terms_compact}
\begin{aligned}
    \P{\cA}
    \leq &
    \max_{\mathbf m'\in N'}\P{\#\set{v\in\rforest_\dstat~:~ \height\left({\rforest_\dstat}^{\uparrow v}\right)>\tfrac{\eps\gamma_n}{2}}>\tfrac{\eps^2n}{2\gamma_n},\ \dstat\preceq\mathbf m'}\\
    &+
    \P{\not\exists \mathbf m'\in N' : \dstat\preceq \mathbf m'}.
\end{aligned}
\end{equation}

We note that the second term of equation \eqref{eq:split_terms_compact} can be controlled by the union bound of when $\dstat_1$ is too large, when $\dstat$ corresponds to too few components, and when $\fforest$ is too small,
\[\set{\not\exists \mathbf m'\in N' : \dstat\preceq \mathbf m'}\subseteq\set{\dstat_1>\beta n(1-\alpha_n)}\cup\set{r_n < \frac{Rn}{\gamma_n}}\cup\set{\#\fforest<n(1-\alpha_n)}.\]
Or equivalently,
\begin{align*}
    \left\{\exists \mathbf m'\in N' : \dstat\preceq \mathbf m' \right\}
    \supseteq 
    \set{\dstat_1\leq \beta n(1-\alpha_n)}\cap\set{r_n\geq \frac{Rn}{\gamma_n}}\cap\set{\#\fforest\geq n(1-\alpha_n)}.
\end{align*}

Indeed, suppose $\dstat$ satisfies all three conditions. We build $\dstat'\in N'$ with $\dstat\preceq \dstat'$ as follows: by \cref{lem:degree_order} we obtain a sub-binary degree statistic $\dstat'$ which succeeds $\dstat$ in the partial order. From the construction of $\dstat'$, we have that $\dstat'$ has the same size as $\dstat$ and corresponds to $r_n$ components. Furthermore, \cref{lem:bounded_deg1} shows that $\dstat'_1\leq \beta\#\fforest$, and thus, $\dstat'\in N'$. Now bounding the second term of \eqref{eq:split_terms_compact},

\begin{align*}
\P{\not\exists \mathbf m'\in N' : \dstat\preceq \mathbf m'} 
\leq & 
\P{\dstat_1 > \beta n(1-\alpha_n)}
+
\P{r_n < \frac{Rn}{\gamma_n}}\\
&+
\P{\#\fforest<n(1-\alpha_n)}
<
\frac{\delta}{2},\tageq\label{eq:bad_deg_stat_compact}
\end{align*}
where, for large enough $n$, the inequality follows from \cref{lem:bounded_deg1} and \eqref{eq:roots_fringe}. 

To control the first term on the right-hand-side of \eqref{eq:split_terms_compact}, fix some $\mathbf{m}'\in N'$ so that $\dstat\preceq \mathbf{m}'$ and for brevity let
\begin{equation}\label{eq:def_abbreviation_pm}
    p(\mathbf m')
    =
    \P{\#\set{v\in\rforest_\dstat~:~ \height\left({\rforest_\dstat}^{\uparrow v}\right)>\tfrac{\eps\gamma_n}{2}}>\tfrac{\eps^2n}{2\gamma_n},\ \dstat\preceq\mathbf m'}.
\end{equation}
We aim to bound $p(\mathbf m')$ uniformly in $\mathbf m' \in N'$. By Markov's inequality
\begin{align*}
    p(\mathbf m')
    \leq
    \frac{2\gamma_n}{\eps^2n}\E{\I{\dstat\preceq\mathbf m'}\cdot\#\set{v\in\rforest_\dstat~:~ \height \left( {\rforest_\dstat }^{\uparrow v} \right) > \frac{\eps\gamma_n}{2}}}.
\end{align*} \cref{lem:more_tall_trees_when_subbinary} states that
\begin{align*}
    p(\mathbf m') 
    \leq &
    \frac{2\gamma_n}{\eps^2n}\E{\#\set{v\in\rforest_{\mathbf m'}~:~ \height\left({\rforest_{\mathbf m'}}^{\uparrow v}\right)>\frac{\eps\gamma_n}{2}}}\\
    =&
    \frac{2\gamma_n}{\eps^2n}\sum_{v\in\dsForest_{\mathbf m'}}\P{\height\left({\rforest_{\mathbf m'}}^{\uparrow v}\right)>\frac{\eps\gamma_n}{2}}.
\end{align*}

Given $\mathbf m'$, take $\mathbf b = \mathbf m'-(0,\mathbf m_1', 0, 0,\ldots)$ be the binary degree statistic after removing all out-degree one vertices. Then for $V'$ a uniformly chosen vertex from $\rforest_{\bdstat}$, applying \cref{lem:height_subbinary_ds} gives:

\begin{align*}
p(\mathbf m')
< &
 \frac{2\gamma_n}{\eps^2}
\P{\height\left({\rforest_{\mathbf b}}^{\uparrow V'}\right)>\frac{(1-\beta)\eps\gamma_n}{8}} \\
&+
\frac{2\gamma_n}{\eps^2}\Bigg(\P{{\#\rforest_{\mathbf b}}^{\uparrow V'}>\frac{\eps^3\gamma_n^3}{8}} + \frac{\eps^3\gamma_n^3}{8}\exp\set{-\frac{(1-\beta)\eps\gamma_n}{16}}
\Bigg).
\end{align*}
Note that by construction of $N'$ and the fact that $\alpha_n,\beta \in(0,1)$, then $\bdstat$ has size $n(1-\alpha')$ for some $\alpha'\in(0,1)$. Additionally, $\bdstat$ corresponds to $\frac{R'n}{\gamma_n}$ components for some $R'\geq R$. Thus, applying \cref{lem:height_binary_ds} twice yields for any $\xi,\xi'>0$
\begin{align*}
    p(\mathbf m')
    <&
    \frac{2\gamma_n}{\eps^2}\left( \frac{8\xi}{(1-\beta)\eps\gamma_n} + c\sqrt{\frac{8}{\eps^3\gamma_n^3}} + \frac{\xi'}{\sqrt{n}}  \right)
    +
    \frac{\eps\gamma_n^4}{4}\exp\set{-\frac{(1-\beta)\eps\gamma_n}{16}}\\
    =&
    \frac{16\xi}{(1-\beta)\eps^3}
    +
    \frac{4\sqrt{2}c}{\sqrt{\eps^7 \gamma_n}}
    +
    \frac{2\xi'\gamma_n}{\eps^2\sqrt{n}}
    +
    \frac{\eps\gamma_n^4}{4}\exp\set{-\frac{(1-\beta)\eps\gamma_n}{16}}    
\end{align*}
choosing $\xi=\frac{\delta(1-\beta)\eps^3}{128}$, $\xi'=\frac{\delta\eps^2}{16}$, and recalling that $\gamma_n\to\infty$ but $\gamma_n\leq \sqrt{n}$ we obtain, for $n$ large enough,
\begin{equation}\label{eq:p_m_prime_final_bound}
    p(\mathbf m') \leq \frac{\delta}{2}.
\end{equation}
Combining \eqref{eq:bad_deg_stat_compact}, \eqref{eq:def_abbreviation_pm} and \eqref{eq:p_m_prime_final_bound} into \eqref{eq:split_terms_compact} concludes the proof of \cref{fact:non_uniform}.
\end{proof}

We are now ready to prove \cref{lem:prop_tall_trees}.

\begin{proof}[Proof of \cref{lem:prop_tall_trees}]
It only remains to show that in \cref{fact:non_uniform}, $R$ can be chosen independently of the sequence $(\gamma_n)_n$. Fix $\varepsilon,\delta>0$. For $0<g\le\sqrt n$, write $F_{n,g}^{\lambda}$ for the fringe forest when the skeleton is spanned by the root and $\lfloor \lambda n/g^2\rfloor$ uniform leaves, and set
$$
p_{n,g}(\lambda)
:=
\P{
\#\left\{v\in F_{n,g}^{\lambda}:
\operatorname{ht}\bigl((F_{n,g}^{\lambda})^{\uparrow v}\bigr)
>\frac{\varepsilon g}{2}\right\}
>\frac{\varepsilon^2 n}{2g}}.
$$
With this notation, $p_{n,\gamma_n}(\lambda)$ is the quantity to bound to prove \cref{lem:prop_tall_trees}. Coupling the skeletons using the same sequence of sampled leaves shows that they are increasing in $\lambda$. So the fringe forests are decreasing, and thus $p_{n,g}(\lambda)$ is non-increasing in $\lambda$. We claim that
\begin{equation}\label{eq:lem_sec5_unif}
    \exists \ \Lambda,\Gamma>0, \ N\geq 1, \ s.t \ 
    \forall \ n\geq N,
    \ \Gamma \leq g \leq \sqrt{n}
    \quad 
    p_{n,g}(\Lambda)<\delta.
\end{equation}
Indeed, if \eqref{eq:lem_sec5_unif} did not hold, then we could choose strictly increasing integers $n_j$ and numbers $g_j$ satisfying
$$
j\le g_j\le\sqrt{n_j},
\qquad
p_{n_j,g_j}(j)\ge\delta.
$$
Define
$$
\gamma_n'=
\begin{cases}
g_j, & n=n_j \text{ for some }j,\\
\sqrt n, & \text{otherwise}.
\end{cases}
$$
Then $\gamma_n'\to\infty$ and $\gamma_n'\le\sqrt n$. Applying \cref{fact:non_uniform} to this sequence, with $\delta/2$ in place of $\delta$, gives a fixed $\lambda_0>0$ such that
$$
p_{n,\gamma_n'}(\lambda_0)<\delta/2
$$
for all sufficiently large $n$. For all sufficiently large $j$, monotonicity therefore yields
$$
\delta
\le p_{n_j,g_j}(j)
\le p_{n_j,g_j}(\lambda_0)
<\delta/2,
$$
a contradiction, proving \eqref{eq:lem_sec5_unif}. It remains to choose $R=\sigma\sqrt \Lambda$. For every $\lambda>R^2/\sigma^2$ and every sequence $\gamma_n\to\infty$ with $\gamma_n\le\sqrt n$, we have $n\ge N$ and $\gamma_n\ge \Gamma$ for all sufficiently large $n$, and hence
$$
p_{n,\gamma_n}(\lambda)
\le p_{n,\gamma_n}(\Lambda)<\delta.
$$
Thus $R$ can be chosen independently of the sequence $(\gamma_n)_n$, concluding the proof of the lemma.
\end{proof}

\section{Local convergence of the skeleton}\label{sec:local_skeleton}

This section is dedicated to proving  \cref{prop:skeleton_local_limit}, which we recall below. 

\skelLocalLim*

The very first step to prove \cref{prop:skeleton_local_limit} is to study the asymptotics of the size and degree statistic of $\skel$. Let $(\dstat_0(\skel),\dstat_1(\skel), \dots)$ be the degree statistic of $\skel$ when viewed as a tree rooted at the root of $T_n$.  \cref{lem:size_skel}, that was stated without proof in the previous section, studies these quantities. 
 


\sizeskel*

To prove \cref{prop:skeleton_local_limit} we first explicitly provide the distribution of $\skelr$, and then we prove convergence of its underlying combinatorial tree as well as its edge lengths. Finally, we prove that this implies convergence in the topology defined in \cref{subsec:def_marked_local_topology}. We recall that in this paper, by tree we mean plane tree.

\begin{lem}\label{lem:weak_limit_in_steps}
    Conditioned on having degree statistic $\dstat$, $\skel$ is a uniform tree with said degree statistic. Moreover, $\skelr$ is a tree with edge lengths given by a uniform weak composition of $\dstat_1$ into $-1+\dstat_0+\dstat_2+\dots$ and an underlying combinatorial tree that is uniform among trees with degree statistic $(\dstat_0,0,\dstat_2,\dots)$, independent of the weak composition.
\end{lem}

\begin{proof}
    It is well known that that for a Bienaymé tree $T$ and a fixed  tree $t$, $\P{T=t}$ depends only on $t$ through its degree statistic. This property extends to the subtree spanned by the root and a fixed number of leaves picked uniformly at random in a size-conditioned Bienaymé tree. We give a short informal proof of this fact and point the interested reader to \cite[Section 2.2]{duquesne2007growth} for the corresponding reconstruction procedure. Fix two degree statistics $\dstat$ and $\dstat'$. For a  rooted tree $t$ of degree statistic $\dstat'$, the number of ways to complete it by adding edges and vertices up to attaining degree statistic $\dstat$ depends on $t$ only though $\dstat'$. This proves our first statement.

    Further, the way to add $\dstat_1$ out-degree $1$ vertices to a tree $t$ with degree statistics $(\dstat_0,0,\dstat_2,\dots)$ is the number of weak compositions of $\dstat_1$ into the number of edges $-1+\dstat_0+\dstat_2+\dots$. As this number depends on $t$ only through its degree statistic, our second statement is proven.

\end{proof}


Now, we study convergence of the underlying combinatorial tree of $\skelr$. We proved that given the degree statistic $\dstat$ of $\skel$, the underlying combinatorial tree of $\skelr$ is an uniform tree of degree statistic $(\dstat_0,0,\dstat_2,\dots)$. We prove weak local convergence of this combinatorial tree under some assumptions on $(\dstat_0,0,\dstat_2,\dots)$, which follows from \cref{lem:size_skel}.

\begin{lem}\label{lem:loc_convergence_shape}
    Fix $k\geq0$ and any sequence of degree statistics $\big(\dstat^{(n)}\big)_{n\geq 1}$, where $\dstat^{(n)}:=\big(\dstat_i^{(n)}\big)_{i\geq 0}$ such that 

    \begin{equation}\label{eq:assumption_local_limit}
        |\dstat^{(n)}|\to \infty; 
        \quad \frac{\dstat_0^{(n)}}{\dstat_2^{(n)}}\xrightarrow{P} 1; 
        \quad
        \frac{|\dstat^{(n)}|}{\dstat_0^{(n)}}\xrightarrow{P} 2.
    \end{equation}


    \noindent Let $T_{\dstat}$ be a rooted tree sampled uniformly at random from degree statistic $\dstat$ and $V$ a vertex chosen uniformly at random in $T_{\dstat}$. 
    Also let $(K,P,\rho)$ be a rooted binary-Kesten tree, where $P$ is its infinite distinguished path. 
    Then, for any tree $t$

    \begin{equation}\label{eq:local_conv_goal}
    \P{F_k(V,T_{\dstat})=t\mid T_{\dstat}}
    \xrightarrow{P}
    \P{F_k(\rho,K)=t}.
    \end{equation}
\end{lem}

\begin{proof}


We call trees distributed as $F_k(\rho,K)$ Bienaymé $k$-fringes with offspring distribution $\mu=\tfrac{1}{2}\delta_0+\tfrac{1}{2}\delta_2$, and denote by $B$ such a random tree.  
Then, for $t$ a fixed binary $k$-fringe
    
\begin{equation}\label{eq:proba_k-rooted_bienamye}
    \P{B=t}=2^{-\#t}.
\end{equation}
Indeed, if $B=t$ the degrees of the $\#t-k+1$ non-distinguished vertices match and for the $k-1$ distinguished vertices, the choice of left or right for the next distinguished vertex match. Then, we introduce
    
$$N\big(t,T_{\dstat}\big)=\# \{v \in T_{\dstat}, \ F_k(v,T_{\dstat})=t \},$$
the number of vertices of $T_{\dstat}$ having $t$ as a $k$-fringe. Doing so,

\begin{equation}\label{eq:proba_fringe_quenched}
    \P{F_k(V,T_{\dstat})=t\mid T_{\dstat}}=\frac{N\big(t,T_{\dstat}\big)}{|\dstat|}
\end{equation}
and

\begin{equation}\label{eq:proba_fringe_as_N}
    \P{F_k(V,T_{\dstat})=t}=\frac{ \E{N\big(t,T_{\dstat}\big)}}{|\dstat|}.
\end{equation}
To prove \eqref{eq:local_conv_goal} we estimate the first and second moment of $N(t,T_{\dstat})$. Recall that $T_{\dstat}$ is an uniform rooted tree with degree statistics $\dstat$, and so

$$\E{N\big(t,T_{\dstat}\big)}
=
\frac{\sum_{T\in \mathbb{T}_{\dstat}}N(t,T)}{\#\mathbb{T}_{\dstat}}
=
\frac{\sum_v \sum_{T\in\mathbb{T}_{\dstat}}\I{F_k(v,T)=t}}{\#\mathbb{T}_{\dstat}},
$$
where $\mathbb{T}_{\dstat}$ is the set of rooted trees with degree statistics $\dstat$. First, we get from \cite[equation 283]{pitman2006combinatorial} that
    
$$\#\mathbb{T}_{\dstat}
=
\frac{(|\dstat|-1)!} {\prod_{i\in \N_0} \dstat_i !}
$$
To estimate the expectation of $N\big(t,T_{\dstat}\big)$ it remains to count, for a fixed vertex $v$, how many trees with degree statistics $\dstat$ have $F_k(v,T_{\dstat})=t$. To find all these trees, it suffices to fix a rooted tree on the remainder of vertices, with a marked vertex, marking where vertex $v_k$ is grafted. That is the same as choosing a rooted tree with degree statistic $(\dstat_0-\dstat_0(t)+1, \ \dstat_2-\dstat_2(t),\ \dstat_3,\dots)$, where the extra leaf is marked and will be merged with $v$. There are

$$(\dstat_0-\dstat_0(t)+1) 
\frac{(|\dstat|-\#t)!} {(\dstat_0-\dstat_0(t)+1) !(\dstat_2-\dstat_2(t)) !\prod_{i\geq 3} \dstat_i ! }
$$
such trees. The first factor of the product stems from the choice of the marked leaf, the remainder from the number of rooted trees with said degree statistics. Putting it all together yields

\begin{align*}
    \E{N(t,T_{\dstat})}
    =& \ 
    \frac{\dstat_0! \dstat_2! }{(\dstat_0-\dstat_0(t)) !(\dstat_2-\dstat_2(t)) !}
    \frac{(|\dstat|-\#t)!}{(|\dstat|-1)!}.
\end{align*}
Using Stirling formula alongside \eqref{eq:assumption_local_limit} leads to

\begin{equation}
    \frac{\E{ N(t,T_{\dstat})}}{|\dstat|}
    \xrightarrow{P}
    2^{-\#t}
\end{equation}
Plugging the above into \eqref{eq:proba_fringe_as_N} and recalling the law of $B$ given in \eqref{eq:proba_k-rooted_bienamye} we get

$$
\P{F_k(V,T_{\dstat})=t}\xrightarrow{P}\P{B=t}=\P{F_k(\rho,K)=t}.$$
Then, we compute the second moment of $N(t,T_{\dstat})$. Doing so directly is combinatorially hard because, for two vertices $v_1$ and $v_2$, their $k$-fringes can intersect. This is why we define
    
$$N(t,t',T_{\dstat})=\# \{v_1,v_2 \in T_{\dstat}, \ F_k(v_1,T_{\dstat})=t, \ F_k(v_2,T_{\dstat})=t', \ F_k(v_1,T_{\dstat})\cap F_k(v_2,T_{\dstat})=\emptyset \},$$
counting the number of pairs of vertices having $t$ and $t'$ as disjoint $k$-fringe. Similar computation as done previously yields

\begin{equation}\label{eq:limit_count_two_fringes}
    \frac{\E{ N(t,t',T_{\dstat})}}{|\dstat|^2}
    \xrightarrow{P}
    2^{-\#t}2^{-\#t'}.
\end{equation}
Remark that for $V_1$ and $V_2$ two vertices of $T_{\dstat}$ chosen uniformly at random we have 

$$\P{ F_k(V_1,T_{\dstat})=t, \ F_k(V_2,T_{\dstat})=t'}
\geq
\frac{\E{ N(t,t',T_{\dstat})}}{|\dstat|^2},
$$
with the inequality stemming from the condition of non-intersection in the definition of $N(t,t',T_{\dstat})$. However, we established in \eqref{eq:limit_count_two_fringes} that the limit of the right-hand-side defines a probability distribution on pairs of planar-$k$-fringes, hence 

$$\P{ F_k(V_1,T_{\dstat})=t, \ F_k(V_2,T_{\dstat})=t'}
\xrightarrow{P}
2^{-\#t}2^{-\#t'}.
$$
Moreover, 

$$\P{F_k(V_1,T_{\dstat})=F_k(V_2,T_{\dstat})=t}=\frac{ \E{N(t,T_{\dstat})^2}}{|\dstat|^2},$$
so we just proved that

\begin{equation}
    \frac{\E{N(t,T_{\dstat})^2}}{\E{N(t,T_{\dstat})}^2}
    \xrightarrow{P}1.
\end{equation}
Recalling \eqref{eq:proba_fringe_quenched} alongside a second moment argument proves that

\begin{equation}
    \P{F_k(V,T_{\dstat})=t\mid T_{\dstat}}
    \xrightarrow{P}
    \P{F_k(\rho,K)=t},
\end{equation}
concluding the proof of the lemma.
\end{proof}

The proof of the convergence of the edge lengths takes for input \cref{lem:weak_limit_in_steps} and the following technical lemmas, their proofs are postponed to \cref{app:skeleton_local_conv}. 

\begin{restatable}{lem}{compositionconv}\label{lem:limit_composition}
    Let $(E_1,\dots,E_{K_N})$ be an uniform weak composition of $N$ into $K:=K_N$, that is, $K$ non-negative integers summing to $N$ chosen uniformly at random. Suppose that $K$ goes to infinity with $N$ and that $K=o(N)$. Then, for any fixed $\ell$ we have

    $$\frac{K_N}{N}(E_1,\dots,E_{\ell}) \convd (\mathcal{E}_1,\dots,\mathcal{E}_{\ell}),$$
    a vector of $\ell$ independent copies of $\operatorname{Exp}(1)$ random variables.
\end{restatable}

We later use this lemma with $N=\dstat_1^{(n)}$ and $K_N=-1+\dstat^{(n)}_0+\dstat^{(n)}_2+\dots$. Next, we prove that if the edge length and the underlying combinatorial tree of $\skelr$ converges then convergence holds in the topology defined in \cref{subsec:def_marked_local_topology}.

\begin{restatable}{lem}{weakconvergencesteps}\label{lem:steps_weak_convergence}
    Let $\cT_n$ be a tree with edge lengths and $T_n$ its underlying combinatorial tree. Suppose that $T_n$ has size $n$ and that, for every finite collection $e_1,\dots,e_\ell$ of edges chosen independently from their lengths, their lengths $l_{e_1},\dots,l_{e_\ell}$ converge in distribution to independent $\operatorname{Exp}(1)$ random variables. Also suppose that for $V$ a vertex of $T_n$ chosen uniformly at random, any $k\geq 0$ and any $k$-fringe $t$

    \begin{equation}\label{eq:local_conergence_for_topology}
    \P{F_k(V,T_{n})=t\mid T_n}
    \xrightarrow{P}
    \P{F_k(\rho,K)=t},
    \end{equation}
    where $K$ is the binary-Kesten tree. Then, $\cT_n$ converges in the marked local topology to the binary-Kesten tree with $\operatorname{Exp}(1)$ edge lengths.
\end{restatable}

We are now ready to prove \cref{prop:skeleton_local_limit}.

\begin{proof}[Proof of \cref{prop:skeleton_local_limit}]

    In this proof we let $\dstat^{(n)}$ be the degree statistic of $\skel$. By construction, $\dstat_0=\lfloor \tfrac{\lambda n}{\gamma_n^2}\rfloor$. By \cref{lem:size_skel}, $\frac{\gamma_n^2 \dstat_2(\skel)}{n}\convp \lambda $, and so

    \begin{equation}\label{eq:n0_equivalent_n1}
        \frac{\dstat_0}{\dstat_2}\xrightarrow{P}1.
    \end{equation}
    Because $\skel$ is a tree, $\dstat_0=2+\sum_{i\geq 2}\dstat_i(i-1)$, so


    \begin{equation}\label{eq:n0_+_n2_is_n}
        \frac{|\dstat|-\dstat_1}{\dstat_0}\xrightarrow{P}2.
    \end{equation}
    Using \eqref{eq:roots_fringe}, which translates to $\tfrac{\gamma_n}{n}\sum \dstat_i \convp \tfrac{\sqrt{2\lambda}}{\sigma}$, we get

    \begin{equation}\label{eq:limit_for_edge_length}
        \frac{\gamma_n \dstat_1}{n} \convp \frac{\sqrt{2\lambda}}{\sigma}; \quad  \frac{\gamma_n^2}{n}\left(-1+\dstat_0+\sum_{i\geq 2}\dstat_i\right)
    \convp
    2\lambda.
    \end{equation}
    We recognize $\#\skelr$ in the right-hand-side of \eqref{eq:limit_for_edge_length}, proving the last statement of \cref{prop:skeleton_local_limit}. Then, \cref{lem:weak_limit_in_steps} proves that the underlying combinatorial tree of $\skelr$ is an uniform tree with degree statistic $(\dstat_0,0,\dstat_2,\dots)$. Plugging \eqref{eq:n0_equivalent_n1} and \eqref{eq:n0_+_n2_is_n} into \cref{lem:loc_convergence_shape} proves that \eqref{eq:local_conergence_for_topology} holds for $T_n$ the underlying combinatorial tree of $\skelr$. \cref{lem:weak_limit_in_steps} also proves that the edge lengths of $\skelr$ are distributed as an uniform weak composition of $\dstat_1$ into $-1+\dstat_0+\dstat_2+\dots$. Plugging \eqref{eq:limit_for_edge_length} into \cref{lem:limit_composition} proves that any finite collection of edges of $\sigma\sqrt{2\lambda}\skelr/\gamma_n$ converges in distribution to $\operatorname{Exp}(1)$ random variables, independent of everything else.  It remains to use \cref{lem:steps_weak_convergence} to conclude the proof of \cref{prop:skeleton_local_limit}.
\end{proof}




\begin{figure}
    \centering
    \begin{tikzpicture}[
    scale=0.6,
    every node/.style={
        circle,
        draw=black,
        fill=white,
        minimum size=6mm,
        inner sep=0pt,
        line width=0.8pt
    },
    spine/.style={line width=1pt},
    tree/.style={line width=1pt},
    lab/.style={draw=none,fill=none,font=\normalsize}
]


\coordinate (v)  at (0,0);
\coordinate (v2) at (3,-1.2);
\coordinate (v3) at (5.6,-2.2);
\coordinate (v4) at (8.2,-3.2);
\coordinate (rho) at (11.4,-3.2);


\draw[black,spine]   (v) -- (v2);
\draw[black,spine]    (v2) -- (v3);
\draw[black,spine] (v3) -- (v4);
\draw[dashed,gray!70,line width=1pt] (v4) -- (rho);


\coordinate (g2) at (0,1.3);
\coordinate (g3) at (-1.1,2.6);
\coordinate (g4) at (0,2.6);
\coordinate (g5) at (-1.1,1.3);

\draw[green!60!black,tree] (v)--(g2);
\draw[green!60!black,tree] (g2)--(g3);
\draw[green!60!black,tree] (g2)--(g4);
\draw[green!60!black,tree] (v)--(g5);


\coordinate (b1) at (3,0);
\coordinate (b2) at (3,1.7);
\coordinate (b3) at (1.8,1.7);
\coordinate (b4) at (3,3.0);
\coordinate (b5) at (1.8,3);
\coordinate (b6) at (4.2,3);
\coordinate (b7) at (3.5,1);

\draw[blue,tree] (v2)--(b1);
\draw[blue,tree] (b1)--(b2);
\draw[blue,tree] (b2)--(b5);
\draw[blue,tree] (b2)--(b6);
\draw[blue,tree] (b1)--(b7);


\coordinate (r1) at (5.6,-0.9);

\draw[red,tree] (v3)--(r1);


\coordinate (o1) at (8.2,-1.5);
\coordinate (o2) at (7.1,0);
\coordinate (o3) at (9.3,0);

\draw[orange,tree] (v4)--(o1);
\draw[orange,tree] (o1)--(o2);
\draw[orange,tree] (o1)--(o3);


\foreach \x in {v,g2,g3,g4,g5}
    \filldraw[fill=green!80!lime] (\x) circle (3.2pt);

\foreach \x in {v2,b1,b2,b5,b6,b7}
    \filldraw[fill=blue] (\x) circle (3.2pt);

\foreach \x in {v3,r1}
    \filldraw[fill=red] (\x) circle (3.2pt);

\foreach \x in {v4,o1,o2,o3}
    \filldraw[fill=orange] (\x) circle (3.2pt);

\filldraw[fill=black] (rho) circle (3.2pt);


\node[lab] at (-0.15,-0.55) {$v$};
\node[lab] at (2.9,-1.85) {$v_2$};
\node[lab] at (5.55,-2.85) {$v_3$};
\node[lab] at (8.05,-3.85) {$v_4$};
\node[lab] at (11.35,-3.85) {$\rho$};

\end{tikzpicture}

    \caption{Sketch of a $4$-rooted fringe.}
    \label{fig:k-rooted}
\end{figure}
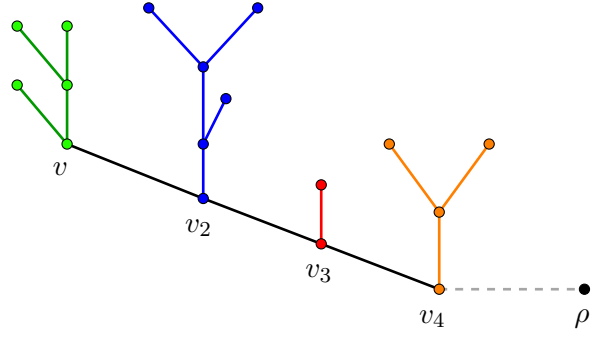

\section{Continuity of packings}\label{sec:continuity_covering}

In this section we discuss packings of compact $\R$-trees, focusing specfically on the continuity of the type with respect to the Gromov--Hausdorff topology. The basic definitions in \cref{sec:packing_number_local} for the type of a packing, the $r$-type of a tree, and an $r$-type-attaining packing all straightforwardly extend to rooted compact $\R$-trees. We will be viewing the type as a function $\type_r:\T \to [0,r)$ in this section. Our goal is to prove the following preliminary result needed in the proof of \cref{thm:bien}. For the next proposition we introduce the notation $\cT^{(\sigma,x)}$ to indicate a random $\R$-tree encoded by a Brownian excursion $(\frac{2}{\sigma}e_t^{(\sigma x /2 )} : t \geq 0)$, where $(e^{(y)}_t : t \geq 0)$ denotes an excursion of Brownian motion conditioned to have height at least $y$ for any $y \geq 0$.

\begin{prop}\label{prop:convergence_of_the_type}
For any $\sigma,x,r>0$, $\type_r$ is almost surely continuous at $\cT^{(\sigma,x)}$.
\end{prop}

When dealing with finite trees, the existence of $r$-TAPs was evident from the existence of an algorithm that generates them. However, the algorithm required us to initialize with packing centres on every leaf, something that will not work on $\R$-trees due to their potentially uncountable set of leaves. Luckily, the existence of $r$-TAPs is a relatively direct consequence of compactness, and we will briefly pause here to record it as a fact.

\begin{fact}
    Let $t \in \T$ and let $r > 0$. There exists an $r$-TAP of $t$.
\end{fact}

\begin{proof}
    Let $(P_n)_{n \geq 0}$ be a sequence of maximum $r$-packings such that $\type(P_n) \to \type_r(t)$ as $n \to \infty$. Since $t$ is compact, $P_n$ has some convergent subsequence $P_{n_k}$ in $t$ in the sense that, if $P_{n_k} = \{x^{1}_{n_k},...,x^{\cP_r(t)}_{n_k}\}$, then there exists a subset $P = \{x^{1},...,x^{\cP_r(t)}\}$ of $t$ such that $x^i_{n_k} \to x^i$ as $k \to \infty$. By the definition of a packing we have that $d(x_{n_k}^i,x_{n_k}^j) \geq r$, and so we also have that $d(x^i,x^j) \geq r$, implying that $P$ is an $r$-packing. Moreover, by assumption, we have that $\type(P) = \lim_{k \to \infty} \type(P_{n_k}) = \type_r(t)$.
\end{proof}

\subsection{A sufficient condition for continuity of the type}

Recall that $\T$ is the space of compact rooted $\R$-trees equipped with the rooted Gromov--Hausdorff topology. Let $t \in \T$, and let $P$ be an $r$-packing of $t$ for some fixed $r > 0$. If $x,y \in P$ satisfy $d(x,y) = r$, then we call them a \emph{tight pair}. We call $P$ \emph{tight} if it contains a tight pair, and \emph{relaxed} if it does not contain a tight pair. We will begin by showing that having relaxed $r$-TAPs is a sufficient condition for $\type_r$ being continuous at a particular tree.

\begin{prop}\label{prop:no_tight_continuity}
    Let $t \in \T$ admit a relaxed $r$-TAP. Then, $\type_r:\T \to [0,r)$ and $\cP_r:\T \to \Z^+$ are both continuous at $t$.
\end{prop}

In the proof of \cref{prop:no_tight_continuity} we use the fact that $\cP_r$ is upper semi-continuous on the entirety of $\T$.

\begin{lem}
    $\cP_r: \T \to \Z^+$ is upper semi-continuous for all $r > 0$.
\end{lem}

\begin{proof}
We will prove the statement by indirect reasoning. Write $m=\cP_r(t)$ and suppose that for any $n\ge 1$, there exists $(t_n)_{n \geq 1}$ such that $d_{GH}(t,t_n)<1/(2n)$ and there exists an $r$-packing $P_n=\{y^n_1,\dots, y^n_{m+1}\}$ of $t_n$. Let $R_n$ be a correspondence of $t$ and $t_n$ with distortion smaller than $1/n$. For all $n$ and $i$ choose $x^n_i \in t$ such that $(x^n_i,y^n_i)\in R_n$, implying that 
\[d(x^n_i, x^n_j)>d(y^n_i,y^n_j)-1/n\ge r-1/n.\]
Then, since $t$ is compact, there is a subsequence $n_k$ such that for any $1\le i \le m+1$ there is an $x_i\in t$ for which $x^{n_k}_i\to x_i$. Since $d(x^{n_k}_i, x^{n_k}_j)>r-1/n_k$ it must hold that $d(x_i,x_j)\ge r$ for $1\le i<i\le m+1$, contradicting that $m=\cP_r(t)$.

    
\end{proof}

\begin{proof}[Proof of \cref{prop:no_tight_continuity}.]

By the previous lemma, to show continuity of $\cP_r$ at $t$, it suffices to show $\cP_r$ is lower semi-continuous at $t$. Let $P$ be an $r$-TAP of $t$ with no tight pairs. There exists some $\eps > 0$ such that $\inf_{x,x' \in P} d(x,x') \geq r + \eps$. Let $t' \in \T$ be such that $\dgh(t,t') \leq \eps/2$ and let $R \subseteq t \times t'$ be a correspondence such that $\dis(R) \leq \eps$. Let $P'$ be a set in $t'$ which is obtained by taking exactly one representative from the set $\{x' \in t' : (x,x') \in R\}$ for each $x \in P$. Then, for any $x_1',x_2' \in P'$ we have $x_1,x_2 \in P$ such that $d(x_1',x_2') \geq d(x_1,x_2) - \eps \geq r$. Thus, $P'$ is an $r$-packing of $t'$, and $\cP_r(t') \geq \cP_r(t)$.

Now we turn towards the continuity of $\type_r$. Let $t,t',P,P',R$ be as above, and, by potentially making $\eps$ smaller, using the upper semi-continuity of $\cP_r$, suppose that $\#P'=\cP_r(t')$. Let $\rho$ and $\rho'$ be the roots of $t$ and $t'$ respectively. Since $(\rho,\rho') \in R$ and $P$ is an $r$-TAP of $t$ we have
$$
\type_r(t') \geq \type(P') = \inf_{p' \in P'} d(\rho',p') \geq \inf_{p \in P} d(\rho,p) - \eps = \type_r(t) - \eps,
$$
implying that $\type_r$ is lower semi-continuous at $t$.

To prove upper semi-continuity we consider a sequence of trees $(t_n)_{n \geq 1}$ such that $t_n \to t$ and $\type_r(t_n) \to \limsup_{t' \to t} \type_r(t')$ as $n \to \infty$. Since $\cP_r$ is continuous at $t$ and discrete, we can assume without loss of generality that $\cP_r(t_n) = \cP_r(t) := m$ for all $n \geq 1$. For all $n\geq 1$, let $R_n \subseteq t \times t_n$ be correspondences such that $\dis(R_n) < 2\dgh(t_n,t)$ if $t_n\neq t$ and if $t_n = t$ take the trivial correspondance. Let $P_n = \{ x_n^{(1)} , ... , x_n^{(m)} \}$ for all $n \geq 1$. For all $n \geq 1$, let $Q_n = \{ y_n^{(1)} , ... , y_n^{(m)} \}$ be a set of points in $t$ formed by taking $y_n^{(i)} \in \{ x \in t : (x,x_n^{(i)}) \in R_n \}$ for each $1 \leq i \leq m$. By assumption, $Q_n$ is an $(r-\dis(R_n))$-packing of $t$, and $\type(Q_n) \geq \type_r(t_n) - \dis(R_n)$. By compactness, we can take a subsequence such that $y_{n_k}^{(i)} \to y^{(i)}$ for some $y^{(i)} \in t$ for some collection $Q = \{y^{(1)} , ... , y^{(m)}\}$. Moreover, since $\dis(R_n) \to 0$ as $n \to \infty$ and $\type_r(t_n) \to \limsup_{t' \to t} \type_r(t')$ we get that $Q$ is a maximum $r$-packing of $t$ with type at least $\limsup_{t' \to t} \type_r(t')$.
\end{proof}

\subsection{Passing off to the Brownian CRT}

In this subsection, we relate the Brownian tree conditioned to have large height to its more well-studied relative, the Brownian CRT (conditioned to have mass $1$). In particular, we have the following result.

\begin{lem}\label{lem:absolute_continuity}
Let $\sigma,x > 0$. For all $r > 0$, define 
$$
E_r = \left\{t \in \T : t \text{ does not admit a relaxed $r$-TAP}\right\}.
$$
If $\prob\{\cT \in E_r\} = 0$ for all $r > 0$, then $\prob\{\cT^{(\sigma,x)} \in E_r\} =0$ for all $r > 0$.
\end{lem}

\begin{proof}

Let $\mathbf{e}^{(s)}_t$ be a length $s$ Brownian excursion. By Brownian scaling, $(\frac{1}{\sqrt{s}}\mathbf{e}^{(s)}_{st} : 0 \leq t \leq 1)$ is a length $1$ Brownian excursion. In particular, if $\cT_s$ is the tree encoded by $2\mathbf{e}^{(s)}$, then $\cT_s \eqdist \frac{1}{\sqrt{s}}\cT$. Hence, $\prob\{\cT_s \in E_r\} = \prob\{\cT \in E_{r/\sqrt{s}}\}$ for all $r > 0$. Given this fact, the proof is essentially a direct consequence of the excursion theory for Brownian motion. We will outline the argument without introducing all of detailed constructions and facts, though the reader may look to \cite[Chapter 7]{the_bible} for a clearer picture.

Let $\bn$ denote the excursion measure of Brownian motion, and let $\bn^+$ denote the restriction of the excursion measure to $C([0,\infty) , [0,\infty))$. Let $e = (e_t : t \geq 0)$ be the canonical process for $\bn^+$. The lifetime of $e$ is given by $R = \inf\{t > 0 : e_t = 0\}$, and the height by $H = \sup_{t \geq 0} e_t$. The laws of both $R$ and $H$ under $\bn^+$ are given as follows:
$$
\bn^+\{R \geq s\} = \sqrt{\frac{2}{\pi s}}, \quad \bn^+\{H \geq s\} = \frac{1}{s}.
$$
The measure $\bn^+(\cdot \cap \{H \geq s\})/\bn^+\{H \geq s\}$ is the law of a Brownian excursion conditioned to have height at least $s$. In particular, this means that the law is absolutely continuous with respect to $\bn^+$. Additionally, for $\bn^+_s$  the law of a Brownian excursion of length $s$, then,

$$
\bn^+\{\cdot\} = \int_0^\infty \bn^+_s(\cdot \cap \{R = s\}) \frac{1}{\sqrt{2\pi s^3}} \ ds.
$$
If $\cT(e)$ is the compact $\R$-tree encoded by the canonical process $e$, then then by the argument in the first paragraph we have, for $r > 0$,
$$
\bn^+\{\cT(e) \in E_r\} = \int_0^\infty \prob\left\{\cT \in E_{r/\sqrt{s}}\right\} \frac{1}{\sqrt{2\pi s^3}} \ ds = 0,
$$
which means that $\prob\{\cT^{(\sigma,x)} \in E_r\} = 0$.
\end{proof}

\subsection{The Brownian CRT and its local limit}

Here we review some facts about the Brownian CRT and its various representations that will be used in the proof of \cref{prop:convergence_of_the_type}

Let $U = (U_k)_{k \geq 1}$ be a sequence of independent random variables uniformly distributed over $[0,1]$. Let $C = (C_k)_{k \geq 1}$ be the arrival times of an inhomogeneous Poisson process on $t\ge 0$ with intensity $t dt$ which is independent of the sequence of uniforms. For convenience define $C_0 = U_0 = 0$. Let $(\be_k)_{k \geq 1}$ be the standard orthonormal basis for the space $\ell_1 = \{(x_k)_{k \geq 1} : \sum_{k \geq 1} x_k < \infty\}$. We can construct a random rooted compact $\R$-tree from $U$ and $C$ as a subset of $\ell_1$ that is distributed as a Brownian CRT. In particular, the compact subspace $\cT := \overline{f([0,\infty))}$ of $\ell_1$ is a Brownian CRT, where we define $f:[0,\infty) \to \ell_1$ by setting $f(0) = \bf{0}$ and defining the rest of the function according to the following rule:
$$
f(s) = f(U_k C_k) + (s-C_k)\be_{k+1}, \quad \forall C_k < s \leq C_{k+1}.
$$
for each $k \geq 1$, we let $\cT_k = f([0,C_k])$. We view $\cT$ and $\cT_k$ for each $k \geq 1$ as being rooted at $\mathbf{0}$. Heuristically, $\cT_{k+1}$ is derived from $\cT_k$ by selecting a uniform point in $\cT_k$, and attaching a stick of length $C_{k+1}-C_k$ to it. This particular construction of the Brownian CRT is commonly referred to as the \emph{stick-breaking construction} and was first introduced by Aldous in \cite{aldousContinuumRandomTree1991}.

Also introduced in \cite{aldousContinuumRandomTree1991} is a description of the local geometry of the Brownian CRT at the root, which is referred to as the self-similar CRT. For each $s > 0$, let $\psi_s$ denote the law of $s \cT$, the Brownian CRT re-scaled by a factor of $s$. Then, define a $\sigma$-finite measure $\psi$ on $\ell_1$ 
$$
\psi\{ A \} = \sqrt{\frac{2}{\pi}}\int_0^\infty \frac{1}{s^2}\psi_s\{A\} \ ds.
$$
Finally, let $\Psi$ denote the measure on $[0,\infty) \times \ell_1$ which is the product of the Lebesgue measure with $\psi$. Let $\Xi$ be a Poisson point process with intensity measure $\Psi$. Consider the subset of $[0,\infty)\times \ell_1$ 
$$
 \cR = \Big\{ \big(x,(x_k)_{k \geq 1}\big) : \exists \ \xi \text{ where } (x,\xi) \in \Xi \text{, and } (x_k)_{k \geq 1} \in \xi\Big\}
$$
equipped with the metric $d_\Xi\big(( x , (x_k)_{k \geq 1} ) , ( y , (y_k)_{k \geq 1} ) \big)$, given by the following expression:
\begin{align*}
|x-y| + \I{x = y} d_{\ell_1}\big((x_k)_{k \geq 1} , (y_k)_{k \geq 1}\big) +\I{x \ne y}\sum_{k\ge 1}(x_k+y_k) 
\end{align*}
Then, $(\cR, (0,\mathbf{0}) , d_\Xi) \in \T_{\operatorname{loc}}$. Heuristically, we can view $\cR$ as the resultant of the process where, for each $(x,\xi) \in \Xi$, we glue the compact re-scaled Brownian CRT $\xi$ at distance $x$ along an infinitely long stick. 

\begin{thm}[Theorem 11 (iii) in \cite{aldousContinuumRandomTree1991}]\label{thm:crt_local_limit}
    Let $(\cT,\mathbf{0})$ be the Brownian CRT, and $(\cR,(0,\mathbf{0}))$ the self-similar CRT. Then, $(n\cT,\mathbf{0}) \to (\cR,(0,\mathbf{0}))$ as $n \to \infty$ with respect to the local Gromov--Hausdorff topology.
\end{thm}

To finish the review, we present a number of known facts about the Brownian CRT that we will use below. In the facts we reference sampling uniform points in $\cT$. The uniform measure $\mu$ on the CRT may be constructed via martingale limits inside the stick-breaking procedure; see \cite{aldousContinuumRandomTree1991}. 

\begin{enumerate}
    \item $\height(\cT)$ follows a known explicit distribution, and in particular, it satisfies $\e[\height(\cT)] < \infty$ \cite{wang2015height}.
    \item \label{point:2} Almost surely, $\mu$ on $\cT$ has full support and satisfies $\mu(\cL(\cT)) = 1$ \cite{aldousContinuumRandomTree1991}.
    \item If $X$ is a uniform point in $\cT$, then $(\cT,X) \eqdist (\cT,\mathbf{0})$. This property we refer to below as \emph{re-rooting invariance} \cite{duquesne2009re-rooting}.
    \item Almost surely, every branch point of $\cT$ has degree three, i.e., for all $b \in \branch(\cT)$, $\cT \setminus \{b\}$ has three connected components \cite{albenque2015brownian}.
    \item \label{point:5} Let $X_1,X_2,X_3$ be three independent uniform points in $\cT$, and let $B$ be the almost surely unique branch point in $\Span\{X_1,X_2,X_3\}$. Let $\cT_1,\cT_2,\cT_3$ be the connected components of $\cT \setminus \{B\}$ containing $X_1,X_2,X_3$ respectively with the point $B$ added back in. Then, given $(\mu(\cT_1),\mu(\cT_2),\mu(\cT_3))$, the trees $(\cT_1,\cT_2,\cT_3)$ are independent, and they satisfy that $((\mu(\cT_i))^{-1/2}\cT_i,B) \eqdist (\cT,\mathbf{0})$ for $i=1,2,3$ \cite{aldous1994recursive,albenque2015brownian}.
\end{enumerate}

\subsection{Tight $r$-TAPs in a Brownian CRT}

Here we complete the proof of \cref{prop:convergence_of_the_type} by proving that the Brownian CRT admits relaxed $r$-TAPs almost surely for every $r > 0$. Throughout this subsection, we keep $r$ fixed and always take $\cT$ to be a Brownian CRT generated by the stick-breaking construction.

Let $t \in \T$. We call a point $x \in t$ \emph{$\gamma$-mobile} if there exists some $\eps \leq r/(2\gamma)$, and some $b \in \branch(t)$ such that $d(x,b) \leq \eps$ and the closure of $\cT \setminus \{b\}$ contains a tree $t_\eps\not\ni \rho$ (rooted at $b$) with $\height(t_\eps) \in (\gamma\eps , 2\gamma \eps)$. 

We will call an $r$-TAP of $\cT$, $P$, \emph{leaf-supported} if $P \subseteq \cL(\cT)$. 
 
\begin{lem}\label{lem:non_leaf_implies_immobile}
    If every $r$-TAP of $\cT$ is not leaf-supported, then there exists some $x \in \cT \setminus \cL(\cT)$ such that $x$ is not $1$-mobile.
\end{lem}

\begin{proof}
    Let $P$ be an $r$-TAP of $\cT$ that minimizes the size of $P\setminus \cL(\cT)$. Suppose that $x \in  P \setminus \cL(\cT)$. Let $\eps < r/2$, and suppose that there exists a branch point $b \in \branch(\cT)$ such that $d(x,b) \leq \eps$ and the closure of $\cT \setminus \{b\}$ contains a tree $t_\eps\not\ni\rho$ with $\height(t_\eps) \in [\eps , 2\eps]$. Let $y \in t_\eps$ be a leaf that attains the height of $t_\eps$. We claim that $(P \cup \{y\}) \setminus \{x\}$, is also an $r$-TAP of $\cT$. First, $P \cap t_\eps\setminus\{x\}$ must be empty as $y$ attains the height, and $d(x,y) \leq 2\eps < r$. Because of this, and because $x$ was in the original packing, we have
$$
\inf_{x' \in P\setminus\{x\}} d(y,x') \geq \inf_{x' \in P\setminus\{x\}} d(x,x') +d(y,b) - d(x,b) \geq r + \eps - \eps = r.
$$

For a similar reason, since $\rho\not\in t_\eps$, we can conclude that $d(\rho,y) \geq d(\rho,x)$, and thus, that $P'$ is an $r$-TAP. By construction, $\#\{P' \setminus \cL(\cT)\} < \#\{P \setminus \cL(\cT)\}$, and we have contradicted the optimality of $P$.
\end{proof}

\begin{lem}\label{lem:fixed_point_mobile}
    Let $\gamma > 0$ be fixed. Let $U$ be a uniformly chosen point in $\cT$. Almost surely, $U$ is $\gamma$-mobile.
\end{lem}

\begin{proof}
Consider the following subset of marked trees in $\T_{\operatorname{loc}}$, which we denote by $A_n$ for all $n \geq 1$: 
\[\bigcup_{\eps < n^{-1}}\left\{(t,x):\exists b \in \branch(t): d(x,b) <\eps, \ \exists t_\eps \subseteq t \setminus \{b\} \text{ with } \height(t_\eps)\in(\gamma\eps, 2\gamma \eps)\right\} .\]
It is clear that if $n \geq 2\gamma/r$, then $x$ is $\gamma$-mobile in $t$ if $(t,x)\in A_n$ and $d(x,\rho)>3n^{-1}$, where the last condition ensures that $\rho\not\in t_\eps$. Moreover, $A_n$ is an open set with respect to the local Gromov--Hausdorff topology on $\T_{\operatorname{loc}}$. Then, using the re-rooting invariance of the Brownian CRT we have,
\begin{align*}&\P{U \text{ $\gamma$-mobile in }\cT}  \geq \P{(\cT,U)\in A_n}-\P{d(U,\rho) \leq 3n^{-1}}\\
&\qquad=\P{\left(\cT,\mathbf{0}\right)\in A_n} +o(1) = \prob\left\{ \left(n\cT,\mathbf{0}\right) \in A_1 \right\}+o(1).\end{align*}
Hence, we are done if we show that $\P{(n\cT,\mathbf{0})\in A_1}\to 1$ as $n \to \infty$. By \cref{thm:crt_local_limit}, $(n\cT,\mathbf{0})$ converges to $(\mathcal{R},(0,\mathbf{0}))$ in the local Gromov--Hausdorff topology as $n \to \infty$. Since $A_1$ is open, we know that 
\[\liminf_{n\to \infty} \P{(n\cT,\mathbf{0})\in A_1}\ge \P{(\cR,(0,\mathbf{0}))\in A_1},\]
so it remains to show that $\P{(\cR,(0,\mathbf{0}))\in A_1}=1$. Using the definition of $\psi$ a simple change of variables we can get,

\begin{align*}
    \psi\{ t : \height(t) > a \} &= \sqrt{\frac{2}{\pi}}\int_0^\infty \frac{1}{s^2}\prob\left\{ \height(\cT) > \frac{a}{s} \right\} \ ds \\
    &= \sqrt{\frac{2}{\pi}}\frac{1}{a}\int_0^\infty  \prob\{\height(\cT) > u\} \ du \\
    &= \sqrt{\frac{2}{\pi}}\frac{\e[\height(\cT)]}{a},
\end{align*}
which gives, for $A_1' = \{(x,t) : 0 \leq x \leq 1, \ \height(t) \in (\gamma x , 2\gamma x)\}$,
\begin{align*}
    \Psi\{A_1'\} = \frac{\e[\height(\cT)]}{\sqrt{2\pi}}\int_0^1\frac{1}{\gamma x} \ dx = \infty.
\end{align*}
Then, recalling the Poisson construction of $\cR$ from the Poisson point process $\Xi$ we get,
\begin{align*}
    \P{(\cR,(0,\mathbf{0}))\in A_1} &\geq \prob\left\{\operatorname{Poi}(\Psi\{A_1'\}) > 0 \right\} = 1.
\end{align*}
\end{proof}

\begin{cor}\label{cor:leaves}
    Almost surely, $\cT$ has a leaf-supported $r$-TAP.
\end{cor}

\begin{proof}
    Take $(X_k)_{k \geq 1}$ to be a sequence of independent uniform points in $\cT$. Since the uniform measure on $\cT$ has full support, $(X_k)_{k \geq 1}$ is almost surely dense in $\cT$, and by \cref{lem:fixed_point_mobile}, almost surely each point in $(X_k)_{k \geq 1}$ is $2$-mobile. This implies that almost surely every point in $\cT$ is $1$-mobile, and so by \cref{lem:non_leaf_implies_immobile}, implies that $\cT$ almost surely has a leaf-supported $r$-TAP.
\end{proof}

The goal now is to show that, almost surely, a leaf-supported $r$-TAP of $\cT$ can be turned into a relaxed $r$-TAP. Recall from the previous subsection that each $b \in \branch(\cT)$ almost surely has $\deg(b) = 3$. Thus, the subtree $\cT^{\uparrow b}$ consists of two trees that intersect only at $b$. We introduce the notation $\cT^{\uparrow b^-}$ and $\cT^{\uparrow b ^+}$ to denote these two subtrees, which we both view as being rooted at $b$.

The next lemma requires a bit more notation to be introduced. Let $P$ be a leaf-supported packing of $\cT$. Almost surely, $\Span(P \cup \{\mathbf{0}\})$ is isometric to some planted binary tree with edge lengths. The underlying graph structure of that tree we call the \emph{frame} of the packing $P$, and we will denote it by $\mathrm{Fr}(P)$. Note that the most recent common ancestor of any pair in $P$ naturally corresponds to a branchpoint in $\mathrm{Fr}(P)$.

\begin{lem}\label{lem:tight_pair=same_type}
    Almost surely, if $\cT$ does not admit relaxed $r$-TAP, then there exists $b \in \branch(\cT)$ such that $\type_r(\cT^{\uparrow b ^-}) = r-\type_r(\cT^{\uparrow b^+})$.
\end{lem}

\begin{proof}
    Let $(X_n^{1},X_n^2,X_n^3)_{n \geq 1}$ be a sequence of independent triples of independent uniform points in $\cT$. Let $B_n$ be the almost surely unique branch point in $\Span\{X_n^1,X_n^2,X_n^3\}$. Since the uniform measure on $\cT$ has full support, we know that $\{B_n : n \geq 1\} = \branch(\cT)$ almost surely. Then, combining \cref{cor:leaves} with the self similarity of $\cT$ at random branch points, we can assume that $\cT$, $\cT^{\uparrow b^-}$, and $\cT^{\uparrow b^+}$ all have a leaf-supported $r$-TAP for all $b \in \branch(\cT)$. 

     For an $r$-TAP $P$ for $\cT$ contained in $\cL(\cT)$, let $k$ be the smallest generation in $\mathrm{Fr}(P)$ that contains a greatest common ancestor of a tight pair. Suppose $P$ is chosen such that, firstly, this $k$ is maximized, and, secondly among the packings for which $k$ is maximized, the number of such greatest common ancestors in generation $k$ is minimized. This is well-defined because for any $P$, $\mathrm{Fr}(P)$ has at most $\mathcal{P}_r(\mathcal{T})$ generations and at most $\binom{\mathcal{P}_r(\mathcal{T})}{2}$ tight pairs. If $k=\infty$, we found a relaxed $r$-TAP. If not, we let $b$ be a greatest common ancestor of a tight pair in generation $k$ and we will show that $\type_r(\cT^{\uparrow b ^-}) = r-\type_r(\cT^{\uparrow b^+})$ holds. 

     Assume that $\type_r(\cT^{\uparrow b ^-}) \ne r-\type_r(\cT^{\uparrow b^+})$. We show that we can find an $r$-TAP $P_b$ contained in $\mathcal{L}(\mathcal{T})$ that violates one of the two constraints that $P$ is chosen to optimize. (I.e., either the smallest generation $k$ containing a greatest common ancestor of a tight pair will be one higher in $P_b$ than $P$, or in $P_b$ that generation $k$ will have one less greatest common ancestor than in $P$.)

      Let $P^-,P^+$ be leaf-supported $r$-TAPs of $\cT^{\uparrow b^-}$ and $\cT^{\uparrow b^+}$, and let $z^- \in P^-$ and $z^+ \in P^+$ be packing centres whose distance from $b$ equals the type of their respective packings.  We proceed by splitting into cases. We will construct $P_b$ as a subset of $P \cup P^- \cup P^+$ in each case, so that $P_b$ will always be trivially leaf supported.
    
    Suppose that both $P \cap \cT^{\uparrow b^-}$ and $P \cap \cT^{\uparrow b^+}$ are maximum packings of their respective subtrees. Then, we take $P_b = P \setminus \cT^{\uparrow b} \cup (P^- \cup P^+)$. By the definition of type, we have that $\bt(P \cap \cT^{b^-}) \leq \bt_r(\cT^{\uparrow b^-})$ and $\bt(P \cap \cT^{b^+}) \leq \bt_r(\cT^{\uparrow b^+})$, which implies that $P_b$ is an $r$-TAP. Since  $\type_r(\cT^{\uparrow b ^-}) \ne r-\type_r(\cT^{\uparrow b^+})$, we can also conclude that $\bt_r(\cT^{\uparrow b^-}) + \bt_r(\cT^{\uparrow b^+}) > r$, implying that $b$ is not the greatest common ancestor of a tight pair in $P_b$.

    Suppose that neither $P \cap \cT^{\uparrow b^-}$ nor $P \cap \cT^{\uparrow b^+}$ are maximum packings of their respective subtrees. Then, we take $P_b = P \setminus \cT^{\uparrow b} \cup (P^- \cup P^+) \setminus \{z^- , z^+\}$. The second closest centre to the root in an $r$-TAP can at worst be distance $r/2$ from the root. Thus, $\bt(P^- \setminus \{z^-\}) \wedge \bt(P^+ \setminus \{z^+\}) \geq r/2$, implying that $P_b$ is still an $r$-TAP. Since $\type_r(\cT^{\uparrow b ^-}) \ne r-\type_r(\cT^{\uparrow b^+})$ we know that $\bt(P^- \setminus \{z^-\}) \vee \bt(P^+ \setminus \{z^+\}) > r/2$, which means that $b$ is not the greatest common ancestor of a tight pair in $P_b$.

    Suppose that exactly one of $P \cap \cT^{\uparrow b^-}$ or $P \cap \cT^{\uparrow b^+}$ are a maximum packings of their respective subtrees. Assume without loss of generality that $\bt_r(\cT^{\uparrow b^-}) \leq \bt_r(\cT^{\uparrow b^+})$. Then, we take $P_b = P \setminus \cT^{\uparrow b} \cup (P^- \cup P^+) \setminus \{z^-\}$. Since the second closest packing centre to $b$ in $P^-$ can be at worst distance $r/2$ to $b$, we obtain $\bt(P^- \setminus \{z^-\}) \geq r/2$. Since one of $P \cap \cT^{\uparrow b^-}$ or $P \cap \cT^{\uparrow b^+}$ is a maximum packing, and since $\bt_r(\cT^{\uparrow b^-}) \leq \bt_r(\cT^{\uparrow b^+})$, we have that $\bt(P^+) \geq \bt(P \cap \cT^{\uparrow b^-}) \wedge \bt(P \cap \cT^{\uparrow b^+})$. This implies that $P_b$ is an $r$-TAP. Using the same assumptions we also obtain $r < \big(r - \bt_r(\cT^{\uparrow b^-})\big) + \bt_r(\cT^{\uparrow b^+})$, and thus, $r < \bt(P^- \setminus \{z^-\}) + \bt(P^+)$, which implies that $b$ is not the greatest common ancestor of a tight pair in $P_b$.


    
   In all cases, we have been able to construct a $P_b$ that violates the optimality of $P$, so, by contradiction, $\type_r(\cT^{\uparrow b ^-}) \ne r-\type_r(\cT^{\uparrow b^+})$ must hold.

\end{proof}

\begin{lem}\label{lem:bridge_height}
    For any fixed $m \geq 1$ and $z > 0$, almost surely, there does not exist $b \in f([ C_m , C_{m+1})) \cap \branch(\cT)$ such that the subtree rooted at $b$ pendent to $f([ C_m , C_{m+1}))$ has height exactly $z - \height(b)$.
\end{lem}

\begin{proof}
    Let $\cF_{m+1}$ be the $\sigma$-algebra generated by $C_1,\dots, C_{m+1},U_1,\dots, U_m$. Let $(H_k)_{k \geq 0}$ be the heights of the trees pendent to the segment $[C_m,C_{m+1})$, with $(\rho_k)_{k \geq 1}$ the their roots on $f([C_m,C_{m+1}))$. For each $k \geq 1$ set $B_k = d(\rho,\rho_k)$, and let $A_m$ be the event that $B_k = z - H_k$ for some $k \geq 1$. Note that, $(H_k)_{k \geq 1}$ and $(\rho_k)_{k \geq 1}$ are independent, as the placement of sticks is independent of their lengths. Moreover, $B_k = \height(f(C_m)) + U_k(C_{m+1} - C_m)$ for some $U_k$ uniformly distributed over $[0,1]$, independent of $(H_k)_{k \geq 1}$ and $\cF_{m+1}$. Thus,
    \begin{align*}
        \probC{A_m}{\cF_{m+1}} &= \Ce{\Cprob{ \bigcup_{k \geq 1} \left\{ B_k = z - H_k \right\}}{ (H_k)_{k \geq 1}, \cF_{m+1} }}{\cF_{m+1}}\\
        &\leq \Ce{\sum_{k \geq 1} \probC{ B_k = z - H_k}{H_k , \cF_{m+1} }}{\cF_{m+1}}.
    \end{align*}
    By the above remarks, $\probC{ B_k = z - H_k}{H_k , \cF_{m+1} } = 0$, and so $\prob\{A_m\} = 0$.
\end{proof}

\begin{prop}\label{prop:atomless}
    For any $r > 0$, the distribution of $\type_r(\cT)$ is atomless.
\end{prop}

\begin{proof}
    Fix a sample of $\cT$ with a leaf-supported $r$-TAP that we will denote by $P$ (such a $P$ exists almost surely by \cref{cor:leaves}). Suppose for now that $P \subseteq \cL(\cT) \setminus (\cup_{n \geq 1} \cT_n)$. Since $P$ is finite and $\cT=\overline{(\cup_{n \geq 1} \cT_n)}$, we know that, almost surely, there exists some $N \geq 1$ such that for all $x,y \in P$, $\llbracket x ,y \rrbracket \cap \cT_N \ne \emptyset$. Thus, for each $x \in P$ there is a $b_x \in \cT_N$ that is closest to $x$, with $b_x\ne b_y$ if $x\ne y$. Moreover, since $x \notin \cT_N$, $b_x \in \branch(\cT)$ ($b_x$ cannot be a leaf in $\cT_N$ and must be the root of the subtree pendent to $\cT_N$ which contains $x$ as this is the closest point to $x$ in $\cT_N$). Let $v_x$ be a vertex that attains the height of the subtree rooted at $b_x$ pendent to $\cT_N$. Because $\cT_N$ separates the points of $P$, moving packing centres as far from $\cT_N$ in their pendent subtree will maintain the property of being an $r$-TAP. Thus, the set $\{v_x : x \in P\}$ is also an $r$-TAP of $\cT$. This implies that, if $\type_r(\cT) = z$ there must exist some $1 \leq m \leq N$ and some $b \in f([C_m,C_{m+1})) \cap \branch(\cT)$ such that the subtree rooted at $b$ pendent to $\cT_N$ has height exactly $z - d(\rho,b)$. For $m \geq 1$, let $A_m$ be the event that there is such a point on the path $f([C_m,C_{m+1}))$. Then, by \cref{lem:bridge_height},
    $$
    \prob\{\type_r(\cT) = z\} \leq \sum_{m \geq 1} \prob\{A_m\} = 0
    $$
    for any $z > 0$. To remove the assumption that the packing centres in $P$ are not in $\cT_n$ for any $n \geq 1$ it is enough to observe that there are finitely many leaves in $\cT_n$, and their heights are absolutely continuous random variables. Then, one can take a union bound over all $n \geq 1$, noting that the leaves have height equal to $z$ with probability zero. Then, we perform the above argument with all of the other packing centres.
\end{proof}

\cref{prop:convergence_of_the_type} follows from \cref{prop:no_tight_continuity}, \cref{lem:absolute_continuity} and the following Corollary.

\begin{cor}
    Almost surely, $\cT$ admits a relaxed $r$-TAP.
\end{cor}

\begin{proof}
    Let $A$ be the event that $\cT$ does not admit relaxed $r$-TAPs. By \cref{lem:tight_pair=same_type}, we know that, on $A$, there is a branch point $b \in \branch(\cT)$, where $\type_r(\cT^{\uparrow b^-}) = r-\type_r(\cT^{\uparrow b^+})$. Let $(B_n)_{n\ge 1}$ be the random sequence of branch points from the proof of \cref{lem:tight_pair=same_type}, so that $\{B_n : n \geq 1\} = \branch(\cT)$ almost surely. For each $n \geq 1$, let $(\Delta_n^-,\Delta_n^+)$ be the mass of subtrees $\cT^{\uparrow B_n^-}$ and $\cT^{\uparrow B_n^+}$. As was noted in \eqref{point:5}, conditionally given $\Delta_n^-$ and $\Delta_n^+$, the trees $(\Delta_n^-)^{-1/2}\cT^{\uparrow B_n^-}$ and $(\Delta_n^+)^{-1/2}\cT^{\uparrow B_n^+}$ are independent Brownian CRTs. Thus, using the scaling properties of packings we get,
    \begin{align*}
        &\prob\{\type_r(\cT^{\uparrow B_n^-}) = r-\type_r(\cT^{\uparrow B_n^+})\} \\
        \leq& \e\left[ \Cprob{ \type_r\left(\cT^{\uparrow B_n^-} \right) = r-\type_r\left(\cT^{\uparrow B_n^+} \right)}{\Delta_n^- , \Delta_n^+} \right] \\
        \leq&  \e\left[ \Cprob{\type_{\frac{r}{\sqrt{\Delta_n^-}}}\left(\frac{1}{\sqrt{\Delta_n^-}}\cT^{\uparrow B_n^-} \right) = r-\type_{\frac{r}{\sqrt{\Delta_n^+}}}\left(\frac{1}{\sqrt{\Delta_n^+}}\cT^{\uparrow B_n^+} \right)}{\Delta_n^- , \Delta_n^+} \right].
    \end{align*}
    The final line is equal is zero because any two independent atomless random variables are almost surely never equal. Taking a union bound we get,
    $$
    \prob\{A\} \leq \sum_{n \geq 1} \prob\big\{\type_r(\cT^{\uparrow B_n^-}) = r-\type_r(\cT^{\uparrow B_n^+})\big\} = 0.
    $$
\end{proof}
Recalling \cref{lem:absolute_continuity} we can conclude the result for $\cT^{(\sigma,x)}$ as well, proving \cref{prop:convergence_of_the_type}. 
\section{Packing the CRT}\label{sec:covering_CRT_fringe}

In this section we prove the two main ingredients in the proof of \cref{thm:crt}, which are \cref{prop:fringe_CRT} and \cref{prop:crt_local_limit}.

\subsection{Covering the fringe of the CRT}

Let $\delta > 0$ and let $(t,\rho,d,\cX) \in \T^{(m)}$ be a compact rooted $\R$-tree with $m+1$ marked points $\cX\ni \rho$. We define a function $\fs_{\delta} : (0,\delta) \times \T^{(m)} \to [0,\infty)$ as $\fs_{\delta}(\theta,t) = \#A_{\theta,t}$ for all $(\theta,t) \in (0,\delta) \times\T^m$, where
$$
A_{\theta,t} := \{ x \in t \setminus \Span(\cX) : d(x,\Span(\cX)) = \theta  \ ( \mathrm{mod} \  \delta), \ \height(t^{\uparrow x}) > \delta \}.
$$

\begin{fact}
    For any $0 < \theta < \delta$ and $(t,\rho,d,\cX) \in \T^{(m)}$, $\fs_{\delta}(\theta,t) < \infty$.
\end{fact}

\begin{proof}
This is a direct consequence of compactness. We just need to observe that for each $x\in A_{\theta,t}$ there is an $a_x\in t^{\uparrow x}$ with $d(x,a_x)=\delta$, so that the $a_x$ are distinct and $\delta$-separated. An infinite $\delta$-separated set in $t$ would contradict compactness. 
\end{proof}

\begin{lem}\label{lem:weird_function_lower_semi}
    For any $\delta > 0$, the function $\fs_{\delta}$ is lower semi-continuous with respect to the product topology on $(0,\delta) \times \T^{(m)}$.
\end{lem}

\begin{proof}
    Let $(t,\rho,d,\cX = \{\rho,x_1,...,x_m\}), \ (t',\rho',d',\cX' = \{\rho',x_1',...,x_m'\}) \in \T^{(m)}$ and let $0 < \theta,\theta' < \delta$. Suppose that $\dgh^m(t,t') < \frac{\delta^* \wedge \delta^\dag \wedge \theta}{8} =: \frac{\eps}{2}$, where $\delta^* = \min_{a \in A_{\theta,t}} (\height(t^{\uparrow a})-\delta)$ and $\delta^\dag=\min_{a, a'\in A_{\theta,t}}d(a,a')$. Suppose also that $|\theta - \theta'| < \eps$. Let $R$ be a marked correspondence such that $\dis(R) < \eps $. We will argue that $\fs_{\delta}(\theta,t) \leq \fs_{\delta}(\theta',t')$ by providing an injection $f:A_{\theta,t} \to A_{\theta't'}$.

    Let $a \in A_{\theta,t}$. By definition there exists some $z' \in t'$ such that $(a,z') \in R$. Observe that
    \begin{equation}\label{eq:goofy_mgh_1}
    d'(z',\Span(\cX')) = \min_{1 \leq i \leq m} d'(z',\llbracket \rho' , x_{i}'\rrbracket),
    \end{equation}
    and
    \begin{equation}\label{eq:goofy_mgh_2}
    d(a,\Span(\cX)) = \min_{1 \leq i \leq m} d(a,\llbracket \rho , x_{i}\rrbracket).
    \end{equation}
    By partitioning the paths connecting $a$ to both $\rho$ and $x_{i}$ into the parts inside $\llbracket \rho , x_{i} \rrbracket$ and the parts outside we can see that
    $$
    d(a,\rho) + d(a,x_{i}) = 2d(a,\llbracket \rho , x_i \rrbracket) + d(\rho,x_{i}),
    $$
    which by re-arranging gives, for all $1 \leq i \leq m$,
    $$
    d(a,\llbracket \rho , x_i\rrbracket) = \frac{d(a,\rho) + d(a,x_{i}) - d(\rho,x_i)}{2}.
    $$
    For the same reason we also have, for all $1 \leq i \leq m$,
    $$
    d'(z',\llbracket \rho' , x_i' \rrbracket ) = \frac{d'(z',\rho') + d'(z',x_{i}') - d'(\rho',x_{i}')}{2}.
    $$
    By the definition of a marked correspondence we have that $(\rho,\rho'), (x_i,x_i') \in R$, and so for all $1 \leq i \leq m$ combining the previous two lines gives,
    $$
    \big| d'(z',\llbracket \rho' , x_i' \rrbracket) - d(a,\llbracket \rho , x_i \rrbracket) \big| < \frac{3\eps}{2}.
    $$
    Thus, recalling equations \eqref{eq:goofy_mgh_1} and \eqref{eq:goofy_mgh_2} we obtain
    $$
    \big| d(a,\Span(\cX)) - d'(z',\Span(\cX')) \big| < \frac{3\eps}{2}.
    $$
    Using that $d(a,\Span(\cX)) = \theta$ and $|\theta-\theta'| < \eps$ we get
    \begin{equation}\label{eq:z'theta'}
    \big| d'(z',\Span(\cX')) - \theta' \big| < \frac{5\eps}{2}.
    \end{equation}
    Moreover, since $(a,z')\in R$, $\height((t')^{\uparrow z'}) > \delta + \delta^* - \eps.$ Combined with \eqref{eq:z'theta'},  this implies that there exists some $a' \in t'$ such that:
    \begin{enumerate}
        \item $d'(a' , \Span(\cX')) = \theta' \ ( \mathrm{mod}  \ \delta)$,
        \item $d'(a',z') < 5\eps/2$
        \item $\height((t')^{\uparrow a'}) > \delta + \delta^* - \frac{7\eps}{2} > \delta$.
    \end{enumerate}
    Hence, $a' \in A_{\theta',t'}$, and we can set $f(a) = a'$. We finish by showing that $f$ is an injection.

    Suppose that $f(a_1) = f(a_2)$ for $a_1,a_2 \in A_{\theta,t}$. This would imply that there exists $z_1',z_2'$ such that $(a_1,z_1'),(a_2,z_2') \in R$, and $d'(z_1',z_2') < 5\eps$ by point (2) in the previous list. However, this gives $d(a_1,a_2) < 6\eps \leq \delta^\dag$, which implies that $a_1 = a_2$ by the fact that $\cX$ is $\delta^\dag$-separated.
\end{proof}

For the next lemma, for $S$ a Polish space, we let $\cP(S)$  denote the set of all probability measures on $S$. We equip it with the topology of weak convergence

\begin{lem}
    Let $\delta > 0$. The function $\fI_\delta:\T^{(m)} \times \cP((0,\delta)) \to [0,\infty)$ given by
    $$
    \fI_\delta(t,\nu) = \int_{0}^\delta \fs_\delta(\theta,t) \ \nu(d\theta)
    $$
    is lower semi-continuous on $\T^{(m)} \times \cP((0,\delta))$. In particular, the set $\{ (t,\nu) : \fI_\delta > x \}$ is open for any $x > 0$.
\end{lem}

\begin{proof}
    The lower semi-continuity of $\fs_{\delta}$ from \cref{lem:weird_function_lower_semi} and a modified version of Fatou's lemma for weakly converging probability measures (see \cite[Theorem 1.1]{feinberg2014fatou}) yields the result.
\end{proof}

We are now prepared to prove \cref{prop:fringe_CRT}, which is our continuum analogue of \cref{prop:fringe_output}.

\begin{prop}\label{prop:cover_fringe_CRT}
    Let $(\cT,\rho,\mu)$ be a Brownian CRT conditioned to have mass 1, let $(X_k)_{k \geq 1}$ be a sequence of uniform points in $\cT$ sampled from $\mu$, and for any $\delta > 0$, let $\cS^{\delta,\lambda} = \Span\{\rho,X_1,...,X_{\lfloor \lambda/\delta^2 \rfloor}\}$. Then, for any $\eps > 0$,
    $$
    \lim_{\lambda \to \infty}\limsup_{\delta \downarrow 0}\prob\left\{ \cC_{\delta}(\cT) - \cC_{(1-\eps)\delta}(\cS^{\delta,\lambda}) > \frac{\eps}{\delta^2} \right\} = 0.
    $$
\end{prop}

\begin{proof}
    Let $\cX_{\delta,\lambda} = (\rho,X_1,...,X_{\lfloor \lambda/\delta^2 \rfloor})$. The proof will consist of a careful application of \cref{lem:prop_tall_trees}. Towards that end, let $(T_n,\rho_n ,\mu_n)_{k \geq 1} \subseteq \T_{\mathrm{meas}}$ be a sequence where $T_n$ is a Bienaym\'{e} tree with a mean 1 and variance 1 offspring distribution conditioned to have size $n$ and $\mu_n$ is the uniform probability measure on the vertices of $T_n$. Indeed, by the standard theory of scaling limits for random trees we have,
    $$
    \left(\frac{1}{\sqrt{n}}T_n, \rho_n , \mu_n\right) \convd \left(\cT, \rho ,\mu\right)
    $$
    in the GHP topology, and so, if $\cX_{n,\delta,\lambda} = \left(\rho, X_1^{(n)},...,X_{\lfloor \lambda/\delta_2\rfloor}^{(n)}\right)$ for $X_1^{(n)},...,X_{\lfloor \lambda/\delta_2\rfloor}^{(n)}$ independent samples from $\mu_n$, then
    $$
    \left(\frac{1}{\sqrt{n}}T_n, \rho_n , \cX_{n,\delta,\lambda}\right) \convd \left(\cT, \rho ,\cX_{\delta,\lambda}\right)
    $$
    in the $(\lfloor\lambda/\delta^2\rfloor+1)$-marked GH topology.

    Set $\gamma_n = \delta \sqrt{n}$, and define the following probability measure on $[0,\eps\delta/2]$:
    $$
    \nu_n := \frac{1}{\lfloor \frac{\eps\gamma_n}{2}\rfloor}\sum_{i=0}^{\lfloor \frac{\eps \gamma_n}{2}\rfloor - 1} \delta_{in^{-1/2}}.
    $$
    Let $\nu$ denote the uniform probability measure on $[0,\eps\delta/2]$, and note that $\nu_n \to \nu$ weakly. Moreover, observe that
    \begin{align*}
    \frac{1}{\lfloor \frac{\eps\gamma_n}{2}\rfloor}\#\left\{ v \in F_n^\lambda : \height\left({\fforest}^{\uparrow v}\right) > \frac{\eps\gamma_n}{2}\right\}
    =\int_0^{\eps\delta/2} \fs_{\eps\delta/2}\left(\theta,\frac{1}{\sqrt{n}}T_n\right) \ \nu_n(d\theta) + \frac{X_0}{\lfloor \frac{\eps\gamma_n}{2}\rfloor},
    \end{align*}
    where $X_0$ denotes the number of roots in $F_n^\lambda$ whose corresponding tree has height more than $\frac{\eps\gamma_n}{2}$. By compactness, we know that $X_0/\lfloor \eps\gamma_n/2 \rfloor \to 0$ almost surely. Hence, we can apply the Portmanteau theorem to get
    \begin{align*}
    \liminf_{n \to \infty}\prob\left\{ \frac{1}{\lfloor \frac{\eps\gamma_n}{2}\rfloor}\#\left\{ v \in F_n^\lambda : \height\left({\fforest}^{\uparrow v}\right) > \frac{\eps\gamma_n}{2}\right\} > \frac{\eps}{2\delta^2} \right\} 
    \geq\prob\left\{ \int_{0}^{\eps\delta/2} \fs_{\eps\delta/2}(\theta,\cT) > \frac{\eps}{2\delta^2} \right\}.
    \end{align*}
    By \cref{lem:prop_tall_trees} we get
    $$
    \lim_{\lambda \to \infty}\limsup_{\delta \downarrow 0}\liminf_{n \to \infty}\prob\left\{ \frac{1}{\lfloor \frac{\eps\gamma_n}{2}\rfloor}\#\left\{ v \in F_n^\lambda : \height\left({\fforest}^{\uparrow v}\right) > \frac{\eps\gamma_n}{2}\right\} > \frac{\eps}{2\delta^2} \right\} = 0,
    $$
    and so we also get
    $$
    \lim_{\lambda \to \infty}\limsup_{\delta \downarrow 0}\prob\left\{ \int_{0}^{\eps\delta/2} \fs_{\eps\delta/2}(\theta,\cT) > \frac{\eps}{2\delta^2} \right\} = 0. 
    $$
    Since it is necessary to have some $\theta$ with $\fs_{\eps\delta/2}(\theta,\cT) \le \frac{\eps}{2\delta^2}$ in order to have $\int_{0}^{\eps\delta/2} \fs_{\eps\delta/2}(\theta,\cT) \le \frac{\eps}{2\delta^2}$, this gives
    $$
    \lim_{\lambda \to \infty}\limsup_{\delta \downarrow 0}\prob\left\{\fs_{\eps\delta/2}(\theta,\cT) > \frac{\eps}{2\delta^2} \ \forall \theta \in \left(0,\frac{\eps\delta}{2}\right) \right\} = 0. 
    $$
    From here, the proof essentially follows the same logic as the proof of \cref{prop:fringe_cover}. Fix some sample where there exists $\theta^* \in (0,\eps\delta/2)$ such that $\fs_{\eps\delta/2}(\theta^*,\cT) \leq \eps\delta/2$. If we let $C$ be a $(1-\eps)\delta$-covering of $\cS^{\delta,\lambda}$, then $C \cup A_{\theta^*,\cT}$ is a $\delta$-covering of $\cT$, and so
    $$
    \left\{ \cC_\delta(\cT) - \cC_{(1-\eps)\delta}(\cS^{\delta,\lambda}) > \frac{\eps}{\delta^2} \right\} \subseteq \left\{ \fs_{\eps\delta/2}(\theta,\cT) > \frac{\eps}{2\delta^2} \ \forall \theta \in \left(0,\frac{\eps\delta}{2}\right) \right\},
    $$
    which gives the desired result.
\end{proof}

This previous result directly implies \cref{prop:fringe_CRT}.

\begin{proof}[Proof of \cref{prop:fringe_CRT}]
    The proof follows from \cref{prop:cover_fringe_CRT} evaluated at $\delta/2$ and the relation between packing and covering, that is, $\cC_{\delta/2}(\cT)=\cP_{\delta}(\cT)$ and $\cC_{(1-\eps)\delta/2}(\skelr)=\cP_{(1-\eps)\delta}(\skelr)$.
\end{proof}

\subsection{Local convergence of the skeleton}

\cref{prop:crt_local_limit} is mostly a consequence of a result of Aldous which characterizes the distribution of the span of uniformly sampled points in a CRT. We introduce the notation $\T_{k}^{bin}$ for the set of planted leaf-labelled binary trees with exactly $k$ leaves Every $T \in \T_{k}^{bin}$ has exactly $2k-1$ edges. Moreover, for a rooted compact $\R$-tree $t$ with only finitely many branch points, we will abuse notation and also let $t$ denote the corresponding combinatorial tree with edge lengths. We will call the underlying combinatorial tree the \emph{shape} of $t$.

\begin{prop}[cf Aldous \cite{aldous1993continuum}]\label{prop:crt_finite_marginals}
    Let $k = \lambda/\delta^2$ and let $(\cT,\rho,\mu)$ be a Brownian CRT equipped with the uniform measure on the vertices. Let $(\ell_1,...,\ell_{2k-1})$ be the edges in $\cS^{\lambda,\delta}$ in some arbitrary ordering. Then $\cS^{\lambda,\delta}$ is absolutely continuous and has the following density function:
    $$
    f\left(t; \ell_1 , ... , \ell_{2k-1} \right) = se^{-s^2/2}, \ \text{where } s = \sum_{i=1}^{2k-1} \ell_i.
    $$
    In particular, the shape of $\cS^{\lambda,\delta}$ is uniformly distributed over $\T_{k}^{bin}$, and the edge lengths are exchangeable.
\end{prop}

\begin{prop}
    We have 
    \[\frac{\cskel}{\delta/\sqrt{2\lambda}}\convp T^* \]
    in the marked local topology on trees with edge lengths, where $T^*$ is the binary Kesten tree with independent $\Exp(1)$ edge lengths.
\end{prop}

\begin{proof}

    The shape of $\cS^{\lambda,\delta}$ is that of a uniform planted leaf-labelled binary tree, which is structurally equivalent to a size-conditioned binary Bienaym\'{e} tree which is given a new root directly below the original one. Since the offspring distribution is binary and the edge lengths are exchangeable, applying \cref{lem:loc_convergence_shape} and \cref{lem:steps_weak_convergence} we can reduce local limit problem to that of showing that, for any fixed $m \geq 0$, $\frac{\sqrt{2\lambda}}{\delta}(\ell_1,...,\ell_m) \convd (E_1,...,E_m)$ as $m\to \infty$, where $E_1,...,E_m$ are independent $\Exp(1)$ random variables to complete the proof.

    It is a standard fact following from the stick breaking construction of the Brownian CRT that $L_{2k-1} :=\sum_{i=1}^{2k-1} \ell_i \eqdist \sqrt{2G_k}$, where $G_k$ is a $\text{Gamma}(k)$ (see e.g., Remark 2.2.19 in \cite{haastert2025scaling}). Then, conditionally given $L_{2k-1}$, the vector $\frac{1}{L_{2k-1}}(\ell_1,...,\ell_{2k-1})$ is uniform over the unit simplex. Using the standard uniform spacing representation of this distribution (see e.g., \cite{devroye1986nonuniform} Section 2.1) along with the convergence in probability of $L_{2k-1}/\sqrt{2k}$ gives the desired result.

\end{proof}

\appendix

\section{Packings and coverings of $\R$-trees}\label{app:covering=packing}

In this appendix we prove \cref{prop:R_tree_packink_is_covering}. The proof is an adaptation of the proof from \cite{meir1975relations}, which gives the result for combinatorial trees without edge lengths. 
Recall that, because we consider $\R$-trees, packing and covering centres can lie on edges and that, when one writes for $x\in T$ we refer not only to branch points of $T$ but all points along the edges. 
\begin{prop}[cf Moon Meir \cite{meir1975relations}]\label{prop:packink_is_covering_closed}
    For any compact $\R$-tree $T$, $\mathcal{C}_{r}(T)=\mathcal{P}_{2r} (T)$.
\end{prop}
One direction is easy, which we immediately prove. Fix some compact $\R$-tree $T$, and let $P$ and $C$ be a maximum $2r$-packing and a minimum $r$-covering respectively.  Observe that for any $x\in C$, letting $B_r(x)=\{y\in T, \ d(x,y)< r\}$ be the $r$-open-neighbourhood of $x$ in $T$, 

$$\#\{B_r(x)\cap P\}\leq 1.$$
Because we also have $\cup_{x\in C}B_r(x)=T$, 

\begin{equation}\label{eq:pack_less_cover}
     \cC_r(T)\ge \cP_{2r}(T).
\end{equation}




The proof for the other direction relies on a recursive trimming process that iteratively reduces a compact $\R$-tree to an $\R$-tree with diameter at most $2r$, for which both the $2r$-packing number and $r$-covering number are trivially equal to $1$. Together with the property that each step of the trimming reduces the $2r$-packing number by at least $1$ and reduces the $r$-covering number by at most $1$, this proves the statement.  

We now define the trimming. Fix a compact $\R$-tree $T$ and a radius $r>0$. Suppose that $T$ has diameter at least $2r$. The recursive process is started by having $T^{(0)}=T$ and picking any longest path in $T^{(0)}$.  Call its endpoints $a_0$ and $b_0$. For any $0\leq z\leq d(a_0,b_0)$ let $x_z^{(0)}$ be the only point on the path from $a_0$ to $b_0$ at distance $z$ from $a_0$. $T^{(1)}$ is the smallest closed subtree containing all vertices $v$ such that $d(x_{r-\eps}^{(0)} ,v)\geq r$ for all $\eps>0$ so that, in particular, $b_0\in T^{(1)}$.
See \cref{fig:trimming_algo} for an illustration of the trimming process. Remark that $T^{(1)}$ is a compact $\R$-tree itself. By repeating this construction as long as $T^{(i)}$ has diameter greater than $2r$, we define $a_i$, $b_i$, $x_z^{(i)}$ and $T^{(i)}$. 

\begin{figure}[ht]
    \centering
    \begin{tikzpicture}[
        scale=0.95,
        vertex/.style={
            circle,
            draw=black,
            fill=black,
            inner sep=0pt,
            minimum size=5.5pt
        },
        greenvertex/.style={
            vertex,
            fill=green!75!black
        },
        redvertex/.style={
            vertex,
            fill=red
        },
        inside/.style={
            draw=black!65,
            line width=1.15pt
        },
        outside/.style={
            draw=green!65!black,
            line width=1.25pt
        }
    ]

    \coordinate (xa) at (-5,0);
    \coordinate (v1) at (-3,0);
    \coordinate (xr) at (-2,0);
    \coordinate (v2) at (-1.0,0);
    \coordinate (v3) at (0.2,0);
    \coordinate (v4) at (1.6,0);
    \coordinate (v5) at (2.9,0);
    \coordinate (xb) at (4.5,0);

    \draw[inside]
        (xa) -- (v1) -- (xr) -- (v2) -- (v3);

    \draw[outside]
        (v3) -- (v4) -- (v5) -- (xb);


    \coordinate (l1) at (-3,0.8);
    \coordinate (l2) at (-2.75,1.6);
    \coordinate (l3) at (-3.25,1.6);

    \draw[inside] (v1) -- (l1);
    \draw[inside] (l1) -- (l2);
    \draw[inside] (l1) -- (l3);


    \coordinate (u1) at (-1,0.8);
    \coordinate (u2) at (-1.25,1.6);
    \coordinate (u3) at (-0.75,1.6);

    \draw[inside] (v2) -- (u1);
    \draw[inside] (u1) -- (u2);
    \draw[inside] (u1) -- (u3);

    \coordinate (d0) at (0.2,0.8);
    \draw[inside] (v3) -- (d0);


    \coordinate (g1) at (0.45,-1);
    \coordinate (g2) at (1,-1.7);
    \coordinate (g3) at (1.2,-2.);
    \coordinate (g4) at (1.7,-2);
    \coordinate (g5) at (2,-1.5);

    \draw[outside] (v3) -- (g1);
    \draw[outside] (g1) -- (g2);
    \draw[outside] (g2) -- (g3);
    \draw[outside] (g2) -- (g4);
    \draw[outside] (g2) -- (g5);

    \coordinate (g6) at (1.6,0.8);
    \draw[outside] (v4) -- (g6);

    \coordinate (g7) at (2.9,0.8);
    \draw[outside] (v5) -- (g7);

    \draw[
        black!65,
        dashed,
        line width=1.1pt
    ]
        (xr) circle[radius=3];

    \draw[
        black!65,
        densely dotted,
        line width=1.15pt,
        -{Latex[length=3mm,width=2mm]}
    ]
        (xr) -- ++(-0.65,-2.95)
        node[pos=.48,right=1pt] {$r$};

    \node[redvertex] at (xa) {};
    \node[redvertex] at (xr) {};
    \node[redvertex] at (xb) {};


    \foreach \p in {v1,v2}
        \node[vertex] at (\p) {};

    \foreach \p in {v3,v4,v5}
        \node[greenvertex] at (\p) {};

    \foreach \p in {l1,l2,l3,u1,u2,u3,d0}
        \node[vertex] at (\p) {};

    \foreach \p in {g1,g2,g3,g4,g5,g6,g7}
        \node[greenvertex] at (\p) {};

    \node[below left=4pt] at (xa)
        {$a_o$};

    \node[below right=4pt] at (xr)
        {$x_0^{(r)}$};

    \node[below=4pt] at (xb)
        {$b_0$};

    \node[
        green!55!black,
        below left=5pt
    ] at ($(v4)!0.55!(v5)$)
        {$T^{(1)}$};

    \end{tikzpicture}

    \caption{A sketch of the first step of the trimming with $T^{(1)}$ in green.}
    \label{fig:trimming_algo}
\end{figure}
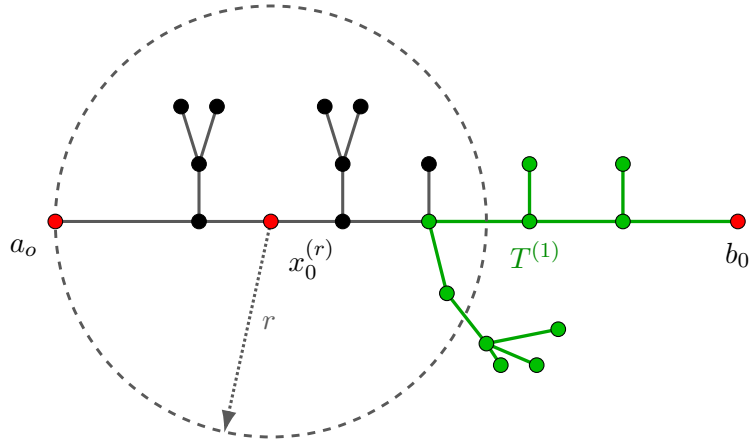


\begin{lem}\label{lem:trimming_stops}
    For any compact $\R$-tree $T$, the trimming process stops in a finite number of steps.
\end{lem}

\begin{proof}
    For some compact $\R$-tree $T$, suppose that the trimming process can be run at least $k$ times. We are going to prove that for any $i<j\leq k$

    \begin{equation}
        d(a_i,a_j)\geq r,
    \end{equation}
    constructing a $r$-packing of $T$ with $k$ elements. Because $T$ is compact, its maximum $r$-packing is finite, hence $k$ is upper bounded. The core of the argument is proving that $d(a_i,T^{(i+1)})\geq r$.

    We start by proving that for all $\eps_0>0$, $x_{r-\eps_0}^{(i)}\notin T^{(i+1)}$. If $x_{r-\eps_0}^{(i)}\in T^{(i+1)}$ then, since $b_i\in  T^{(i+1)}$, there is a path in $T^{(i+1)}$ from a leaf $y$ to $b_i$ that passes through $x_{r-\eps}^{(i)}$. Since $y$ is a leaf in $T^{(i+1)}$, $d(y,x_{r-\eps}^{(i)})\geq r$ for all $\eps>0$. Then $d(y,b_i)=d(y,x_{r-\eps_0}^{(i)})+d(x_{r-\eps_0}^{(i)},b_i)\geq 
    r+d(x_{r-\eps_0}^{(i)},b_i)=d(a_i,b_i)+\eps_0$, which contradicts the fact that $a_i$ and $b_i$ are the endpoints of a longest path in $T^{(i)}$.  
    
    As a consequence, for all $y\in T^{(i)}$, if $d(a_i,y)< r$ then $y\notin T^{(i+1)}$. Indeed, suppose by contradiction that there exists $y\in T^{(i+1)}$, with $d(a_i,y)< r$. Then the whole path from $y$ to $b_i$ is in $T^{(i+1)}$, but this path then contains $x_{r-\eps}^{(i)}$ for some $\eps$ small enough, which is a contradiction. Hence $d(a_i,T^{(i+1)})\geq r$ and by recursion, for any $j>i$
 \begin{equation}
        d(a_i,x_a^{(j)})\geq r.
    \end{equation}
\end{proof}

We now proceed to prove \cref{prop:R_tree_packink_is_covering} by recursion.

\begin{proof}[Proof of \cref{prop:R_tree_packink_is_covering}]

As observed before,  if $T$ is a compact $\R$-tree of diameter smaller than $2r$, then $P_{2r}(T)=C_{r}(T)=1$. Thus, we fix some compact $\R$-tree $T$ of diameter at least $2r$, and we are done if we show that 
\begin{align}
    \label{eq:trim_cover}
    \cC_r(T)&\leq \cC_{r}(T^{(1)})+1\\
    \label{eq:trim_packing}
    \cP_{2r}(T)&\geq \cP_{2r}(T^{(1)})+1
\end{align}

Let $P^{(1)}$ and $C^{(1)}$ be a maximum $2r$-packing and a minimum $r$-covering of  $T^{(1)}$ respectively. In the proof of \cref{lem:trimming_stops} we established that for all $\eps>0$, $x_{r-\eps}^{(0)}\notin T^{(1)}$. Hence, $\#\{C^{(1)}\cup \{x_{r-\eps}^{(0)}\}\}=\cC_{r}(T^{(1)})+1$.

Thus, to prove \eqref{eq:trim_cover}, it is sufficient to show that we can choose $\eps$ small enough such that $C^{(1)}\cup \{x_{r-\eps}^{(0)}\}$ is an $r$-covering for $T$. Firstly, since $T^{(1)}$ is closed and $C^{(1)}$ is an open covering, we have that there is some $0<\delta<r/4$ such that $C^{(1)}$ is also a $(r-2\delta)$-covering of $T^{(1)}$. So it is sufficient to show that we can choose $\eps>0$ small enough that for all $v$ with $d(v,T^{(1)})\geq \delta$, $d(v,x_{r-\eps}^{(0)})< r$.\\
By contradiction, suppose that for all $n$ there exists $v_n$ such that $d(v_n,T^{(1)})\geq \delta$ and $d(v_n,x_{r-1/n}^{(0)})\geq r$. Because $T$ is compact, $v_n$ has a converging subsequence $v_{\phi(n)}$, its limit is in $T^{(1)}$, which is a contradiction because we supposed $d(v_n,T^{(1)})\geq \delta$.





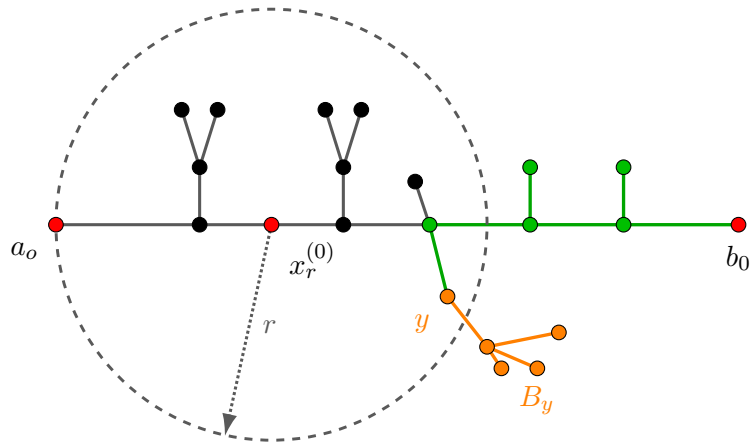
\begin{figure}[ht]
    \centering
    \begin{tikzpicture}[
        scale=0.95,
        vertex/.style={
            circle,
            draw=black,
            fill=black,
            inner sep=0pt,
            minimum size=5.5pt
        },
        greenvertex/.style={
            vertex,
            fill=green!75!black
        },
        redvertex/.style={
            vertex,
            fill=red
        },
        orangevertex/.style={
            vertex,
            fill=orange
        },
        inside/.style={
            draw=black!65,
            line width=1.15pt
        },
        outside/.style={
            draw=green!65!black,
            line width=1.25pt
        },
        BY/.style={
            draw=orange,
            line width=1.25pt
        }
    ]

    \coordinate (xa) at (-5,0);
    \coordinate (v1) at (-3,0);
    \coordinate (xr) at (-2,0);
    \coordinate (v2) at (-1.0,0);
    \coordinate (v3) at (0.2,0);
    \coordinate (v4) at (1.6,0);
    \coordinate (v5) at (2.9,0);
    \coordinate (xb) at (4.5,0);

    \draw[inside]
        (xa) -- (v1) -- (xr) -- (v2) -- (v3);

    \draw[outside]
        (g1) -- (v3) -- (v4) -- (v5) -- (xb);


    \coordinate (l1) at (-3,0.8);
    \coordinate (l2) at (-2.75,1.6);
    \coordinate (l3) at (-3.25,1.6);

    \draw[inside] (v1) -- (l1);
    \draw[inside] (l1) -- (l2);
    \draw[inside] (l1) -- (l3);


    \coordinate (u1) at (-1,0.8);
    \coordinate (u2) at (-1.25,1.6);
    \coordinate (u3) at (-0.75,1.6);

    \draw[inside] (v2) -- (u1);
    \draw[inside] (u1) -- (u2);
    \draw[inside] (u1) -- (u3);

    \coordinate (d0) at (0,0.6);
    \draw[inside] (v3) -- (d0);


    \coordinate (g1) at (0.45,-1);
    \coordinate (g2) at (1,-1.7);
    \coordinate (g3) at (1.2,-2.);
    \coordinate (g4) at (1.7,-2);
    \coordinate (g5) at (2,-1.5);

    \draw[BY] (g1) -- (g2);
    \draw[BY] (g2) -- (g3);
    \draw[BY] (g2) -- (g4);
    \draw[BY] (g2) -- (g5);

    \coordinate (g6) at (1.6,0.8);
    \draw[outside] (v4) -- (g6);

    \coordinate (g7) at (2.9,0.8);
    \draw[outside] (v5) -- (g7);

    \draw[
        black!65,
        dashed,
        line width=1.1pt
    ]
        (xr) circle[radius=3];

    \draw[
        black!65,
        densely dotted,
        line width=1.15pt,
        -{Latex[length=3mm,width=2mm]}
    ]
        (xr) -- ++(-0.65,-2.95)
        node[pos=.48,right=1pt] {$r$};

    \node[redvertex] at (xa) {};
    \node[redvertex] at (xr) {};
    \node[redvertex] at (xb) {};
    \node[greenvertex] at (v3) {};


    \foreach \p in {v1,v2}
        \node[vertex] at (\p) {};

    \foreach \p in {v4,v5}
        \node[greenvertex] at (\p) {};

    \foreach \p in {l1,l2,l3,u1,u2,u3,d0}
        \node[vertex] at (\p) {};

    \foreach \p in {g6,g7}
        \node[greenvertex] at (\p) {};

    \foreach \p in {g1,g2,g3,g4,g5}
        \node[orangevertex] at (\p) {};

    \node[below left=4pt] at (xa)
        {$a_o$};

    \node[below right=4pt] at (xr)
        {$x_r^{(0)}$};

    \node[below=4pt] at (xb)
        {$b_0$};

    \node[orange, below left=4pt] at (g1)
        {$y$};

    \node[orange, below right=4pt] at (g3)
        {$B_y$};


    \end{tikzpicture}

    \caption{A sketch of the situation where a maximum packing of $T^{(1)}$ has a packing centre $y$ at distance less than $r$ from $x_r^{(0)}$. In orange the tree $B_y$, while $T^{(1)}$ is the union of the orange and green trees.}
    \label{fig:trimming_for_packing}
\end{figure}


\noindent  
To prove \eqref{eq:trim_packing}, we establish that there exists a maximum $2r$-packing $P^{(1)}$ of $T^{(1)}$ with all vertices at distance at least $2r$ from $a_0$, so that $P^{(1)}\cup \{a_0\}$ is a $2r$-packing for $T$.  Let $P^{(1)}$ be a maximum $2r$-packing of $T^{(1)}$. It suffices to show that all vertices of $P^{(1)}$ are at distance at least $r$ from $x^{(0)}_r$. Suppose that there exists $y\in P^{(1)}$ such that $d(y,x_r^{(0)})< r$. Then, there exists $\eps_0>0$ such that $d(y,x_{r-\eps_0}^{(0)})< r$. We let $B_y$ be the subtree spanned by all vertices $x \in T^{(1)}$ such that $y$ lies on the path from $x$ to $b_0$ (see \cref{fig:trimming_for_packing} for an illustration of this construction). Suppose that there exists $w\neq y$ with $w\in B_y\cap P^{(1)}$. Because $w$ and $y$ are in the packing, $d(w,y)\geq 2r$. But then,
$$d(w,b_0)=d(w,y)+d(y,b_0)\geq 2r+d(y,b_0).$$
By the triangle inequality 

$$d(y,b_0)\geq d(x_{r-\eps_0}^{(0)},b_0)-d(y,x_{r-\eps_0}^{(0)})\geq d(a_0,b_0)-2r+\eps_0,$$
leading to

$$d(w,b_0)>d(a_0,b_0),$$
which contradicts the definitions of $a_0$ and $b_0$. Thus, we established that $B_y\cap P^{(1)}=\{y\}$. Because simultaneously $d(y,x_{r-\eps_0}^{(0)})< r$ and $y\in T^{(1)}$ we know that there exists $v\in T$ with $d(v,x_r^{(0)})\geq r$ and $y$ lies on the path from $v$ to $b_0$. Then $v$ is in $B_y$ and is such that, replacing $y$ by $v$ in $P^{(1)}$ we still have a maximum $2r$-packing of $T^{(1)}$. Hence, without loss of generality, we can suppose that the maximal packing $P^{(1)}$ is such that for any $x\in P^{(1)}$, $d(x,x_0^{(r)})\geq r$. Then, $d(x,a_0)\geq 2r$ and so $P^{(1)}\cup \{ a_0 \}$ is a $2r$-packing of $T$. Because $x_0^{(0)}\notin T^{(1)}$ we have $\#\{P^{(1)}\cup \{ x_0^{(0)} \}\}=\cP_{2r}^{(1)}+1$, and so

\begin{equation}
    \cP_{2r} \geq \cP_{2r}^{(1)}+1,
\end{equation}
which concludes the proof of the proposition.

\end{proof}


\section{Height of continuous binary Bienaymé tree}\label{app:bien}

\begin{lem}\label{lem:height_continuous_Bienayme_tree}
For $B$ a planted Bienaymé tree with offspring distribution $\tfrac{1}{2}\delta_0+\tfrac{1}{2}\delta_2$ and $\operatorname{Exp}(1)$ edge length, it holds that 
\[
\P{\height(B)>t}=\frac{2}{2+t}.
\]
\end{lem} 

\begin{proof}
     The height of $B$ is the extinction time of a continuous branching process $Z$, defined as follows. $Z$ is started from a single particle, so $Z_0=1$. Each particle lives for an Exp$(1)$ time and is then replaced by $\xi$ particles, where $\xi$ follows the distribution $\tfrac{1}{2}\delta_0+\tfrac{1}{2}\delta_2$ and is independent of everything else. Hence,

$$\P{\height(B)>t}=\P{Z_t>0}.$$
We start by introducing $f(s)=\E{s^{\xi}}=\tfrac{1}{2}(1+s^2)$, the offspring probability generating function. We also introduce the probability generating function for $Z$

$$G(t,s)=\E{s^{Z_t}}.$$
First, let $\tau$ be the first branching of the process $Z$. Then $\tau\sim \text{Exp}(1)$ and conditioning on $\tau$

\begin{equation*}
    Z_t=
    \begin{cases}
    1, & \text{if } \tau>t,\\
    \sum_{i=1}^{\xi} Z_{t-\tau}^{(i)}, & \text{if } \tau\le t,
    \end{cases}
\end{equation*}
where the $Z^{(i)}$ are iid copies of $Z$. Plugging this relation into the definition of $G$ yields

\begin{equation}
\begin{aligned}
    G(t,s)
    =&
    \E{s^{Z_t}\I{\tau>t}+s^{Z_t}\I{\tau\le t}}
    = se^{-t}+\int_{0}^t dr e^{-r} \E{\E{s^{Z_{t-r}}}^{\xi}}\\
    =&
    se^{-t}+\int_{0}^t dr e^{-r} f\left( G(t-r,s)\right).
\end{aligned}
\end{equation}
The change of variable $u=t-r$ followed by a derivation on $t$ leads to

$$G(t,s)=se^{-t}+\int_{0}^t du e^{-(t-u)} f\left( G(u,s)\right)$$
and 

\begin{equation}
\begin{aligned}
    \frac{\partial G(t,s)}{\partial t}
    =&
    -se^{-t}
    -\int_{0}^t du e^{-(t-u)} f\left( G(u,s)\right)
    +f\left( G(u,s)\right) \\
    =&
    f\left( G(u,s)\right) -G(t,s).
\end{aligned}
\end{equation}
We want to solve this differential equation for $s=0$. For clarity introduce $u(t)=G(t,0)$. After checking that $G(0,0)=0$ and recalling that $f(s)=\tfrac{1}{2}(1+s^2)$ we get that

\begin{equation*}
\begin{cases}
    u'=\frac{1}{2}+\frac{1}{2}u^2-u=\frac{1}{2}(1-u)^2,\\
    u(0)=0.
\end{cases}
\end{equation*}
This differential equation has as its only solution

$$t\mapsto \frac{t}{2+t}.$$
It remains to recall that $u(t)=G(t,0)=\P{Z_t=0}$ to conclude the proof of the lemma.
\end{proof}


\section{Consistency of the limiting constant}\label{app:consistency_limit}

Here we prove the following fact that implies the result in \cref{thm:unscaled_local_packing_number}.

\begin{fact}\label{fact:basic_prob}
    Let $(X_{n,k})_{n,k \geq 1}$ be a collection of random variables. Let $C > 0$, and let $C_k > 0$ for all $k \geq 1$. Suppose that, for any $(k_n)_{n \geq 1}$ with $k_n \to \infty$, we have that $X_{n,k_n} \convp C$. Moreover, suppose that $X_{n,k} \convp C_k$ as $n \to \infty$ for any fixed $k$. Then, $C_k \to C$ as $k \to \infty$.
\end{fact}

\begin{proof}
Suppose for contradiction that $C_k \not\to C$. Then, upon taking a subsequence if needed, there exists $\eps > 0$
such that $|C_{k} - C| \geq \eps$ for all $k \geq 1$. Since $X_{n,k} \convp C_k$ for each fixed $k$, we may choose an $n_k$ large enough that
$$
    \prob\left\{\left|X_{n_k,k} - C_k\right| \geq \frac{\eps}{4}\right\} \leq \frac{1}{4}.
$$
Now define a sequence $(k_n)_{n \geq 1}$ by setting $k_n = k$ for
$n \in [n_k, n_{k+1})$. We can suppose without loss of generality that $k_n \to \infty$, and so by assumption we have
$X_{n,k_n} \convp C$. In particular, for all sufficiently large $n$,
$$
    \prob\left\{\left|X_{n,k_n} - C\right| \geq \frac{\eps}{4}\right\} \leq \frac{1}{4}.
$$
By the triangle inequality, since $|C-C_k| \geq \eps$, we have that
$$
\prob\left\{\left|X_{n_k,k_{n_k}} - C_k\right| \geq \frac{\eps}{4}\right\} \bigcup \left\{\left|X_{n_k,k_{n_k}} - C\right| \geq \frac{\eps}{4}\right\} =1.
$$
Therefore, by taking $k$ large we can get,
$$
    1 \leq \prob\left\{\left|X_{n_{k},k_{n_k}} - C_{k_{n_k}}\right| \geq \frac{\eps}{4}\right\}
    +
    \prob\left\{\left|X_{n_{k},k_{n_k}} - C\right| \geq \frac{\eps}{4}\right\} \leq \frac{1}{4} + \frac{1}{4} = \frac{1}{2},
$$
which is a contradiction.
\end{proof}

\section{Postponed Proofs from \cref{sec:covering_fringe}}\label{sec:postponed}
\subsection{Stochastic domination}
We prove \cref{cor:stochastic_order_forest} by first showing that the partial order in \cref{dfn:partial_order} induces a stochastic ordering on the height of a random fringe subtree of uniformly random trees with a given degree statistic and fixed root degree:

\begin{thm}\label{thm:stochastic_order_tree}
Fix $r\geq 1$ and let $\preceq$ be the partial order on degree statistics as in \cref{dfn:partial_order}.
Take $\dstat=(\dstat_i)_{i\geq 0}$, $\dstat'=(\dstat_i')_{i\geq 0}$ to be degree statistics corresponding to one component and have the same size.  Suppose $\dstat_r,\ \dstat_r' \geq 1$, and let $\rtree_\dstat, \rtree_{\dstat'}$ be uniformly selected trees from $\dsTree(\dstat),\ \dsTree(\dstat')$ respectively, such that the roots of $\rtree_\dstat$ and $\rtree_{\dstat'}$ have $r$ children. Lastly, take $W, W'$ to be uniformly chosen vertices from $\rtree_\dstat$ and $\rtree_{\dstat'}$ . Then, 
    \[\dstat\preceq\dstat' \implies \height\left(\rtree_\dstat^{\uparrow W}\right) \preceq_{\mathrm{st}}\height\left(\rtree_{\dstat'}^{\uparrow W'}\right).\]
\end{thm}

\begin{figure}[h!]    
\centering

\begin{subfigure}{\linewidth}
\centering
\includegraphics[width=0.342\linewidth]{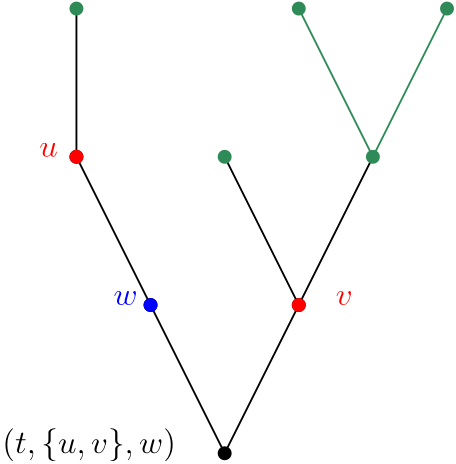}
\hfill
\includegraphics[width=0.315\linewidth]{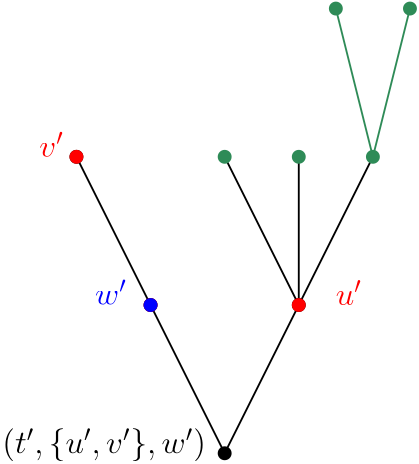}
\hfill
\includegraphics[width=0.315\linewidth]{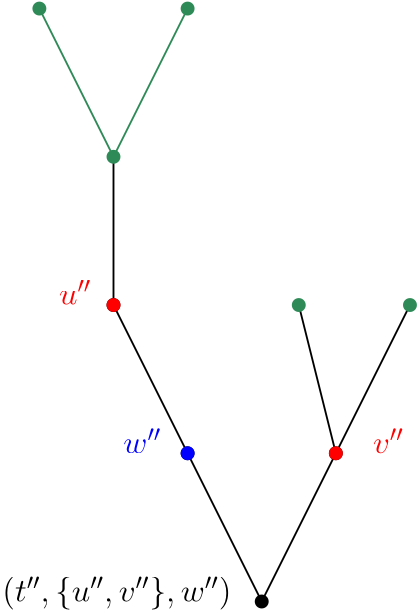}
\caption{Three marked trees in the same equivalence class. The root of each tree is the vertex at the bottom. In this example, $u$ and $v$ are not ancestrally related, and the marked vertex $w$ has exactly one of $u$ or $v$ as a descendant. The green subtrees are the fringe subtrees that are detached and reordered.}
\end{subfigure}
\vspace{0.5cm}
\begin{subfigure}{\linewidth}
\centering
\includegraphics[width=0.235\linewidth]{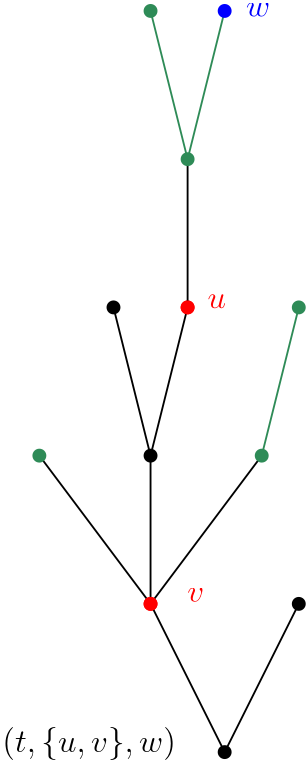}
\hfill
\includegraphics[width=0.292\linewidth]{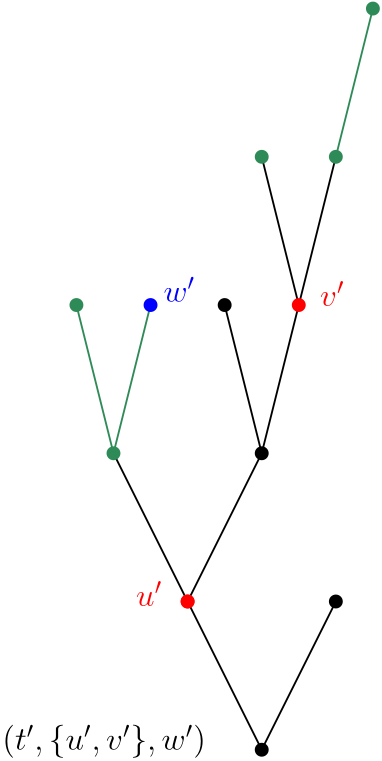}
\hfill
\includegraphics[width=0.327\linewidth]{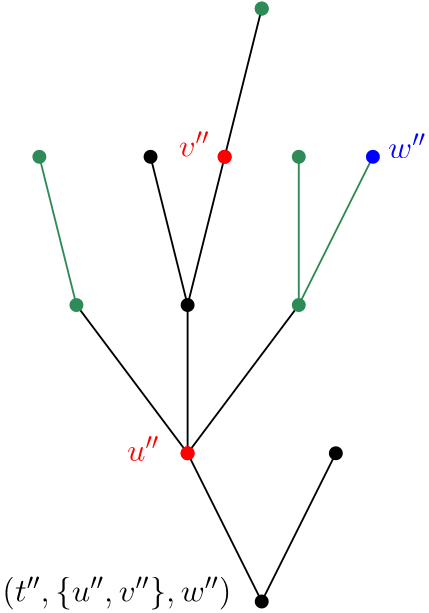}
\caption{Three marked trees in the same equivalence class. The root of each tree is the vertex at the bottom. In this example, $u$ and $v$ are ancestrally related. The marked vertex $w$ has neither $u$ nor $v$ as a descendant. The green subtrees are the fringe subtrees that are detached and reordered.}
\end{subfigure}
\caption{Examples of elements in the same equivalence class. These figures were adapted from \cite{donderwinkel2026tight} with permission of the authors.}
\label{fig:equiv_class}
\end{figure}

To prove this statement, it is enough to consider $\dstat'$ that covers $\dstat$. Then, the more general statement can be obtained by considering a chain of covering relations between $\dstat$ and $\dstat'$ yielding a corresponding chain of inequalities that proves that $\height\left(\rtree_\dstat^{\uparrow W}\right) \preceq_{\mathrm{st}}\height\left(\rtree_{\dstat'}^{\uparrow W'}\right)$. Thus, we may assume there are a $k\le \ell$ such that for each $i$
\begin{equation}\label{eq:cover}
\mathbf{n}_i'=\mathbf{n}_i+\I{i=k}-\I{i=k-1}+\I{i=\ell}-\I{i=\ell+1}.
\end{equation}
We will closely follow the proof of Theorem 5.5 in \cite{donderwinkel2026tight}, which in turn is based on the proof of Theorem 9 in \cite{height}. 

 Recall that in this paper, by tree we mean plane tree, and let $\leq$ be the ordering on the vertices. We will use an equivalence relation between trees with three marked vertices that, informally, can be described as follows. Let $t$ be a plane tree and $u,v,w$ vertices in $t$ with $u\ne v$ and let $t'$ be a plane tree and $u',v',w'$ with $u'\ne v'$ vertices in $t'$.
Imagine placing red marks on $u$ and $v$ in $t$ and on $u'$ and $v'$ in $t'$, and a blue mark on $w$ in $t$ and on $w'$ in $t'$. Thus, it is possible that a vertex has both a red and a blue mark. Then, $(t,\{u,v\},w)\sim (t',\{u',v'\},w')$ if we can obtain marked $t$ from marked $t'$ by detaching some fringe subtrees (that do not contain a red mark themselves) of the red marked vertices and reattaching each at a red marked vertex. For a formal definition see \cite[Appendix A]{donderwinkel2026tight} (therein, there is no additional vertex $w$, but the adaptation is straightforward). Also see \cref{fig:equiv_class}. 

Observe that $(t,\{u,v\},w)\sim (t',\{u',v'\},w')$ in particular implies that the sum of the out-degrees of $u$ and $v$ in $t$ and of $u'$ and $v'$ in $t'$ is equal. 

The key feature of this equivalence relation is that it allows us to define a coupling between $(\rtree_\dstat,W)$ and $(\rtree_{\dstat'},W')$ under \eqref{eq:cover} by letting both trees be a sample from the same equivalence class under $\sim$. This is possible because of the following proposition. 

Let $U$ be a uniformly random vertex of out-degree $k-1$ unequal to the root in $\rtree_\dstat$, let $V$ be a uniformly random vertex of out-degree $\ell+1$ unequal to the root in $\rtree_\dstat$, and let $W$ be a uniformly random vertex in $\rtree_\dstat$. Similarly, let $U'$ be a uniformly random vertex of out-degree $k$ unequal to the root in $\rtree_{\dstat'}$,  let $V'$ be a uniformly random vertex of out-degree $\ell$ unequal to the root and to $U'$ in $\rtree_{\dstat'}$, and let $W'$ be a uniformly random vertex in $\rtree_{\dstat'}$. 

\begin{prop}
    For any $\dstat$, $\dstat'$ as in \eqref{eq:cover} and any equivalence class $\cC$ under $\sim$,  \[\P{(\rtree_\dstat,\{U,V\},W)\in \cC}=\P{(\rtree_{\dstat'},\{U',V'\},W')\in \cC}.\]
\end{prop}
\begin{proof}
    First, we may sample $(\rtree_\dstat,\{U,V\},W)$ by first sampling $(\rtree_\dstat,\{U,V\})$ and then sampling $W\in\rtree_\dstat$, since every tree in $\dsTree(\dstat)$ has the same number of vertices. Similarly, $(\rtree_{\dstat'},\{U',V'\},W')$ is sampled by first sampling $(\rtree_{\dstat'},\{U',V'\})$ and then  sampling $W'\in\rtree_{\dstat'}$. Under assumption \eqref{eq:cover}, $\rtree_{\dstat'}$ has the same size as $\rtree_\dstat$, let this size be $n$.
Fix an equivalence class $\mathcal{C}$ for the relation with two red marks and one blue mark, and let $\hat{\mathcal{C}}$ be the class obtained by forgetting the blue-marked vertex, which is the equivalence class for the original relation in \cite[Appendix A]{donderwinkel2026tight}. For $(t,\{u,v\})\in\hat{\mathcal{C}}$, let 
    \[m(t,\{u,v\}) =\#\set{w\in t: (t,\{u,v\},w)\in\mathcal{C}}.\]
    Note that every element of $\cC$ produces the same marked multiset of trees upon deleting the out-edges of the red-marked vertices, and $m(t,\{u,v\})$ equals the multiplicity of the tree containing the blue mark in this multiset. Consequently,  $m(t,\{u,v\})$ is constant over $\hat{\cC}$, and we denote this fixed value by $m_{\hat{\cC}}$.

    Thus combining with the fact that $W$ is uniform conditional on the tree and the red marks, then
    \begin{align*}
    &\P{(\rtree_\dstat,\{U,V\},W)\in \cC} \\
    &=\P{(\rtree_\dstat,\{U,V\}) \in\hat{\cC}}\probC{(\rtree_\dstat,\{U,V\},W)\in \cC}{(\rtree_\dstat,\{U,V\}) \in\hat{\cC}}\\
    &=\frac{m_{\hat{\cC}}}{n}\P{(\rtree_\dstat,\{U,V\}) \in\hat{\cC}}=\frac{m_{\hat{\cC}}}{n}\P{(\rtree_{\dstat'},\{U',V'\})\in \hat{\cC}}\\
&=\P{(\rtree_{\dstat'},\{U',V'\},W')\in \cC},
    \end{align*}
    where the penultimate equality follows from \cite[Proposition A.1]{donderwinkel2026tight}.
\end{proof}

The previous proposition implies that we may sample  $(\rtree_\dstat,\{U,V\},W)$ and\\
$(\rtree_{\dstat'},\{U',V'\},W')$ by first sampling a random equivalence class and then sampling both conditional on lying in this class. 

 We now describe this conditional sampling. Let $\cC$ be an equivalence class with\\
 $\P{(\rtree_\dstat,\{U,V\},W)\in \cC}>0$. Let $(t,\{x,y\},w)$ be some element of $\cC$ (so that the sum of the out-degrees of $x$ and $y$ in $t$ must be $k+\ell$). Place a red mark on $x$ and $y$ and a blue mark on $w$ in $t$. 

If $x$ and $y$ are not ancestrally related, then we first sample a uniformly random ordering of the $k+\ell$ fringe subtrees of $x$ and $y$. Then sample a uniformly chosen random subset $A$ of $[k+\ell]$ of size $k-1$ or size $\ell+1$. If we now detach all fringe subtrees of $x$ and $y$, attach those with an index (in the random ordering) in $A$ at the first (in plane order) red marked vertex (respecting the random order we sampled) and attach the remaining at the other red marked vertex (respecting the random order we sampled), then, writing $\{U,V\}$ for the red vertices and $W$ for the blue vertex, and $T$ for the resulting tree, we obtain a sample of $(T,\{U,V\},W)$ conditional on $(\rtree_\dstat,\{U,V\},W)\in \cC$. (Indeed, it is easy to see that this sampling procedure assigns equal mass to all trees with degree statistic $\dstat$ in $\cC$.)

If $x$ and $y$ are ancestrally related, there are only $k+\ell-1$ fringe subtrees of $x$ and $y$ that do not contain a red marked vertex. Sample a uniformly random ordering of these fringe subtrees. Then sample a uniformly random subset $B$ of $[k+\ell-1]$ of size $k-1$ or size $\ell+1$. If we now detach all fringe subtrees without red marks of $x$ and $y$, attach those with an index (in the random ordering) in $B$ at the descendant red marked vertex (respecting the random order we sampled) and attach the remaining at the other red marked vertex (respecting the random order we sampled, and inserting the tree that we did not detach in a random position), then, writing $\{U,V\}$ for the red vertices and $W$ for the blue vertex, and $T$ for the resulting tree, then we obtain a sample of $(T,\{U,V\},W)$ conditional on $(\rtree_\dstat,\{U,V\},W)\in \cC$. 

Sampling $(T',\{U',V'\},W')$ conditional on $(\rtree_{\dstat'},\{U',V'\},W')\in \cC$ goes analogous, replacing the possible sizes of $A$ (resp. $B$) by $k$ and  $\ell$, and denoting these sets by $A'$ (resp. $B'$).

We can use this conditional sampling procedure to prove the following proposition. 

\begin{prop}
For any $z$, any $\dstat$, $\dstat'$ as in \eqref{eq:cover} and for any equivalence class $\cC$ under $\sim$, \[\p\left\{\height\left(\rtree_{\dstat}^{\uparrow W}\right)>z\mid (\rtree_\dstat,\{U,V\},W)\in \cC\right\}
\le
\p\left\{\height\left(\rtree_{\dstat'}^{\uparrow W'}\right)>z\mid (\rtree_{\dstat'},\{U',V'\},W')\in \cC\right\}\]
\end{prop}
\begin{proof}

    Since we assume that $(\rtree_\dstat,\{U,V\},W)\in \cC$, we can fix tree $t$ equipped with two distinct vertices $x,y\in t$ so that $x<y$, and $w\in t$ not necessarily distinct so that $(t, \{x,y\}, w)\in \cC$. We look at cases, depending on the ancestral relations of $x$ and $y$ to $w$.

    First assume that $w$ is an ancestor of both $x$ and $y$, and in particular $w$ does not lie in any of the fringe subtrees that get detached and rearranged. 
    
    Write $\hat{\cC}$ for the equivalence class obtained from $\cC$ by forgetting the blue mark. Since the rearrangement of the detached fringe subtrees does not affect the ancestral path from the root to $w$, then the distance from the root to the blue-marked vertex is preserved. Thus, conditioning on $(\rtree_\dstat,\{U,V\},W)\in \cC$, the subtree rooted at $W$ satisfies
    \[\height\left(\rtree_{\dstat}^{\uparrow W}\right) = \height\left(\rtree_{\dstat}\right)-\height(W).\]
    Analogously, conditioning on $(\rtree_{\dstat'},\{U',V'\},W')\in \cC$ gives the above identity for $\rtree_{\dstat'}$ and $W'$. Furthermore, in $\cC$,  $\height(W)=\height(W')=h$ for some $h$.
    
    Applying \cite[Proposition A.2]{donderwinkel2026tight} to the underlying two-red-marked equivalence class $\hat{\cC}$, we get 
    \begin{align*}
    \probC{\height\left(\rtree_{\dstat}^{\uparrow W}\right)>z}{(\rtree_\dstat,\{U,V\},W)\in \cC}&=\probC{\height\left(\rtree_{\dstat}\right)>z+h}{(\rtree_\dstat,\{U,V\})\in \hat{\cC}}\\
    &\leq\probC{\height\left(\rtree_{\dstat'}\right)>z+h}{(\rtree_{\dstat'},\{U',V'\})\in \hat{\cC}}\\
    &=\probC{\height\left(\rtree_{\dstat'}^{\uparrow W'}\right)>z}{(\rtree_{\dstat'},\{U',V'\},W')\in \cC}.
    \end{align*}

    Now consider the case where neither $x$ or $y$ are descendants of $w$. Then for the fixed equivalence class $\cC$, the blue marked vertex has a deterministic maximal distance from its descendants, so
    \begin{align*}
    &\probC{\height\left(\rtree_{\dstat}^{\uparrow W}\right)>z}{(\rtree_\dstat,\{U,V\},W)\in \cC}\\
    &=\probC{\height\left(\rtree_{\dstat'}^{\uparrow W'}\right)>z}{(\rtree_{\dstat'},\{U',V'\},W')\in \cC}\in\{0,1\}.
    \end{align*}

    For the case where exactly one of $x$ or $y$ are a descendant of $w$, we use the ``eggs-in-one-basket lemma" \cite[Lemma 29]{height}. We use the conditional sampling procedure of $(\rtree_\dstat,\{U,V\},W)$ described above and the notation introduced therein.

    Suppose $x<y$ but are not ancestrally related. Let $t_1,\ldots t_{k+\ell}$ be the $k+\ell$ fringe subtrees of $x$ and $y$, and let $A$ be a uniformly chosen random subset of $[k+\ell]$ of size $k-1$ or size $\ell+1$. Since $A$ and $[k+\ell]\setminus A$ have the same distribution, then without loss of generality, we may assume that $x$ is a descendent of $w$. Thus, the set of fringe subtrees attached at $x$ has the same law as $\set{t_i:i\in A}$. Let $h_i= 1+\height(t_i)$ for $i\in[k+\ell]$, then under the conditioning that $(\rtree_\dstat,\{U,V\},W)\in \cC$, we have
    \[\height\left(\rtree_{\dstat}^{\uparrow W}\right)\overset{d}{=} d_t(w,x) + \max\set{h_i: i\in A}.\]
    Similarly, for a uniformly chosen random subset $A'$ of $[k+\ell]$ of size $k$ or size $\ell$, and noting that $[k+\ell]\setminus A'$ has the same distribution, then we again assume that $x$ is a descendent of $w$. Thus, the set of fringe subtrees attached at $x$ in $\rtree_{\dstat'}$ conditionally on $(\rtree_{\dstat'},\{U',V'\},W')\in \cC$ has the same law as $\set{t_i:i\in A'}$, so under this conditioning we have
        \[\height\left(\rtree_{\dstat'}^{\uparrow W'}\right)\overset{d}{=} d_t(w,x) + \max\set{h_i: i\in A'}.\]
    Thus for any $z>0$, applying the first part of \cite[Lemma 29]{height} we have
    \begin{align*}
    \probC{\height\left(\rtree_{\dstat}^{\uparrow W}\right)>z}{(\rtree_\dstat,\{U,V\},W)\in \cC}
    &=\P{ d_t(w,x) + \max\set{h_i: i\in A}>z}\\
    &\leq\P{ d_t(w,x) + \max\set{h_i: i\in A'}>z}\\
    &=\probC{\height\left(\rtree_{\dstat'}^{\uparrow W'}\right)>z}{(\rtree_{\dstat'},\{U',V'\},W')\in \cC} .
    \end{align*}

    Lastly, suppose $x<y$ with $x$ an ancestor of $y$, so $y$ is a descendant of $w$ and $w$ is a descendent of $x$. We argue similarly as above, and let $t_1,\ldots t_{k+\ell-1}$ be the $k+\ell-1$ fringe subtrees of $x$ and $y$. Then for  a uniformly chosen random subset $B$ of $[k+\ell-1]$ of size $k-1$ or size $\ell+1$, the set of fringe subtrees attached at $y$ in $\rtree_{\dstat}$ conditional on $(\rtree_{\dstat},\{U,V\},W)\in \cC$ has the same law as $\set{t_i:i\in B}$. Again letting $h_i=1+\height(t_i)$ for $i\in[k+\ell-1]$, then under this conditioning we have
        \[\height\left(\rtree_{\dstat}^{\uparrow W}\right)\overset{d}{=} d_t(w,y) + \max\set{h_i: i\in B}.\]
     Similarly, for a uniformly chosen random subset $B'$ of $[k+\ell-1]$ of size $k$ or size $\ell$, the set of fringe subtrees attached at $y$ in $\rtree_{\dstat'}$ conditionally on $(\rtree_{\dstat'},\{U',V'\},W')\in \cC$ has the same law as $\set{t_i:i\in B'}$, so under this conditioning we have
        \[\height\left(\rtree_{\dstat'}^{\uparrow W'}\right)\overset{d}{=} d_t(w,y) + \max\set{h_i: i\in B'}.\]
    Thus for any $z>0$, and applying the second part of \cite[Lemma 29]{height} we have
    \begin{align*}
    \probC{\height\left(\rtree_{\dstat}^{\uparrow W}\right)>z}{(\rtree_\dstat,\{U,V\},W)\in \cC}
    &=\P{ d_t(w,y) + \max\set{h_i: i\in B}>z}\\
    &\leq\P{ d_t(w,y) + \max\set{h_i: i\in B'}>z}\\
    &=\probC{\height\left(\rtree_{\dstat'}^{\uparrow W'}\right)>z}{(\rtree_{\dstat'},\{U',V'\},W')\in \cC}.
    \end{align*}
\end{proof}

Finally, \cref{thm:stochastic_order_tree} follows from the previous two propositions by the law of total probability, partitioning the probability space into the different equivalence classes under $\sim$. See the proof of Theorem 5.5 in \cite{donderwinkel2026tight} for the details.

\subsection{Technical step of \cref{lem:height_binary_ds}}

The next lemma is a technical computation used in \cref{lem:height_binary_ds}. We present it before the proof of \cref{lem:first_exc_cbridge} because some estimates proved here are later used in the proof of \cref{lem:first_exc_cbridge}.

\begin{lem}\label{lem:meas_change_const}
Let $R$ be a positive constant, $\alpha\in [0,1)$, and $1\ll\gamma_n\leq\sqrt{n}$. Set $p_n=\frac{1}{2}\big(1-\frac{R}{(1-\alpha)\gamma_n}\big)$ and

$$
\rho_n(t,k)=
\binom{t}{\tfrac{t+k}{2}}
p_n^{(t+k)/2}
(1-p_n)^{(t-k)/2},$$
so that, for $Y$ as in \cref{lem:height_binary_ds},
$$\P{Y_{n(1-\alpha)}=-\frac{Rn}{\gamma_n}}=\rho_n\big(n(1-\alpha),-\tfrac{Rn}{\gamma_n}\big).$$
For all $n$ large enough that $p_n\in(1/4,3/4)$ and all $\ell\leq \frac{n(1-\alpha)}{2}$,

\[ \frac{\rho_n\big(n(1-\alpha)-\ell,-\tfrac{Rn}{\gamma_n}+1\big)}{\rho_n\big(n(1-\alpha),-\tfrac{Rn}{\gamma_n}\big)}\leq 4e^8.\]
\end{lem}

\begin{proof}
    Start by recalling Robbins' version of Stirling's formula \cite{robbins1955remark}: for $t$ a positive integer,
    \[\sqrt{2\pi t}\pran{\frac{t}{e}}^t\exp\left\{\frac{1}{12t+1}\right\} 
    \leq 
    t! 
    \leq 
    \sqrt{2\pi t}\pran{\frac{t}{e}}^t\exp\left\{\frac{1}{12t}\right\}.\]
    Applying this formula we get, for any $t\geq 2$ and $1\leq k\leq t-1$,

    \begin{equation}\label{eq:best_binom_bound}
    \begin{aligned}
        e^{-4}\sqrt{\frac{t}{k(t-k)}}\cdot\frac{t^t}{k^k(t-k)^{t-k}}
        \leq
        \binom{t}{k}
        \leq 
        e^4\sqrt{\frac{t}{k(t-k)}}\cdot\frac{t^t}{k^k(t-k)^{t-k}}.
    \end{aligned}
    \end{equation}
    Also since $\gamma_n\to\infty$, then for $n$ large enough we have $p_n\in(1/4,3/4)$. In the remainder of the proof, suppose that $n$ is that large. Noting that,
    $$
    n(1-\alpha)p_n
    =
    \frac{n(1-\alpha)-Rn/\gamma_n}{2},
    $$
    we may write
    $$\rho_n\big(n(1-\alpha),-\tfrac{Rn}{\gamma_n}\big)
    =
    \binom{n(1-\alpha)}{n(1-\alpha)p_n}
    p_n^{n(1-\alpha)p_n}
    (1-p_n)^{n(1-\alpha)(1-p_n)}$$
    and \eqref{eq:best_binom_bound} evaluated at $t=n(1-\alpha)$, $k=n(1-\alpha)p_n$ yields

    \begin{equation}\label{eq:best_LB_rho_n}
        \rho_n\big(n(1-\alpha),-\tfrac{Rn}{\gamma_n}\big)
        \geq
        e^{-4}\frac{1}{\sqrt{n(1-\alpha)p_n(1-p_n)}}
        \geq
        2e^{-4}\frac{1}{\sqrt{n(1-\alpha)}}.
    \end{equation}
    On the other hand,

    $$\frac{
    \binom{t}{k+1}p_n^{k+1}(1-p_n)^{t-k-1}
    }
    {
    \binom{t}{k}p_n^{k}(1-p_n)^{t-k}
    }
    =
    \frac{t-k}{k+1}\cdot\frac{p_n}{1-p_n},$$
    which is at least $1$ if and only if $k\leq tp_n +p_n-1$. For brevity, suppose that $\binom{t}{k}p_n^{k}(1-p_n)^{t-k}$ is maximized exactly at an integer $k=tp_n$, then we can apply \eqref{eq:best_binom_bound} to get that

    \begin{align*}
        \sup_{0\leq k\leq t} \binom{t}{k}p_n^{k}(1-p_n)^{t-k}
        = &
        \binom{t}{tp_n}p_n^{tp_n}(1-p_n)^{t(1-p_n)}\leq \frac{e^4}{\sqrt{tp_n(1-p_n)}}
        \leq
        4e^4 /\sqrt{t},
    \end{align*}
    where we used that $p_n\in(1/4,3/4)$. If the maximizer $k^*$ is not exactly equal to $tp_n$, an easy computation yields the same result using $p_n\in(1/4,3/4)$ and $|k^*-tp_n|\leq 1$. Using the above with $t=n(1-\alpha)-\ell\geq n(1-\alpha)/2$ yields

    \begin{equation}\label{eq:best_UB_rho_n}
        \rho_n\big(n(1-\alpha)-\ell,-\tfrac{Rn}{\gamma_n}+1\big)
        \leq
        \frac{4e^4}{\sqrt{n(1-\alpha)-\ell}}
        \leq
        \frac{8e^4}{\sqrt{n(1-\alpha)}}.
    \end{equation}
    Combining \eqref{eq:best_LB_rho_n} and \eqref{eq:best_UB_rho_n} concludes the proof of the lemma.
\end{proof}

\subsection{Proof for \cref{lem:first_exc_cbridge}}

In this section we give the proof of \cref{lem:first_exc_cbridge}, which we recall below.

\firstexcbridge*

\begin{proof}
    Since $\alpha$ is bounded away from $1$, up to changing $n$ to $n(1-\alpha)$ and $\gamma_n$ to $(1-\alpha)\gamma_{n/(1-\alpha)}=O(\sqrt{n})$ it suffices to prove that for any $\eps>0$, there is $R$ large enough so that

    $$\probC{\tau_{-1}(X)>\left\lceil\frac{n}{2}\right\rceil}{ X_{n}=-\frac{Rn}{\gamma_n}}\leq \eps n^{-1/2}$$
    for $n$ large enough. Set $k=\lceil \tfrac{n}{2}\rceil$ and $p_n=\frac{1}{2}\big(1-\frac{R}{\gamma_n} \big)$. Note that for $n$ large enough, $p_n\in(1/4,3/4)$. In the remainder of the proof, suppose that $n$ is that large. Let $Y$ be the simple random walk starting at $0$ and for any $j\in \N$, $\P{Y_{j+1}=Y_j+1}=p_n$ and $\P{Y_{j+1}=Y_j-1}=1-p_n$. Conditionally on their endpoint being $-Rn/\gamma_n$ at time $n$, $X$ and $Y$ have the same law. Indeed, they both have the exact same number of up steps and down steps, which are ordered at random. However, note this conditioning is an exponentially low probability event for $X$, but not for $Y$. By the Markov property,

    \begin{align*}
        \probC{\tau_{-1}(X) > k}{X_n=-\frac{Rn}{\gamma_n}}
        =&
        \sum_{j\geq 0} \P{\tau_{-1}(Y) > k, \ Y_k=j}
        \frac{\P{Y_{n-k}=-\frac{Rn}{\gamma_n}-j}}{\P{Y_n=-\frac{Rn}{\gamma_n}}}.
    \end{align*}
    We recall from \cref{lem:meas_change_const}, with $\alpha=0$, that $\P{Y_t=k}=\rho_n(t,k)$ and that we are working with $n$ large enough so that $p_n\in(1/4,3/4)$. Then \eqref{eq:best_LB_rho_n} gives

    $$\P{Y_{n}=-\frac{Rn}{\gamma_n}}
        \geq 
        \frac{2e^{-4}}{\sqrt{n}},$$
    while \eqref{eq:best_UB_rho_n} at $\ell=k$ gives

    $$\P{Y_{n-k}=-\frac{Rn}{\gamma_n}-j}
    \leq 
    \frac{4e^4}{\sqrt{n-k}}.$$
    Recalling that $k=\lceil n/2 \rceil$, we get

    \begin{equation}\label{eq:UB_cond_hit}
        \probC{\tau_{-1}(X) > k}{X_n=-\frac{Rn}{\gamma_n}}
        \leq 
        4e^8 \P{\tau_{-1}(Y)\geq k},
    \end{equation}
    which removes the conditioning. Next, we directly relate $\P{\tau_{-1}(Y)\geq k}$ to $\P{\tau_{-1}(X)\geq k}$ by first noting that
    $$
    \frac{\P{Y_k=j}}{\P{X_k=j}}
    =
    (2p_n)^{\frac{k+j}{2}}\big(2(1-p_n)\big)^{\frac{k-j}{2}}
    =
    \big(4p_n(1-p_n)\big)^{k/2}\left(\frac{p_n}{1-p_n}\right)^{j/2}.
    $$
    By construction we have $p_n\leq \frac{1}{2}$, and so $p_n/(1-p_n)\leq 1$. Summing over all paths such that $\tau_{-1}(\cdot)\geq k$ for which $j\geq0$ yields

    $$\P{\tau_{-1}(Y)\geq k}
    \leq 
    \big(4p_n(1-p_n)\big)^{k/2}
    \P{\tau_{-1}(X)\geq k}.
    $$
    Thus combining with \eqref{eq:UB_cond_hit},
    \begin{equation}\label{eq:hitting_time_goal_split}
        \probC{\tau_{-1}(X) > k}{X_n=-\frac{Rn}{\gamma_n}}
        \leq 
        4e^8 \big(4p_n(1-p_n)\big)^{k/2}
        \P{\tau_{-1}(X)\geq k}.
    \end{equation}
    Note that 

    \begin{equation}\label{eq:hitting_time_step1}
        \big(4p_n(1-p_n)\big)^{k/2}
        =
        \left(1-\frac{R^2}{\gamma_n^2}\right)^{k/2}
        \leq
        \exp\left(-\frac{R^2 k}{2\gamma_n^2}\right).
    \end{equation}
    To bound $\P{\tau_{-1}(X)\geq k}$, simply use Kemperman's formula \cite[(6.3)]{pitman2006combinatorial},

    \begin{align*}
        \P{\tau_{-1}(X)\geq k}
        =
        \sum_{j\geq k}\frac{1}{j}\P{X_j=-1}.
    \end{align*}
    Record that $\P{X_{2j}=-1}=0$ and use \eqref{eq:best_binom_bound} to get

    $$\P{X_{2j+1}=-1}
    \leq
    e^4 2^{-2j-1}\sqrt{\frac{2j+1}{j(j+1)}}\cdot\frac{(2j+1)^{2j+1}}{j^j(j+1)^{j+1}}
    \leq
    \frac{e^7}{\sqrt{j}}
    $$
    to get

    \begin{equation}\label{eq:best_stopping_time}
    \begin{aligned}
        \P{\tau_{-1}(X)\geq k}
        \leq&
        \sum_{j\geq k/2-1}\frac{1}{2j+1}\cdot\frac{e^7}{\sqrt{j}}\leq 
        \frac{e^8}{\sqrt{k}}
    \end{aligned}
    \end{equation}




    \noindent Plugging \eqref{eq:hitting_time_step1} and \eqref{eq:best_stopping_time} into \eqref{eq:hitting_time_goal_split} and recalling that $k=\lceil n/2\rceil $ leads to 

    $$
    \probC{\tau_{-1}(X) > k }{X_n=-\frac{Rn}{\gamma_n}}
    \leq
    \frac{4e^{16}}{\sqrt{n}}
    \exp\left(-\frac{R^2 n}{4\gamma_n^2}\right).
    $$
    Because $n/\gamma_n^2$ is lower bounded by a constant, then we obtain our desired bound for sufficiently large $R$.
\end{proof}

\section{Postponed Proofs from \cref{sec:local_skeleton}}\label{app:skeleton_local_conv}

We give the proofs of \cref{lem:limit_composition} and \cref{lem:steps_weak_convergence} that we recall below.

\compositionconv*

\begin{proof}
    The number of weak composition of $n$ into $k$ is given by $\binom{n+k-1}{k-1}$. This fact alone is enough for the proof of the lemma. We start with the case $\ell=1$. For any $N$, any $K<N/4$ and any $j<N/2$, one has 

    \begin{equation*}
    \begin{aligned}
        \P{E_1=j}
        =& \
        \binom{N-j+K-2}{K-2}\binom{N+K-1}{K-1}^{-1}\\
        =& \
        \frac{K-1}{N+K-j-1}\prod_{i=0}^{j-1}\left(1+\frac{K-1}{N-i}\right)^{-1},
    \end{aligned}
    \end{equation*}
    elementary bounds lead to

    \begin{equation*}
        \frac{K-1}{N+K}\exp\left( -j \tfrac{K}{N-j}-j\left(\tfrac{K}{N-j}\right)^2 \right)
        \leq
        \P{E_1=j}
        \leq
        \frac{K}{N-j}\exp\left( -j \tfrac{K-1}{N+K} \right)
    \end{equation*}
    and later to

    \begin{equation*}
        \frac{K}{N}e^{-j\tfrac{K}{N}}\left( 1-40\tfrac{N+K^2+Kj}{N^2} \right)
        \leq
        \P{E_1=j}
        \leq
        \frac{K}{N}e^{-j\tfrac{K}{N}}\left( 1+40\tfrac{N+K^2+Kj}{N^2} \right).
    \end{equation*}
    Letting $K$ and $j$ depend on $N$, if $K=o(N)$ and $j=o(N)$ then, for any $\lambda>0$

    $$ \lim_{N\to\infty} 
    \P{\frac{K}{N}E_1 \leq \lambda}
    =
    1-e^{-\lambda},
    $$
    concluding the proof of the lemma for $\ell=1$. As the $E_i$ are identically distributed, it only remains to check that they are independent in the limit. The same computation as above can easily be extended for fixed $\ell>1$, resulting in

    \begin{equation}
    \begin{aligned}
        \P{(E_1,\dots,E_{\ell})=(j_1,\dots,j_\ell)}
        =&
        \binom{N-\sum j_i+K-\ell-1}{K-\ell-1}\binom{N+K-1}{K-1}^{-1}\\
        =&
        \left(\frac{K}{N}\right)^{\ell} e^{-\sum j_i\tfrac{K}{N}}\big(1+o(1)\big)
    \end{aligned}
    \end{equation}
    and translating to 

    $$ \lim_{N\to\infty} 
    \P{ \bigcap_{i=1}^{\ell} \left\{ \frac{K}{N}E_i \leq \lambda_i \right\}}
    =
    \prod_{i=1}^{\ell}(1-e^{-\lambda_i}),
    $$
    concluding the proof of the lemma.
\end{proof}

\weakconvergencesteps*

\begin{proof}
In this proof appear both the combinatorial fringe, defined
using graph distance from the metric fringe, defined
using edge lengths. It is clear from the argument of $F_k(.)$ which is used. Conditionally on $T_n$, let $O_n$ and $V_n$ be independent
uniform vertices of $T_n$. For a finite combinatorial fringe $t$, set

$$
q_{n,k}(t)
=
\P{F_k(O_n,T_n)=t\mid T_n},
\qquad
q_k(t)
=
\P{F_k(\rho,K)=t}.
$$
By assumption, $q_{n,k}(t)\xrightarrow{P} q_k(t)$ and conditional
independence of $O_n$ and $V_n$ gives, for any finite fringes $t,t'$,

$$
\P{
F_k(O_n,T_n)=t,\,
F_k(V_n,T_n)=t'
}
=
\E{q_{n,k}(t)q_{n,k}(t')}
\xrightarrow{P} q_k(t)q_k(t').
$$
Since every fixed combinatorial fringe of $K$ is finite almost
surely, it follows that
\begin{equation}\label{eq:pair_combinatorial_fringes}
\bigl(
F_k(O_n,T_n),
F_k(V_n,T_n)
\bigr)
\xrightarrow{d}
\bigl(
F_k(\rho',K'),
F_k(\rho'',K'')
\bigr),
\end{equation}
where $K'$ and $K''$ are independent copies of $K$.

We next include the edge lengths. First, the two fringes on the
left-hand side of \eqref{eq:pair_combinatorial_fringes} are disjoint
with probability tending to one. Indeed, a non-empty intersection implies
that one sampled vertex belongs to the fringe of the other.
Consequently, for every $M\geq1$,

$$
\begin{aligned}
&\P{
F_k(O_n,T_n)
\cap F_k(V_n,T_n)\neq\varnothing
}\leq
2\P{\#F_k(O_n,T_n)>M}
+
2\frac{M}{n}.
\end{aligned}
$$
Letting first $n\to\infty$ and then $M\to\infty$ proves the claim because the $\#F_k(O_n,T_n)$ are tight.

On each fixed pair of disjoint finite fringe shapes, the edge-length assumption gives joint convergence of
all their edge lengths to independent $\operatorname{Exp}(1)$ variables. Restricting first to a finite
collection of shapes with arbitrarily large probability, we obtain
\begin{equation}\label{eq:pair_length_marked_fringes}
\begin{aligned}
\bigl(
(F_k(O_n,T_n),\ell_n),
(F_k(V_n,T_n),\ell_n)
\bigr)\xrightarrow{d}
\bigl(
(F_k(\rho',K'),\ell'),
(F_k(\rho'',K''),\ell'')
\bigr),
\end{aligned}
\end{equation}
where the two limiting weighted-rooted trees are independent.

It remains to pass from combinatorial fringes to metric fringes. We recall that $T^*$ is the binary-Kesten tree with edge length, and that $K$ is its underlying combinatorial tree.
Let $E_1,E_2,\ldots$ be the successive edge lengths along the
marked spine of $T^*$, and write

$$
S_k=\sum_{j=1}^k E_j.
$$
Since the $E_j$ are independent $\operatorname{Exp}(1)$ variables,
$S_k\to\infty$ almost surely. Hence, for every $r>0$,
\begin{equation}\label{eq:spine_nonexplosion}
\P{S_k\leq r}\longrightarrow0
\qquad\text{as }k\to\infty.
\end{equation}

On the event $\{S_k>r\}$, the metric fringe $F_r(\rho,T^*)$
is determined by the combinatorial fringe
$F_k(\rho,K)$. This operation is
continuous whenever no spine vertex lies at distance exactly $r$, which holds almost surely.

For the approximating trees,
\eqref{eq:pair_length_marked_fringes} implies convergence of the
first $k$ spine lengths. Thus, for either sampled root, the
probability that its metric $r$-fringe is not determined by the combinatorial $k$-fringe has limiting upper bound at most $2\P{S_k\leq r}$.

Applying the continuous mapping theorem to
\eqref{eq:pair_length_marked_fringes}, and then letting
$k\to\infty$ using \eqref{eq:spine_nonexplosion}, proves joint
convergence of the two metric $r$-fringes. The same argument
applies simultaneously to any finite collection of radii.
Since the limiting metric fringes are finite almost surely,
these convergences characterize convergence in the marked local
topology. Therefore,

\begin{equation}
\begin{aligned}
\Bigl(
(\cT_n,P_n(O_n),\ell_n,O_n),
(\cT_n,P_n(V_n),\ell_n,V_n)
\Bigr)\xrightarrow{d}
\Bigl(
(K',P',\ell',\rho'),
(K'',P'',\ell'',\rho'')
\Bigr),
\end{aligned}
\end{equation}
with independent limiting copies. This is convergence in probability in the marked local topology.
\end{proof}

\bibliographystyle{plainnat}
\bibliography{burn}
\end{document}